%% file: ex_article.tex
\documentclass[onefignum,onetabnum]{siamonline190516}

\input{ex_shared}

\ifpdf
\hypersetup{
  pdftitle={Efficient linear response of chaos},
  pdfauthor={N. Chandramoorthy and Q. Wang}
}
\fi

\begin{document}

\maketitle

\begin{abstract}
 The sensitivity of time averages in a chaotic system to an infinitesimal parameter perturbation grows exponentially with the averaging time. However, long-term averages or ensemble statistics often vary differentiably with system parameters. Ruelle's response theory gives a rigorous formula for these parametric derivatives of statistics or {\em linear response}. But, the direct evaluation of this formula is ill-conditioned and hence, linear response, and downstream applications of sensitivity analysis, such as optimization and uncertainty quantification, have been a computational challenge in chaotic dynamical systems. This paper presents the space-split sensitivity or the S3 algorithm to transform Ruelle's formula into a well-conditioned ergodic-averaging computation. We prove a decomposition of Ruelle's formula that is differentiable on the unstable manifold, which we assume to be one-dimensional. This decomposition of Ruelle's formula ensures that one of the resulting terms, the stable contribution, can be computed using a regularized tangent equation, similar to in a non-chaotic system. The remaining term, known as the unstable contribution, is regularized and converted into an efficiently computable ergodic average. In this process, we develop new algorithms, which may be useful beyond linear response, to compute the unstable derivatives of the regularized tangent vector field and the unstable direction. We prove that the S3 algorithm, which combines these computational ingredients that enter the stable and unstable contributions, converges like a Monte Carlo approximation of Ruelle's formula. The algorithm presented here is hence a first step toward full-fledged applications of sensitivity analysis in chaotic systems, wherever such applications have been limited due to lack of availability of long-term sensitivities. 
\end{abstract}


\section{Introduction}
\label{sec:intro}
We say that a parameterized family of dynamical systems obeys \emph{linear response} when the infinite-time averages or ergodic 
averages of its smooth observables vary differentiably with the parameter. It was shown by Ruelle \cite{ruelle}\cite{ruelle1} that uniformly hyperbolic maps, which are mathematical idealizations of chaotic attractors, follow linear response; Proposition 8.1 of \cite{gouezel} is a simplified proof using modern transfer operator techniques. Rigorous proofs of linear response have since been extended to uniformly hyperbolic flows \cite{ruelle_hyperbolicflows}, partially hyperbolic systems \cite{dolgopyat}, dissipative stochastic systems \cite{hairer}, stochastically perturbed uniformly hyperbolic systems \cite{lucarini-stochastic} and even to a larger class of stochastic systems with possibly non-hyperbolic unperturbed dynamics \cite{galatolo}, certain nonuniformly hyperbolic systems \cite{bomfim} and intermittent systems \cite{baladi-nonuniform}\cite{baladi-intermittent}\cite{bahsoun}. From the statistical physics point of view, linear response theory has been found to be robust in high-dimensional systems \cite{wormell1}\cite{wormell2}, and has been usefully applied to chaotic systems across disciplines including climate models \cite{lucarini_climate}\cite{courtney}\cite{lucarini-lin-res}, biological systems (see \cite{cessac} for a review of linear response in neuronal networks) and turbulent flows in engineering systems \cite{nisha-luca}\cite{patrick}\cite{angxiu-jfm}\cite{chai}. 

Linear response of a chaotic system quantifies the proportional change in its long-term statistical behavior in response to small parameter perturbations. Apart from providing phenomenological understanding, this measure of long-term sensitivity is immensely useful in practical chaotic systems, which are often high-fidelity numerical simulations, for computational applications such as optimization, uncertainty quantification and parameter selection \cite{shadowing-DA}\cite{patrick}\cite{patrick-rare}\cite{nisha-luca}. Particularly in climate science, theoretical as well as computational studies of violations of linear response and the presence of arbitrarily large linear responses \cite{lucarini-review}\cite{lucarini-tipping}\cite{chekroun}\cite{ruelle_review} are crucial to gain a better understanding of intermittencies and climate tipping points \cite{ashwin}, which are active areas of research. 

Ruelle \cite{ruelle}\cite{ruelle1} established a formula for linear response or the derivative of ergodic averages with respect to parameters. However, a direct evaluation of this formula typically shows a poor convergence rate and is computationally impractical, as previous works have shown \cite{nisha_ES}\cite{eyink}. This poor convergence is due to the exponential growth of infinitesimal perturbations -- the so-called butterfly effect -- which is the defining characteristic of chaotic systems. Due to the butterfly effect, the sensitivity of a state at a time $n$ into the future, to an infinitesimal perturbation to the current state in almost any direction, grows exponentially with $n.$ Now, infinitesimal parameter perturbations may be thought of infinitesimal perturbations to each state applied in different tangent directions. Thus, along any trajectory, the sensitivity to an infinitesimal parameter perturbation also grows exponentially. However, the average of the sensitivities across all trajectories is a bounded quantity at all times. This ensemble-averaged sensitivity is, in fact, exponentially decreasing with time in uniformly hyperbolic systems, and Ruelle's formula for linear response is a series summation of these ensemble-averaged sensitivities. But, ensemble averaging exponentially growing quantities is a computationally challenging task that shows poor convergence.

Traditionally, sensitivities in a dynamical system are estimated by using tangent or adjoint equation solutions, or through automatic differentiation. Since all these methods time-evolve infinitesimal perturbations about a reference trajectory, the sensitivities they compute grow exponentially along any trajectory. Hence, conventional methods for sensitivity analysis have long been recognized as unsuitable in chaotic systems for computing linear response. Some methods circumvent the problem of exponentially growing sensitivities to compute a bounded value for linear response, but they may exhibit a bias \cite{davide}\cite{angxiu}. One such method is least-squares shadowing \cite{qiqi} and its non-intrusive version \cite{angxiu}\cite{angxiu-nilsas}, in which the shadowing lemma (see e.g. Chapter 18 of \cite{katok}) is used to compute sensitivities along a shadowing orbit. However, since shadowing orbits may be nonphysical, i.e., ergodic averages along shadowing orbits may not converge to ensemble averages, the  sensitivities computed by least-squares shadowing are not guaranteed to converge to linear response \cite{nisha-qiqi-shadowing}, although the error may be small for problems with a small ratio of unstable dimension to overall dimension \cite{angxiu-shadowing-error}.

Moreover, the direct evaluation of Ruelle's formula may also be thought of as an adaptation of conventional tangent/adjoint-based sensitivity computations for chaotic systems. This direct evaluation, known as the ensemble sensitivity approach \cite{lea}\cite{eyink}, involves taking a sample average of sensitivities computed by conventional tangent/adjoint methods. However, as we noted earlier, the number of samples needed, to reduce the variance in the exponentially growing sensitivities and compute linear response accurately, makes this approach computationally infeasible. In blended response algorithms \cite{abramov}, the ensemble sensitivity approach for short-time sensitivities is {\em blended} with a fluctuation-dissipation theorem-based approximation of the long-term sensitivities. This approximation is however adhoc since the densities of the SRB measure on unstable manifolds may not follow the fluctuation-dissipation theorem-based approximation, even though linear response holds. 

The purpose of this paper is an efficient computation of linear response in chaotic dynamical systems. We remark that using transfer operator techniques, a rigorous computation of linear response has been developed before, but it has been restricted to low-dimensional expanding maps \cite{bahsoun-computation}. Our aim is to develop a numerical method to evaluate Ruelle's formula that is scalable to high-dimensional practical systems. For this reason, we seek a method to compute Ruelle's formula that is provably convergent and is a computable ergodic average, which does not involve discretization of the phase space. The latter property of trajectory-based computation ensures that the convergence is that of a Monte Carlo computation of Ruelle's formula, at a rate independent of the system dimension. 

In this paper, we develop the space-split senstivity or the S3 method, which is a  scalable, efficient and rigorous computation of linear response. We prove that S3 provably converges in uniformly hyperbolic systems, and the convergence rate is similar to a typical Monte Carlo integration. We focus on uniformly hyperbolic systems with one-dimensional unstable manifolds, but design S3 keeping in mind future extensions to systems that have an unstable manifold of arbitrary dimension. 

\section{Prior work on computation of Ruelle's formula}
\label{sec:priorWork}
Besides previous work referenced in the introduction (section \ref{sec:intro}), in this section we compare and contrast two recent methods on the computation of Ruelle's formula against S3, which is developed in the present paper. In the computable realization of Ruelle's formula that is developed in \cite{nisha-aiaa} (and expanded in \cite{nisha-s3}), we split the perturbation field into its components along unstable and stable subbundles. In uniformly hyperbolic systems, such a direct sum decomposition of a vector field into its components along unstable and stable subbundles exists, and the resulting stable and unstable vector fields are H\"older continuous on the phase space (Chapter 6, 19 of \cite{katok}).

In \cite{nisha-aiaa}\cite{nisha-s3}, a method is then developed to evaluate the two resulting terms of Ruelle's formula. In the evaluation of the response to the unstable component, a finite difference is suggested to compute a part of the unstable divergence term \cite{nisha-aiaa}\cite{nisha-s3}. Besides this cumbersome computation of the unstable divergence, the major limitation of the linear response computation in \cite{nisha-aiaa}\cite{nisha-s3} is a fundamental differentiability assumption  that generally does not hold in uniformly hyperbolic systems. In particular, in the computation of the unstable contribution, the unstable component of the perturbation field is differentiated in the unstable directions, and this derivative is computed partly by finite difference and partly by a recursive formula (section 5 of \cite{nisha-s3}). Underlying this differentiation of the unstable component is the implicit assumption that the stable component of the perturbation is also differentiable along the unstable directions (since the perturbation field, which is the sum of its stable and unstable components, is differentiable by assumption in all directions), and this differentiability does not hold in a generic uniformly hyperbolic system (even in the case of smooth dynamics \cite{hasselblatt_prevalence_1999, pugh}).  

By contrast, to derive S3, we prove a new decomposition of Ruelle's formula -- the S3 decomposition -- which is differentiable on the unstable manifold. The S3 decomposition is not a stable-unstable splitting of the parameter perturbation vector field, which is pursued in \cite{nisha-aiaa}\cite{nisha-s3}.
In a slight abuse of terminology, the two resulting components of the S3 decomposition are still called the stable and unstable contributions to the overall sensitivity, although they do not arise from a stable-unstable splitting of the perturbation field. We develop recursive algorithms to compute the ingredients of both these components. The novel recursive algorithms for the stable and unstable contributions that are developed in this work do not use the covariant Lyapunov vector basis \cite{ginelli-clv-review}, unlike in \cite{nisha-aiaa}\cite{nisha-s3}. We only require an orthonormal basis for the unstable subspace, which is computationally cheaper to obtain, and requires only forward iterations of the dynamics.

A recent work by Ni \cite{angxiu-s3} introduces the linear response algorithm, which appears to be another viable solution to linear response computation. It has been derived for uniformly hyperbolic systems with arbitrary dimensional unstable manifolds. This algorithm provides a ``fast'' computation of the unstable divergence via a recursive formula, analogous to the density gradient for one-dimensional unstable manifolds treated in the present paper. In this linear response algorithm, second-order tangent equations, which are the most expensive step, are solved to differentiate certain vector fields with respect to a modified shadowing direction, which is computed by non-intrusive shadowing \cite{angxiu}. The second-order tangent equations developed in the present paper are derivatives along the one-dimensional unstable manifold. 

The plan for the paper is as follows. In section \ref{sec:background}, we introduce Ruelle's linear response formula and provide the mathematical background for its decomposition and subsequent evaluation via the S3 algorithm. Section \ref{sec:main} is a concise statement of the main contributions of this paper: two theorems concerning the S3 decomposition and evaluation, and the S3 algorithm. In section \ref{sec:splitting}, we derive the S3 decomposition of Ruelle's formula. The computation of the stable contribution that results from the decomposition is discussed in \ref{sec:stableContributionComputation}. An alternative expression of the unstable contribution is derived in section \ref{sec:integrationByParts}, whose computation is tackled in section \ref{sec:unstableContribution}. The validation of S3 on perturbed Baker's maps and Solenoid maps are presented in sections \ref{sec:baker} and \ref{sec:solenoid} respectively. The proofs of the two main theorems are split between sections \ref{sec:proofOfTheorem1} and \ref{sec:proofOfTheorem2}. Section \ref{sec:proofOfTheorem1} proves the existence of the S3 decomposition and its differentiability in the unstable direction, while section \ref{sec:proofOfTheorem2} completes the proof of convergence of the S3 algorithm. In section \ref{sec:conclusion}, we summarize our contributions and present a roadmap for extending the present algorithm to systems with higher-dimensional unstable manifolds.

\section{Preliminaries and problem setup}
\label{sec:background}
Consider a parameterized family $\varphi_s:M\to M$ of $C^\regularityClassOfdynamics$ 
diffeomorphisms of a Riemannian manifold $M,$ which we consider to be specified as a subset of $\mathbb{R}^m.$ Let $s$ be a scalar parameter that can take a small range of values around a reference value, say $s_0$. Corresponding to infinitesimal parameter perturbations at $s_0$, we define a vector field $\chi = d_s|_{s=0} \varphi_s \circ \varphi^{-1}$. 

Suppose $\varphi_s$ exhibits linear response at the reference value $s_0$, which we associate with the unperturbed dynamics. This means that the parametric derivative of long-term averages of $\varphi_s$ exists at the reference value $s_0.$  
Thus, at any $s = s_0 + \delta s$ close to $s_0$, up to first order in $\delta s$, the long-term averages of $\varphi_{s}$ can be expressed using information associated  only to the dynamics $\varphi_{s_0}.$  Let $J \in C^\regularityClassOfObservables(M)$ be an observable of interest, e.g. a lift or a drag in a numerical simulation of a turbulent flow. We are interested in a quantitative determination of how the long-term average of $J$ responds to infinitesimal perturbations in $s,$ at $s_0,$ i.e., to perturbations in the direction $\chi.$ 
\subsection{Ergodic theory and linear response}
The infinite-time 
average of $J$ is defined as $\langle J\rangle(x) := \lim_{N\to\infty} (1/N)\sum_{n=0}^{N-1} (J\circ\varphi^n_s)(x,s),$ for $x \in M.$ 
A probability distribution over the states on $M,$ $\mu_s$ is stationary or $\varphi_s$-invariant, when $\mu_s(\varphi^{-1}_s A) = \mu_s(A),$ for any Borel subset $A \subseteq M.$ When $\mu_s$ is an ergodic, $\varphi_s$-invariant, physical probability distribution for $\varphi_s$, infinite-time averages, also known as ergodic averages, are equal to expectations with respect to $\mu_s,$ at Lebesgue almost every $x$ in the basin of attraction of $\varphi_s.$ The expectation or ensemble average of an observable $J$ with respect to $\mu_s$, which is the Lebesgue integral with respect to the distribution $\mu_s,$ is written as $\langle J, \mu_s\rangle.$ 

That is, in ergodic systems, $\langle J\rangle(x,s)$ is independent of $x$ -- it is only a function of $s$ -- and equal to $\langle J, \mu_s\rangle$ at almost every $x$ in a set of full Lebesgue measure. The quantity we wish to compute is {\em linear response} at $s_0$, which is the parametric derivative of the ergodic/ensemble average, $d_s|_{s_0} \langle J, \mu_s\rangle.$ When $J$ explicitly depends on $s$ and the dependence is smooth, $d_s \langle J, \mu_s\rangle = \langle J, \partial_s\mu_s\rangle + \langle \partial_sJ,\mu_s\rangle$. The second term, $\langle \partial_sJ,\mu_s\rangle$ is simply an ergodic/ensemble average of $\partial_s J,$ which can be computed, using its definition, along almost every orbit: 
$$\langle \partial_s|_{s_0} J,\mu_{s_0}\rangle = \lim_{N\to\infty} (1/N) \sum_{n=0}^{N-1} (\partial_s J)(\varphi^n_{s_0} x,s_0),$$ for Lebesgue almost every $x$ on $M.$ Since this does not pose a computational challenge, we neglect the explicit dependence of $J$ on $s$, and focus on the first term, $d_s|_{s_0} \langle J, \mu_s\rangle = \langle J,\partial_s|_{s_0}\mu_s\rangle.$ 
From here on, we denote linear response at $s_0$ simply as $\langle J, \partial_s \mu_s\rangle,$ for brevity. 
\subsection{Tangent dynamics}
\label{sec:tangentDynamics}
The tangent space at $x \in M$, denoted $T_x M$, is the space of all infinitesimal perturbations to $x,$ which can be identified with $\mathbb{R}^m.$
A vector field on $M$, which is a direction of infinitesimal perturbation to the state at each point on $M$, can be considered as a map from $M$ to $\mathbb{R}^m$.
If $v$ is a vector field and $x\in M$, then $v_x\in T_x M \equiv \mathbb{R}^m$ denotes the value
of the vector field at $x$. Fixing $s$ at its reference value, we denote $\varphi$ the map $\varphi_{s_0}$. Similarly, we refer to the distribution $\mu_{s_0}$ simply as $\mu.$

The matrix $d\varphi^n$ gives the pushforward of a vector field by $\varphi^n$.
That is, $w = d\varphi^n v$ if $w_{\varphi^n x} = d\varphi^n_x\: v_{x}$.  We write $d\varphi^1$ simply as $d\varphi$. In the context of computations, we often fix a particular reference trajectory, say $\left\{x_n\right\}$, where $x_0$ is sampled according to $\mu$. We denote values of a vector field $v$ along the reference trajectory using $v_n := v_{x_n}$ for short; similarly, we write $(d\varphi)_n$ to denote the value of $d\varphi$ at $x_n$. The homogeneous tangent equation tracks the pushforward by $\varphi^n$ along a fixed trajectory, 
\begin{align}
    \label{eqn:homogeneousTangent}
    u_{n+1} = (d\varphi)_n \: u_n, \:\: n\in \mathbb{Z}^+.
\end{align}
At every iteration of the homogeneous tangent equation, starting with $u_0 \neq 0 \in \mathbb{R}^m$, the vector field $u$ is updated to $d\varphi \; u$. When we add a source term to Eq. \ref{eqn:homogeneousTangent}, we refer to the resulting equation as the inhomogeneous tangent equation. For example, when the source term is the parameter perturbation field, $\chi := (\partial_s\varphi_s)(\varphi^{-1}_{s_0}\cdot, s_0),$ the inhomogeneous tangent equation is the conventional tangent equation that is standard in sensitivity analysis,
\begin{align}
    \label{eqn:inhomogeneousTangent}
    u_{n+1} = (d\varphi)_n u_n + \chi_{n+1}. 
\end{align}
In this tangent equation, at every iteration, the vector field $u$ is updated to $d\varphi\: u  + \chi$. In tangent sensitivity analysis, the parametric derivative of time-averages, $\partial_s (1/N) \sum_{n=0}^{N-1} (J\circ\varphi^n_s)(x_0,s_0),$ is usually computed using inhomogeneous tangent solutions (Eq. \ref{eqn:inhomogeneousTangent}) as 
\begin{align*}
    \dfrac{1}{N}\sum_{n=0}^{N-1} \partial_s (J\circ\varphi_s^n)(x_0,s_0) = \dfrac{1}{N} \sum_{n=0}^{N-1} d J(x_n)\cdot u_n.
\end{align*}
Everywhere, $d$ denotes the gradient operator in $\mathbb{R}^m$. For example, in the above equation, $dJ(x_n)= (dJ)_n \in \mathbb T_{x_n}^* M \equiv \mathbb{R}^m$ refers to the gradient of $J$ evaluated at the point $x_n$.  
\subsection{Chaotic systems}
\label{sec:chaos}
The Oseledets multiplicative ergodic theorem (OMET) says that the homogeneous tangent solutions (Eq. \ref{eqn:homogeneousTangent}) grow/decay asymptotically. Further, this asymptotic growth/decay is exponential at a finite number of rates, called the Lyapunov exponents, which, in our setting of ergodic systems, are independent of the starting point $x_0,$ at $\mu_s$ almost every $x_0.$ We denote them in descending order as $\lambda_1 > \lambda_2 > ... > \lambda_\lyapunovdim,$ with $\lyapunovdim \leq d$. A chaotic system, by definition, exhibits at least one positive Lyapunov exponent. 

The tangent space at $\mu$-almost every point $x \in M$  has a direct sum decomposition, $T_x M = \oplus_{i\leq p} E^i_x$, with the subspace $E^i_x$ of tangent vectors having the asymptotic growth/decay rate $\lambda_i.$ In other words, every tangent vector $v_x$ belonging to the
subspace $E^i_x$ is such that $\lim_{n\to\infty}(1/n)\log\|d\varphi^n_x\: v_x\| = \lambda_i.$   
At $\mu$-almost every $x_0,$ almost every choice of $u_0 \in \mathbb{R}^m$ will have a non-zero component on the tangent subspace corresponding to the largest positive Lyapunov exponent. Hence, in chaotic systems, homogeneous tangent solutions starting from every initial condition $u_0$ grow exponentially in time. That is, for almost every $u_0 = u(x_0),$ $\|u_n\| \sim {\cal O}(e^{\lambda_1 n}),$ with $\lambda_1 > 0,$ for large $n.$ The exponent $\lambda_1$, which determines the asymptotic growth factor of the tangent solutions, is the largest among the Lyapunov exponents.

\subsection{Uniform hyperbolicity}
\label{sec:uniformHyperbolicity}
In this paper, we consider our $C^\regularityClassOfdynamics$ diffeomorphism $\varphi:M \to M$ to be equipped with a compact, invariant, hyperbolic attractor $\Lambda,$ which contains a one-dimensional unstable manifold. That is, $\Lambda$ is a compact set and  
$\varphi(\Lambda) = \Lambda.$ We say $\Lambda$ is a hyperbolic set for $\varphi$ if
there exist constants $C > 0$ and $\lambda \in (0,1)$ such that, at every point $x \in M,$ the tangent 
space $T_x M$ has a direct sum decomposition $T_x M = E^u_x\oplus E^s_x,$
where 
\begin{itemize}
\item $E^u_x$ and $E^s_x$ are $\varphi-$invariant, or \emph{covariant}, 
subspaces. That is, $d\varphi_x (E^{u}_x) = E^{u}_{\varphi(x)}$ and $d\varphi_x (E^{s}_x) = E^{s}_{\varphi(x)}.$
\item $E^u_x$ is the 1-dimensional \emph{unstable} subspace consisting of all $v \in T_x M$ 
    such that 
    \begin{align}
    \label{eqn:unstable}
         \norm{d\varphi^n_x\: v} \leq C \: \lambda^n \: \norm{v},
    \end{align}
    for all $n \in \mathbb{Z}^{-},$ and,
\item $E^s_x$ is the $(m-1)$-dimensional \emph{stable} subspace consisting of all $v \in T_x M$ 
    such that 
    \begin{align}
    \label{eqn:stable}
         \norm{d\varphi^n_x\: v} \leq C \: \lambda^n \: \norm{v}
    \end{align}
    for all $n \in \mathbb{Z}^{+}.$
\end{itemize}
Throughout, we use the shorthand $d\varphi_x$ to write the differential of $\varphi$ at $x$, which is a linear map from $T_x M$,
the tangent space at $x$, to $T_{\varphi x } M$, the tangent space at $\varphi x$. Using the standard basis of $\mathbb{R}^m$, $d\varphi_x$ can be represented as an $m\times m$ matrix. 
\begin{notation}
For convenience, we write $f_x$ to denote a scalar, vector or tensor field $f$ evaluated at the point $x\in M.$
\end{notation}
\subsection{The SRB measure}\label{sec:srb}
An ergodic, $\varphi_s-$invariant, physical measure, also known as an SRB measure \cite{srb}, is guaranteed to exist in our setting of uniformly hyperbolic systems. Apart from $\varphi_s-$invariance, we use the ergodicity of the map $\varphi_s$ with respect to $\mu_s$, which implies that there is no subset of the attractor, other than itself, that is invariant under the dynamics and has full $\mu_s$ measure. We also employ the physicality of the SRB measure. This means that ergodic averages starting from Lebesgue-a.e. initial condition chosen in an open set containing the attractor converge to expected values with respect to $\mu_s.$ 

We also exploit exponential decay of correlations with respect to the SRB measure enjoyed by observables in uniformly hyperbolic systems \cite{chernov}\cite{liverani}\cite{young-hyperbolicity}. This means that, for two H\"older continuous observables $J$ and $f$, there is some $c > 0$ and $\delta \in (0,1)$ so that
$\left| \langle J\circ \varphi^n \: f\rangle - \langle J\rangle \langle f\rangle \right|\leq c\: \delta^n,$ for all $n \in \mathbb{Z}^+.$ In the S3 algorithm, we often deal with H\"older continuous functions that have zero expectation, in which case, we use $|\langle J\circ\varphi^n \: f\rangle | \leq c \:\delta^n.$ As a result of exponential decorrelation, H\"older observables also satisfy the central limit theorem (CLT) and the law of the iterated logarithm which implies that for almost every $x \in M,$ the error in the $N$-time ergodic average of a H\"older observable $J$ declines asymptotically as ${\cal O}(\sqrt{\log\log N}/\sqrt{N})$: $|(1/N) \sum_{n=0}^{N-1} J_{\varphi^n x} - \langle J\rangle| \leq c\sqrt{\log\log N}/\sqrt{N},$ for large $N$ and some $c > 0.$

Further, we use the fact that the SRB measure, although typically singular with respect to Lebesgue measure on $\mathbb{R}^m$, has absolutely continuous conditional measures on the unstable manifold. We shall next elaborate on this property as used in the derivation of the S3 algorithm (section \ref{sec:algorithm}). 

\subsection{Parameterization of unstable manifolds}
\label{sec:coordinateSystem}
The unstable subspace is tangent to the local unstable manifolds. Given an $\epsilon > 0,$ a local unstable manifold at an $x \in \Lambda,$ $U_{x,\epsilon}$, contains points whose backward orbits lie $\epsilon$-close to the backward orbit of $x.$ That is, $$U_{x,\epsilon} = \left\{ x'\in M: \|x_{-n} - x'_{-n}\| \leq \epsilon,\;\forall\;n \in \mathbb{Z}^+; \lim_{n\to\infty} \|x_{-n} - x'_{-n}\| = 0 \right\}.$$  
Since $E^u_x$ are one-dimensional subspaces, the local unstable manifolds are also one-dimensional. According to the stable-unstable manifold theorem (see e.g. Theorem 6.2.8 of \cite{katok}, \cite{semyon}), the local unstable manifolds are embedded images of Euclidean spaces of the same dimension, which in this case are real lines. 

In this paper, we work with a particular $C^1$ parameterization of local unstable manifolds. Let $\Xi$ be a measurable partition of $\Lambda$ that is subordinate to the unstable manifold, and let $\Xi_x$ denote the element of the partition containing $x.$ At each $x,$ we can choose an $\epsilon$ depending on $x$ so that $\Xi_x$ contains a local unstable manifold at every $x.$ We choose a parameterization $\Phi^x:[-\epsilon_x,\epsilon_x]\to \Xi_x$ that satisfies the following properties:
\begin{enumerate}
    \item $\Phi^x(0) = x$
    \item $d_\xi \Phi^x(\xi) = q_{x'},$ where $\Phi^x(\xi) = x'$, for all $x'$ in the image of $\Phi^x$.
\end{enumerate}
Such a measurable function $x \to \epsilon_x$ exists that allows the definition of $\Phi^x$ (see Chapter 6 of \cite{katok}; \cite{ledrappier-young}).
Here $q_{x'} \in T_{x'} M$ is the unit vector in the one-dimensional tangent subspace $E^u_{x'}.$ From 2., it follows that $\norm{d_\xi\Phi^x((\Phi^x)^{-1}(x'))} = 1$, for all $x'$ in the image of $\Phi^x.$
Hence we refer to the pointwise coordinate maps $\Phi^x$ as the unit speed parameterization of local unstable manifolds.

\subsection{Iterative differentiation on the unstable manifold}
\label{sec:differentiationOnUnstableManifold}
The orbits of $\varphi^{-1}$ starting in $\Xi_x$ generate corresponding orbits on the real line by this parameterization. More concretely, we define the dynamics on the real line, through the map $\left(\tilde{\varphi}^x\right)^{-1} := \left({\Phi^{(\varphi^{-1} x)}}\right)^{-1} \circ\varphi^{-1}\circ\Phi^x.$ We frequently use the following relationship that can be derived using the chain rule, where $d_\xi$ denotes differentiation with respect to the coordinate $\xi$:
\begin{align}
    \label{eqn:expansionFactorOnRealLine}
    d_\xi \left(\tilde{\varphi}^x\right)^{-1}(\xi) = \dfrac{1}{\alpha_{x'}},
\end{align}
Here $\Phi^x(\xi) = x'$, and $\alpha$ is the scalar field that represents the local expansion factor.
\begin{definition}
We define a scalar field $\alpha:M\to\mathbb{R}^+$ to capture the local expansion of unstable tangent vectors. At each $x \in M,$
\begin{align}
\label{def:localExpansionFactor}
    \alpha_x = \|d\varphi_{\varphi^{-1}x} q_{\varphi^{-1}x}\|.
\end{align}
\end{definition}
Notice that the derivative of $\left(\tilde{\varphi}^x\right)^{-1}$ with respect to the unstable coordinate $\xi$ does not depend on the base point of the coordinate system, $x.$ This fact is crucial to the derivation of the S3 algorithm, where we prominently differentiate scalar and vector fields ``on unstable manifolds'' of a reference orbit. We now describe what this differentiation means in this paper, and a formula useful for performing this differentiation recursively.

Let $x_0 \in M$ be a fixed $\mu$-typical point whose forward orbit serves as the reference orbit for our computation. Let $f:M \to \mathbb{R}$ be a scalar field, and $h$ be the scalar field that represents the derivative of $f$ on the unstable manifold. We define this function, using our coordinate systems centered on $\left\{x_n\right\}$ as follows, $h_{x'_n} := d_\xi (f \circ\Phi^{x_n})(({\Phi^{x_n}})^{-1} (x'_n)),$ for some $x'_n \in \Xi_{x_n}.$ 

Now suppose we wish to compute $h$ recursively along an orbit; the situation where such a computation arises in the S3 algorithm is discussed in section \ref{sec:unstableContribution}. To compute the values $h_{x_n}$ using $\left\{h_{x_m}\right\}_{m< n},$ we use the following iteration that in turn uses Eq. \ref{eqn:expansionFactorOnRealLine},
\begin{align}
\notag 
    d_\xi (f\circ \varphi\circ\Phi^x) (0) &= d_\xi (f \circ \Phi^{\varphi x} \circ \tilde{\varphi}^{\varphi x}) (0) \\
    \label{eqn:chainRuleForPhix}
    &= h_{\varphi x} \: \alpha_{\varphi x}.
\end{align}
We sometimes use the notation $\partial_\xi f$ to denote $h$, which we simply refer to as the unstable derivative of $f$. 
\subsection{Conditional density of the SRB measure on unstable manifolds}
An important property of the SRB measure, which guarantees its physicality, is its absolute continuity along unstable manifolds \cite{srb}. To define this property, we must introduce conditional densities of the SRB measure on the unstable manifold, which we denote $\rho$. We may define the function $\rho$ using our parameterization $\left\{\Phi^x\right\}$ and disintegration to yield a family of functions $\left\{\rho^x\right\}.$ It is worth noting that changes to the parameterization only affect the normalization constant of the densities $\rho^x$. By the disintegration of SRB measure \cite{ledrappier-young}\cite{vaughn}  on the measurable partition $\Xi$ , there exist i) a family of conditional measures denoted 
$\mu^x$, at $\mu$-a.e. $x \in M,$ that are supported on $\Xi_x$, and ii) a quotient measure $\hat{\mu}$ on $M/\Xi,$ such that, for all Borel subsets $A \subseteq M,$
\begin{align}
\label{eqn:disintegration-pre}
    \mu(A) = \int_{M/\Xi} \mu^x(A \cap \Xi_x) \; d\hat{\mu}(x).
\end{align}
An aside on the notation used henceforth: a phase point $x \in M$ that appears in the superscript indicates the point at which our pointwise coordinate system is centered; this is distinct from its appearance on a subscript, which always means evaluation of a scalar/vector field at $x.$ 
Let $\nu^x$ be the normalized pushforward of the Lebesgue measure (uniform probability distribution) on $[0,1],$ by $\Phi^x.$  
The absolute continuity property of the SRB measure on the unstable manifold means that the conditional measure $\mu^x$ is absolutely continuous with respect to $\nu^x$. Thus, a scalar function $\rho^x$ can 
be defined as the \emph{probability density} of $\mu^x.$ In particular,
the unnormalized density $\rho^x_{\rm u}$ of $\mu^x$ has been derived by Pesin \cite{pesin}\cite{miahao} to be,
\begin{align}
    \label{eqn:unnormalizeddensity}
    \rho^x_{\rm u}(x') := \prod_{k=0}^\infty \dfrac{\alpha_{\varphi^{-k} x}}{\alpha_{\varphi^{-k} x'}},
\end{align}
where the local expansion factor $\alpha$ is as defined in Eq. \ref{def:localExpansionFactor}. 
\begin{definition}
\label{defn:density}
The probability density of conditional SRB measures on the unstable manifold is defined as $\rho^x({x'}) = \rho^x_{\rm u}(x')/\bar{\rho}^x,$ for $x' \in \Xi_x,$
where $\bar{\rho}^x := \int_{\Xi_x} \rho^x_{\rm u}(x') 
\; d\nu^{x}(x'),$ is the normalization constant.
\end{definition}
Using the above definition of the density, the disintegration of the SRB measure on the unstable manifold (Eq. \ref{eqn:disintegration-pre}) gives, for any smooth observable $f$,
\begin{align}
    \label{eqn:disintegration}
    \langle f \rangle = \int_{M/\Xi} \int_{\Xi_x} f(x')\; \rho(x')\; d\nu^x(x')\; d\hat{\mu}(x).
\end{align}
\subsection{Linear response of uniformly hyperbolic systems}
\label{sec:linearResponse}
Ruelle \cite{ruelle}\cite{ruelle1} showed that linear response holds and is given by the following formula in uniformly hyperbolic attractors, 
\begin{align}
    \label{eqn:ruellesFormula}
    \langle J, \partial_s \mu_s \rangle 
    = \sum_{k=0}^\infty \langle d (J\circ\varphi^k) \cdot \chi, \mu \rangle,
\end{align}
where, as usual, we have dropped the subscript on $\varphi$ and $\mu$ to indicate their respective values at $s = s_0.$ According to Ruelle's formula (\ref{eqn:ruellesFormula}), linear response or the parametric derivative of statistics, is a series summation of ensemble averages. The sensitivity of the function $J \circ \varphi^k$ to the parameter perturbation $\chi$ is given by $d(J\circ\varphi^k)\cdot \chi,$ and the $k$th term in Ruelle's formula is an ensemble average of this instantaneous sensitivity. With probability 1, the instantaneous sensitivity, $d(J\circ\varphi^k)\cdot \chi,$ grows in norm exponentially with $k$ since $\chi$ almost surely has a non-zero component in $E^u.$ But, due to cancellations over phase space, the ensemble average is bounded at all $k.$ Further, the convergence of Ruelle's series implies that the ensemble average of instantaneous sensitivities decreases asymptotically with time and converges to zero.  

Numerically, the direct evaluation of Ruelle's formula (\ref{eqn:ruellesFormula}) involves approximating each ensemble average (each term in the series) as a sample average of instantaneous sensitivities. These instantaneous sensitivities are, in turn, computed by using the conventional tangent equation (section \ref{sec:tangentDynamics}) or when the dimension of the parameter space is large, using the adjoint equation \cite{angxiu}\cite{eyink}\cite{nisha_ES}. One may also use automatic differentiation in the forward and reverse mode to approximate the tangent and adjoint solutions respectively. However, the resulting sensitivity from any linear perturbation method increases exponentially with $k$ at the rate of the largest Lyapunov exponent, $\lambda_1;$ hence considering the sample average of the sensitivities as a random variable, its variance increases as ${\cal O}(e^{2\lambda_1 k}).$ Thus, the number of samples needed to approximate the integral accurately rapidly increases with $k.$ On the other hand, thresholding the series computation (Eq. \ref{eqn:ruellesFormula}) at a small value of $k$ to reduce the variance introduces a bias. This bias-variance trade-off has been analyzed in previous works \cite{eyink}\cite{nisha_ES}, and the direct evaluation of Ruelle's formula rendered infeasible in practical systems.

\section{Main results: S3 decomposition and computation of Ruelle's response formula}
\label{sec:main}
The two key ideas that enable an efficient computation of Ruelle's formula are i) a particular decomposition of the perturbation field $\chi$ that splits Ruelle's formula into {\em stable} and {\em unstable} contributions, and ii) the development of iterative, ergodic-averaging evaluations of these two components. The first main result concerns the specific decomposition of a given parameter perturbation that we show enables an efficient computation.    
\begin{theorem}
\label{thm:thmS3}
A differentiable vector field $\chi$ has a sequence of decompositions $\chi = a^n\:q \:+\: \: (\chi - a^n\:q)$, where $q$ is the unit vector field tangent to the one-dimensional unstable manifold, such that 
\begin{enumerate}
    \item the sequence of vector fields, $\left\{v^n\right\}_{n \in \mathbb{Z}^+}$ that satisfies
    \begin{align}
        \label{eqn:stableTangentVectorField}
        v^{n+1} &:= d\varphi\: v^n + (\chi - a^{n+1} q) ,\: \;\;\: n \in \mathbb{Z}^+, 
    \end{align}
    with $v^0$ being any bounded vector field that is differentiable on the unstable manifold, converges uniformly; the sequence of scalar fields $\left\{a^n\right\}$ is chosen to orthogonalize $v^n$ to the unstable manifold: $v^n_x\cdot q_x = 0,$ at all $x \in M$ and for all $n \in \mathbb{Z}^+$;
    \item the sequence $\left\{a^n\right\}$ is differentiable on the unstable manifold and uniformly converging;
    and,
    \item the sequence of unstable derivatives of $\left\{a^n\right\}$, namely, $\left\{b^n := \partial_\xi a^n\right\}$ converges uniformly.
\end{enumerate}
\end{theorem}
Let the limits of the sequences we introduced in Theorem \ref{thm:thmS3} be denoted as follows: $a := \lim_{n\to\infty} a^n$, $v := \lim_{n\to\infty} v^n$ and $b := \lim_{n\to\infty} b^n$. From Theorem \ref{thm:thmS3}-2. and \ref{thm:thmS3}-3., the scalar field $b$ is the unstable derivative of $a$, i.e., $ b = \partial_\xi a$. We use the sequence $\left\{v^n\right\}$ to split Ruelle's formula into two computable infinite series. Fixing some finite $K$ that controls the accuracy of the implementation, the S3 decomposition can be construed as a sequence of decompositions of $\chi$ into $\chi = (a^{K-k}\:q) + (\chi - a^{K-k}\:q), \;0\leq k < K$, wherein the former component is tangent to the unstable manifold. This gives the following decomposition, which converges to Ruelle's formula, as $K \to\infty,$  
    \begin{align}
    \label{eqn:splitRuellesFormula1}
    \langle J, \partial_s \mu_s \rangle &=  \lim_{K\to\infty} \sum_{k=0}^{K-1} \langle d(J\circ\varphi^k)\cdot (\chi - a^{K-k} \: q), \mu\rangle + \lim_{K\to\infty}\sum_{k=0}^{K-1} \langle d(J\circ\varphi^k) \cdot a^{K-k}\:q, \mu\rangle.
\end{align}
We refer to the first term on the right hand side as the {\em stable contribution}, and the second as the {\em unstable contribution} to the overall sensitivity. In section \ref{sec:splitting}, we use Theorem \ref{thm:thmS3} to show that the stable and unstable contributions can be alternatively expressed in the following forms that are amenable to their computation.
\begin{align}
\label{eqn:splitRuellesFormula2}
   \langle J, \partial_s \mu_s \rangle &= \langle dJ\cdot v, \mu\rangle  -\lim_{K\to\infty} \sum_{k=0}^{K-1} \langle J\circ\varphi^k \; (a^{K-k}\: g + b^{K-k}) , \mu \rangle,
\end{align}
where we have introduced the {\em logarithmic density gradient} $g$, which is defined, for $x \in \Xi_{x'},$ as 
\begin{align}
\label{eqn:definitionOfg}
    g_{x} := \dfrac{1}{\rho^{x'}(x)} \frac{d(\rho^{x'} \circ \Phi^{x'})}{d\xi}((\Phi^{x'})^{-1}(x)).
\end{align} 

The intuition for the S3 decomposition is developed in section \ref{sec:splitting}, and the
proof of Theorem \ref{thm:thmS3} is presented in section \ref{sec:proofOfTheorem1}.  

\subsection{The S3 algorithm}
\label{sec:algorithm}
The S3 algorithm is an efficient computation of the above split Ruelle's formula (Eq. \ref{eqn:splitRuellesFormula2}). The computation of the different ingredients of this new formula as an efficient ergodic average is the other major contribution of this paper. In particular, we develop iterative algorithms to compute $v$, $g$ and $b$ along trajectories. Before we detail the S3 algorithm, we recall an important notation used throughout this paper. 
\begin{notation}
When a $\mu$-typical phase point $x_0 \in M$ is fixed, the subscript $n$ applied to a scalar function or a vector field, denoted $f$, refers to its corresponding value at $x_n.$ That is, when $x_0 \in M$ is fixed, $f_n := f_{\varphi^n x_0} = f_{x_n}.$
\end{notation}
The S3 algorithm is as follows:
\begin{enumerate}
    \item Obtain a long primal trajectory $x_{-K'},\cdots,x_{N-1},$ where $x_{n+1} = \varphi x_n,\;\; -K'\leq n\leq (N-1),$ with $x_{-K'}$ chosen $\mu$-a.e.
    \item Obtain, at each point $x_n$, the unit tangent vector to the unstable manifold, $q_n$. The following procedure converges exponentially in $n$ to the true value of $q_n$. Solve the homogeneous tangent equation with repeated normalization. That is, solve
    \begin{align}
        \alpha_{n+1}\: q_{n+1} = \left(d\varphi\right)_n\: q_n,\; n=-K',\cdots,0,1,\cdots,
    \end{align}
    with $q_0$ being set to a random vector in $\mathbb{R}^m,$ and $
    \alpha_{n+1} = \|(d\varphi)_n\: q_n\|.$
    \item Solve for $v_n$ the following inhomogeneous tangent equation, which repeatedly projects $v_n$ out of the unstable subspace,
    \begin{align}
        \label{eqn:stableTangentEquationInsideAlgorithm}
        v_{n+1} = \left(d\varphi\right)_n v_n + \chi_{n+1} - a_{n+1} q_{n+1}, \; n=-K',\cdots,0,1,\cdots,
    \end{align}
    where $a_{n+1}$ is such that $v_{n+1}\cdot q_{n+1} = 0.$ The initial condition $v_{-K'}$ is set to $0 \in \mathbb{R}^m.$ This equation is henceforth called the {\em regularized} tangent equation. The error in the solutions $v_n$ when compared to the true value of the vector field $v$ (Theorem \ref{thm:thmS3}) decreases exponentially with $n$ (as shown in Lemma \ref{lem:existenceAndUniquenessOfv}).
    \item Solve the following second-order tangent equation for $p_n$, starting with $p_{-K'} =0 \in \mathbb{R}^m$,
    \begin{align}
    \label{eqn:selfderivativeOfUnstableDirection}
    p_{n+1} = \dfrac{(d^2\varphi)_n(q_n, q_n) + (d\varphi)_n p_n}{\alpha_{n+1}^2},\; n = -K',\cdots,0,1,\cdots,
\end{align}
The solutions $p_n$ converge exponentially with $n$ to the true values of the vector field $p$ along the orbit (Lemma \ref{lem:existenceAndUniquenessOfw}). 
    \item Solve the following recursive second-order tangent equation for $y_n$, for each $n \in \mathbb{Z}^+,$
    \begin{align}
    \label{eqn:derivativeOfStableTangent}
    y_{n+1} &=  \dfrac{(d^2\varphi)_n(q_n, v_n) + (d\varphi)_n y_n}{\alpha_{n+1}} + (d\chi)_{n+1}\: q_{n+1} \\
    \notag 
    &- c_{n+1} q_{n+1} - a_{n+1} p_{n+1}, \;n= -K',\cdots,0,\cdots,
    \end{align}
    where the scalar $c_{n+1}$ is found using the relationship $y_{n+1}\cdot q_{n+1} = - v_{n+1}\cdot p_{n+1}.$ This relationship follows from taking the unstable derivative of $v_{n+1} \cdot q_{n+1} = 0.$ The $y_{n}$ and $c_{n}$ computed using this procedure also converge exponentially with $n$ to their true values, as shown in Lemma \ref{lem:existenceAndUniquenessOfy} and \ref{lem:convergenceOfb} respectively.  
    \item Compute the unstable contribution as the following ergodic average:
    \begin{align}
        \label{eqn:unstableContributionFinalExpression}
        \langle J,\partial_s\mu_s\rangle^{\rm u}  \approx
        -\dfrac{1}{N}\sum_{k=0}^{K-1} \sum_{n=0}^{N-1} J_{n+k} c_n.
    \end{align}
    \item Compute the stable contribution as the following ergodic average:
    \begin{align}
        \label{eqn:stableContributionFinalExpression}
        \langle J,\partial_s\mu_s\rangle^{\rm s}  \approx \dfrac{1}{N}\sum_{n=0}^{N-1} (dJ)_n \cdot v_n.
    \end{align}
    \item The output of the S3 algorithm is the sum of the stable and unstable contributions,
    $\langle J,\partial_s\mu_s\rangle^{\rm u} + \langle J,\partial_s\mu_s\rangle^{\rm s}$.
\end{enumerate}
The following theorem establishes the convergence of the above S3 algorithm.
\begin{theorem}
\label{thm:thmS3Computation}
The ergodic averages in Eq. \ref{eqn:unstableContributionFinalExpression} and Eq. \ref{eqn:stableContributionFinalExpression} converge to the unstable and stable contributions respectively. In particular, 
\begin{enumerate}
    \item for $\mu$-a.e. $x_0$ and for almost every $q_0$, 
    \begin{align}
   \langle J, \partial_s \mu_s \rangle^{\rm u} 
   \label{eqn:unstableContribution}
   &= -\lim_{K\to\infty}\lim_{N\to \infty} \sum_{k=0}^{K-1}     \dfrac{1}{N}\sum_{n=0}^{N-1}  J_{n+k} \; c_n ;
    \end{align}
    \item And, for $\mu$-a.e. $x_0,$     
    \begin{align}
          \langle J,\partial_s\mu_s\rangle^{\rm s} &:= 
         \label{eqn:stableContribution}
         \lim_{N\to\infty} \dfrac{1}{N}\sum_{n=0}^{N-1} (dJ)_n \cdot v_n.
    \end{align}
  
\end{enumerate}
\end{theorem}
We prove Theorem \ref{thm:thmS3Computation} in section \ref{sec:proofOfTheorem2}. The efficiency of the S3 algorithm in comparison to a direct evaluation of Ruelle's formula (Eq. \ref{eqn:ruellesFormula}) stems from the following. The integrand in Ruelle's original series increases in norm exponentially with $k.$ This makes the ergodic averaging approximation of the integral computationally inefficient, as we noted in section \ref{sec:linearResponse}. By contrast, the norm of the integrand in the S3 modification of the formula (Eq. \ref{eqn:splitRuellesFormula2}) is uniformly bounded in both the stable and unstable contributions. Thus, efficient ergodic-averaging approximations are possible for the S3 formula (Eq. \ref{eqn:splitRuellesFormula2}). 
\section{Derivation and computation of the S3 formula}
\label{sec:splitting}
In this section and the next, we describe the S3 decomposition of
Ruelle's formula (Eq. \ref{eqn:splitRuellesFormula1}), 
the derivation of its regularized form (Eq. \ref{eqn:splitRuellesFormula2}), and the computation of the latter. In order to derive an efficient computation of Ruelle's formula, we split Ruelle's formula into two parts, using a particular decomposition of the 
vector field. The contribution to the overall sensitivity made by the unstable component, which is aligned with $q$, is called the unstable contribution; the remaining term is the stable contribution (Eq. \ref{eqn:splitRuellesFormula1}).

However, we do not decompose $\chi$ into its stable and unstable components. That is, although $aq \in E^u$ clearly, $a$ is not chosen such that $\chi - aq$ belongs to $E^s$. Instead, the scalar field $a$ is chosen so as to be differentiable on the unstable manifold and such that both the stable and unstable contributions lead to well-conditioned computations. The significance of the differentiability of $a$ on the unstable manifold will be clear in the derivation below.

\subsection{Regularizing tangent equation solutions}
The S3 decomposition can be motivated as a means of regularizing the solutions of a conventional tangent equation. Fixing a reference orbit $\left\{x_n\right\}_{n\in \mathbb{Z}^+}$, consider the conventional tangent equation (Eq. \ref{eqn:inhomogeneousTangent}), which is a recursive equation for $u_n := (\partial_s \varphi^n_s)(x_0,s)$, starting from $u_0 = 0 \in \mathbb{R}^m$. As noted in section \ref{sec:tangentDynamics}, this equation gives the evolution of a tangent vector corresponding to the parameter perturbation along a fixed trajectory, $\left\{ x_n\right\}_{n\in \mathbb{Z}^+}.$ By construction (Eq. \ref{eqn:inhomogeneousTangent}), we can see that $u_n = \sum_{k=0}^{n-1} (d\varphi^k)_{n-k} \: \chi_{n-k}.$ By definition of chaos (section \ref{sec:chaos}), $\|d\varphi^k_x \chi_x\| \sim {\cal O}(e^{\lambda_1 k}),$ for almost every $x \in M,$ and almost every perturbation $\chi_x \in \mathbb{R}^m.$ Hence, for large $n,$ $\|u_n\| \sim {\cal O}(e^{\lambda_1 n})$ for almost every perturbation $\chi$. On the other hand, if we projected out the unstable component of the tangent solution at each timestep, the solution does not exhibit exponential growth. That is, the following iteration is stable
\begin{align}
    \label{eqn:stableTangentEquation}
    v_{n+1} = (d\varphi)_{n} v_{n} + (\chi_{n+1} - a_{n+1} q_{n+1}),
\end{align}
where $a_{n+1}$ is such that $v_{n+1}\cdot q_{n+1} = 0,$ for all $n \in \mathbb{Z}^+.$
That is, the solution $v_{n}$ of the above tangent equation along with the repeated projections out of the unstable subspace, is in $E^{u\perp}_{x_n}$ at each $n.$ We refer to the solutions $\left\{v_n\right\}$ as {\em regularized tangent solutions.} With this stable iterative procedure as the motivation, we derive a splitting of Ruelle's formula. One part of the split formula -- the stable contribution -- will be computed using the regularized tangent solution, Eq. \ref{eqn:stableTangentEquation}.
\subsection{Alternative expression of the stable contribution}
\label{sec:stableContribution}
In the regularized tangent equation (Eq. \ref{eqn:stableTangentEquation}), a tangent vector is projected out of the unstable subspace at every iteration. At every $x \in M,$ we introduce the orthogonal projection operator, $\unOrthProjection_x$, to refer to this operation.
\begin{notation}
Let $\unOrthProjection_x: T_x M \to T_x M$ denote the linear operator
that projects a tangent vector onto the hyperplane
orthogonal to $E^u_x.$ If $q_x$ is a unit vector in $E^u_x$, then for any $v_x\in T_x M$,
$\unOrthProjection_x\:v_x := (I - q_x q_x^T) v_x.$
Applying $\unOrthProjection_x$ to every point on $M$ would result in a linear operator, hereafter denoted $\unOrthProjection$, on vector
fields of $M$.
\end{notation}

Using this notation, we can reproduce the operations performed by solving Eq. \ref{eqn:stableTangentEquation}, by considering a sequence of vector fields $\left\{ v^n\right\}$ that satisfies the following recurrence relation,
\begin{align}
    \label{eqn:regularizedTangentEquation}
    v^{n+1} = \unOrthProjection (d\varphi\: v^n + \chi), \;\;n\in \mathbb{Z}^+.  
\end{align}
We also introduce a sequence of scalar fields which denote the components on the unstable subspace that are projected out every iteration. That is, let $\left\{ a^n\right\}$ be a sequence of scalar fields given by,
\begin{align}
\label{eqn:an}
    a^{n+1} := q^T (d\varphi\: v^n + \chi). 
\end{align}
In Lemma \ref{lem:existenceAndUniquenessOfv}, we show that the sequence $\left\{v^n\right\}$ converges uniformly, starting from any bounded vector field $v^0;$ the limit of this sequence is denoted $v.$ The uniform convergence of $\left\{ v^n\right\}$ implies the uniform convergence of the scalar field $\left\{a^n\right\}$, as we show in Lemma \ref{lem:convergenceOfa}; the limit of this sequence is denoted $a$. Using these results and uniform hyperbolicity, we prove the following alternative formula for the stable contribution in Proposition \ref{lem:convergenceOfStableContribution}:
\begin{align}
\notag 
    \langle J, \partial_s \mu_s \rangle^{\rm s} &:= \lim_{K\to\infty} \sum_{k=0}^{K-1} \langle d(J\circ\varphi^k)\cdot (\chi - a^{K-k}\: q), \mu \rangle \\
\label{eqn:stableContributionAlternativeExpression}
    &= \langle dJ\cdot v, \mu_s \rangle.
\end{align}
\subsection{Computation of the stable contribution}
\label{sec:stableContributionComputation}
That is, the series summation representing the stable contribution is simply an ensemble average of $dJ\cdot v,$ an inner product of two bounded vector fields. That is, for $x_0$ chosen $\mu$-a.e.,
\begin{align}
 \label{eqn:stableContributionErgodicAverage}
    \langle J, \partial_s \mu_s\rangle 
    = \lim_{N\to\infty} \dfrac{1}{N}\sum_{n=0}^{N-1}
    (dJ)_n \cdot v_n. 
\end{align}
The values of the vector field $v$ along the reference orbit are approached by the regularized tangent solution (Eq. \ref{eqn:stableTangentEquation}). Without loss of generality, we effect starting the recursive Eq. \ref{eqn:regularizedTangentEquation} from a zero vector field, by choosing $v_{-K'} = 0 \in \mathbb{R}^m$, at the point $x_{-K'}.$ For a large run-up time $K'$, the solution at 0, $v_0$ and $a_0$ are already close to the true values of $v$ and $a$ at $x_0$. The solutions thus produced by long-time evolution of the regularized tangent equation (Eq. \ref{eqn:stableTangentEquation}) become exponentially more accurate, and are used to evaluate the stable contribution as per Eq. \ref{eqn:stableContributionErgodicAverage}. In practice, we solve the regularized tangent equation (Eq. \ref{eqn:stableTangentEquation}) over an orbit of a long but finite length $N$. This computes an approximation of the limit on the right hand side up to a finite $N.$ Proposition \ref{lem:convergenceOfStableContribution} shows that such a computation of the stable contribution converges to its true value (Eq. \ref{eqn:stableContributionAlternativeExpression}) as $N \to \infty$; the asymptotic error convergence is as ${\cal O}(1/\sqrt{N}).$
\subsection{Alternative expression of the unstable contribution}
\label{sec:integrationByParts}
Recall that the unstable contribution is the sensitivity due to the perturbation along the unstable subspace,
\begin{align}
\label{eqn:unstableContributionFirst}
    \langle J, \partial_s \mu_s\rangle^{\rm u} = \lim_{K\to\infty} \sum_{k=0}^{K-1} \langle d(J\circ\varphi^k)\cdot a^{K-k}\:q, \mu\rangle.
\end{align}
As noted in section \ref{sec:linearResponse}, the integrand in the $k$th term of the above series, $d(J\circ\varphi^k)\cdot a^{K-k} \:q$, increases exponentially in norm with $k$. Thus, rather than a direct evaluation, we apply integration by parts on the unstable manifold, which moves the derivative away from the time-dependent term $J\circ\varphi^k.$

Before we integrate by parts, we apply disintegration 
(Eq. \ref{eqn:disintegration}) of the SRB measure to rewrite the unstable contribution,
\begin{align}
\notag 
    &\langle J, \partial_s\mu_s\rangle^{\rm u} := \lim_{K\to\infty} \sum_{k=0}^{K-1} \langle a^{K-k}\; d(J\circ\varphi^k) \cdot q, \mu \rangle \\
    \label{eqn:unstableContributionExpansion}
    &= \lim_{K\to\infty} \sum_{k=0}^{K-1} \int_{M/\Xi} \int_{B_x} a^{K-k} \circ \Phi^x(\xi) \; \dfrac{d(J\circ\varphi^k\circ \Phi^x)}{d\xi}(\xi) \; \rho^x\circ\Phi^x (\xi) \; d\xi \; d\hat{\mu}(x),
\end{align}
where $B_x := {\Phi^x}^{-1}(\Xi_x).$ When compared to Eq. \ref{eqn:disintegration}, we have additionally used a change of variables, $x'\to \xi,$ in the inner integral. Since $\norm{{\Phi^{x}}'(\xi)} = 1,$ by construction, this change of variables does not introduce a multiplicative factor in the integrand. Now applying integration by parts on the inner integral, the $k$th term on the right hand side of Eq. \ref{eqn:unstableContributionExpansion} becomes,
\begin{align}
\label{eqn:integrationByParts}
\notag
&\int_{M/\Xi} \int_{\partial B_x} \dfrac{d((a^{K-k} \: J\circ\varphi^k\: \rho^x) \circ \Phi^x)}{d\xi}  d\xi \; d\hat{\mu}(x) \\
&- \int_{M/\Xi} \int_{B_x} J\circ\varphi^k\circ\Phi^x 
\Big( \dfrac{d(a^{K-k}\circ \Phi^x)}{d\xi} + \dfrac{a^{K-k}\circ \Phi^x}{\rho \circ\Phi^x} \dfrac{d(\rho^x\circ\Phi^x)}{d\xi}  \Big)
(\rho^x \circ\Phi^x) d\xi \; d\hat{\mu}(x).
\end{align}
The first term in the above equation vanishes due to cancellations on the boundaries of $B_x$ for different $x$ (Theorem 3.1(b) in \cite{ruelle}\cite{ruelle1}; see also \cite{miahao}). Changing variables back to $x'$, we obtain 
the following regularized expression for the unstable contribution, 
\begin{align}
\notag
    \langle J, \partial_s\mu_s\rangle^{\rm u} &= 
- \lim_{K\to\infty} \sum_{k < K} \int_{M/\Xi} \int_{\Xi_x} J\circ\varphi^k(x')\:  
\Big( \dfrac{ d(a^{K-k}\circ \Phi^x)}{d\xi}(({\Phi^x})^{-1} (x')) \\
& + \dfrac{a^{K-k}(x')}{\rho^x(x')} \dfrac{d (\rho^x\circ\Phi^x)}{d\xi}(({\Phi^x})^{-1} (x'))  \Big)\; \rho^x(x') \; d\nu_{x}(x') \; d\hat{\mu}(x),
\label{eqn:unstableContributionRegularized}
\end{align}
We introduce the logarithmic density gradient function \cite{adam}\cite{angxiu-s3}, 
\begin{align}
g^x(x'):= \dfrac{1}{\rho^x(x')}\dfrac{d (\rho^x\circ\Phi^x)}{d\xi}({(\Phi^x)}^{-1} (x')).
    \label{eqn:densityGradient}
\end{align}
As we show in section \ref{sec:formulaforg}, $g^x$ does not depend on $x$, and hence we denote the density gradient
simply as $g$. We also introduce the scalar field sequence $b^k$ defined as $b_x^k := d_\xi (a^k\circ\Phi^x)(0)$ to denote the unstable derivative of $a^k$, deferring the (constructive) proof of its uniform convergence until Lemma \ref{lem:convergenceOfb}. This leads us to the following expression for the unstable contribution,
\begin{align}
\label{eqn:unstableContributionRegularizedExpression}
    \langle J, \partial_s \mu_s \rangle^u = -\lim_{K\to\infty} \sum_{k = 0}^{K-1} \langle J\circ \varphi^k (g\: a^{K-k} + b^{K-k}), \mu \rangle. 
\end{align}

Recall that we compute $a^k$ iteratively along a typical trajectory as part of the stable contribution computation (section \ref{sec:stableContribution}). We are now left with the computation of $g$ and $b^k$ along the trajectory. These 
are tackled in section \ref{sec:formulaforg} and \ref{sec:formulaforb} respectively. 
   Ni \cite{angxiu-s3} addresses the computation of the divergence on the unstable manifold (analogous to the function $g$ described above), where the differentiation is performed on a set of {\em shadowing} coordinates that are tied to the parameter perturbation. This is part of an algorithm to compute linear response that uses a different decomposition from the present paper, into shadowing and unstable directions. However, the shadowing direction, $v^{\rm sh}$ in \cite{angxiu-s3} is related to our regularized tangent vector field $v$, by the relation $\mathcal{P} v^{\rm sh} = v$ (see \cite{angxiu-s3} or Appendix \ref{appx:shadowing}). 
\section{Computing derivatives along unstable manifolds}
\label{sec:unstableContribution}

 In the previous section, we obtained a regularized expression for the unstable contribution. That is, the integrand in the regularized expression, Eq. \ref{eqn:unstableContributionRegularizedExpression}, is uniformly bounded  -- the uniform boundedness of the scalar field sequences $\left\{a^k\right\}$ and $\left\{b^k\right\}$ and the boundedness of $g$ are proved in Lemma \ref{lem:convergenceOfa}, \ref{lem:convergenceOfb} and \ref{lem:convergenceOfg} respectively. The question still remains how we can compute the unknown scalar fields $g$ and $b.$ For both these computations, we require second-order unstable derivatives, which we discuss next. 
 \subsection{Iteratively computing the curvature of the unstable manifold}
\label{sec:formulaforw}
 The function $\Phi^{x'}(\xi)$ is an arclength travelled by a particle on a local unstable manifold of $x'$, starting from $x'$. The particle travels with unit speed, and its instantaneous velocity at time $\xi$ is $q_{\Phi^{x'}(\xi)}.$ Its acceleration field is given by $w^{x'} := d^2_\xi \Phi^{x'}(\xi).$ We derive a recursive equation satisfied by this family, by starting with a differentiation with respect to $\xi$ of the definition of $\alpha$ (Eq. \ref{def:localExpansionFactor}),
\begin{align}
    \label{eqn:derivativeOfalpha}
    \alpha_{\varphi x}^2 \: \gamma_{\varphi x}^{\varphi x'} = (d\varphi_x q_x)^T \big( d\varphi_x w_x^{x'} + d^2\varphi_x(q_x, q_x)\big).
\end{align}
Here $\gamma^{x'}_x := d_\xi (\alpha \circ \Phi^{x'})((\Phi^{x'})^{-1} x)$ is the unstable derivative of $\alpha$; $d^2\varphi_x$ is the bilinear form that returns a tangent vector $\in T_{\varphi x} M;$ it can be written as an $m\times m\times m$ tensor consisting of component-wise partial derivatives of the Jacobian $d\varphi_x.$
In deriving Eq. \ref{eqn:derivativeOfalpha}, we have used the chain rule in Eq. \ref{eqn:chainRuleForPhix}. 
Now differentiating the statement of covariance of the unstable subspace (section \ref{sec:uniformHyperbolicity}), $\alpha_{\varphi x} q_{\varphi x} = d\varphi_x q_x$,
\begin{align}
\label{eqn:derivativeOfalphaqRelationship}
    \alpha_{\varphi x} \gamma^{x'}_{\varphi x} q_{\varphi x} + 
    \alpha_{\varphi x}^2 w_{\varphi x}^{\varphi x'} = d\varphi_x w_x^{x'} + d^2\varphi_x(q_x, q_x).
\end{align}
Substituting Eq. \ref{eqn:derivativeOfalpha} into Eq. \ref{eqn:derivativeOfalphaqRelationship}, and using the definition of $\alpha$, 
\begin{align}
\label{eqn:computationOfw}
    w_{\varphi x}^{\varphi x'} = \dfrac{1}{\alpha_{\varphi x}^2} \left( I - q_{\varphi x} q_{\varphi x}^T\right) \Big( d\varphi_x w_x^{x'} + d^2\varphi_x(q_x, q_x)\Big),
\end{align}
where $I$ is the $m\times m$ identity matrix. The above equation represents the following relationship between elements of the family of vector fields, $\left\{ w^x\right\}$,
\begin{align}
    w^{\varphi x} = \mathcal{P} \dfrac{d\varphi \: w^x + d^2\varphi(q, q) }{\alpha^2}.
\end{align}
In Lemma \ref{lem:existenceAndUniquenessOfw}, we show that any bounded sequence $\left\{w^n\right\}$ of vector fields that satisfies, 
\begin{align}
\label{eqn:secondOrderTangentEquationForw}
    w^{n+1} = \mathcal{P} \dfrac{d\varphi \: w^n + d^2\varphi(q,q)}{\alpha^2}
\end{align}
 converges uniformly to a unique vector field $w.$ Hence, the family $\left\{ w^x \right\}$ does not depend on the parameterization centers, and simply denotes a single vector field, which we call $w.$ Reconsidering Eq. \ref{eqn:derivativeOfalpha},
 \begin{align}
      \gamma_{\varphi x}^{\varphi x'} = \dfrac{q_{\varphi x}^T}{\alpha_{\varphi x}} \big( d\varphi_x w_x + d^2\varphi_x(q_x, q_x)\big),
 \end{align}
 it is clear that the family of scalar fields $\left\{ \gamma^x\right\}$ is also independent of the parameterization. We henceforth write $\gamma = \partial_\xi \alpha,$ to denote the unstable derivative of the scalar field $\alpha.$
 
  Similarly, the scalar fields $d_\xi (a \circ \Phi^x)$ and $d_\xi (\log{\rho^x}\circ\Phi^x)$ can be attained as limits of (exponentially) uniformly converging sequences of scalar fields, as proved in Lemma \ref{lem:convergenceOfb} and Lemma \ref{lem:convergenceOfg} respectively. Thus, by the same argument as we used for $\left\{w^x\right\}$, these limits are independent of the parameterization centers and are denoted simply $b := \partial_\xi a$ and $g := \partial_\xi \log{\rho^x}$ respectively. 
  
 The curvature of a local unstable manifold is the magnitude of the acceleration experienced by a particle travelling at unit speed. That is, the curvature of the unstable manifold at $x$ is $\|w_x\|.$ From Eq. \ref{eqn:derivativeOfalpha}, it is clear that the computation of the curvature and that of $\gamma$ go hand in hand. To compute both $w$ and then using Eq. \ref{eqn:derivativeOfalpha}, $\gamma$, along a typical trajectory, we solve the second-order tangent equation derived above (Eq. \ref{eqn:secondOrderTangentEquationForw}) iteratively along the trajectory.
 In practice, we assume $w_0^0 := w_{x_0}^0 = 0$ at some $\mu$-typical $x_0$ and iterate Eq. \ref{eqn:secondOrderTangentEquationForw}. Such a computation converges exponentially with $n$ to the true value of $w_{n}$ due to Lemma \ref{lem:existenceAndUniquenessOfw}. At each step of the recursion, the value of $\gamma_{n}$, which is computed through Eq. \ref{eqn:derivativeOfalpha} using the computed values of $w_{n}$, also becomes exponentially more accurate. These values of $\gamma_n$ along the orbit $\left\{x_n\right\}$ are used to obtain $g$ along the orbit, as we shall discuss in the next subsection.
 
 From Eq. \ref{eqn:secondOrderTangentEquationForw}, we can also infer that the vector field $w$ is orthogonal to $q$ (i.e., the acceleration of a particle on a local unstable manifold is perpendicular to its velocity). Visualizations of the unstable manifold curvatures obtained from the norms of $w_n$ computed as above are shown for classical hyperbolic attractors in \cite{nisha_clv}. 

\subsection{Iterative formula for unstable derivatives of SRB density}
\label{sec:formulaforg}

In a previous work \cite{adam}, we provide an intuitive explanation of the density gradient $g$ as well as its computation in one-dimensional expanding maps of the interval, where the global unstable manifold is the background manifold $M$. In the present setting of a one-dimensional unstable manifold, we see that our derivation leads to a similar computation of $g$ as in 1D expanding maps. From our expression for the SRB 
density (Eq. \ref{defn:density}), we find that for $x \in \Xi_{x'},$ 
\begin{align}
    \label{eqn:densityRecursion}
    \rho^{x'}_{x} = \rho^{\varphi^{-1} x'}_{\varphi^{-1} x} \dfrac{\alpha_{x'}}{\alpha_{x}}
\dfrac{\bar{\rho}^{\varphi^{-1} x'}}{\bar{\rho}^{x'}}.
\end{align}
We recall that phase points appearing on superscripts indicate the centers of our coordinate system $\Phi^x$ (section \ref{sec:coordinateSystem}), while on subscripts, they indicate evaluations of the function (e.g., in Eq. \ref{eqn:densityRecursion}, the scalar functions $\rho^{x'}$ and $\rho^{\varphi^{-1} x'}$ are evaluated at $x$ and $\varphi^{-1}x$ respectively). Taking logarithm and differentiating Eq. \ref{eqn:densityRecursion} with respect to $\xi$ on the unstable manifold at $x$, using the definition (Eq. \ref{eqn:densityGradient}) of the 
density gradient $g^x$,
\begin{align}
   \notag 
   g^{x'}_{x} &= \dfrac{g^{\varphi^{-1} x'}_{\varphi^{-1} x}}{\alpha_{x}}  - 
   \dfrac{1}{\alpha_{x}}\dfrac{d(\alpha\circ\Phi^{x'})}{d\xi}((\Phi^{x'})^{-1}(x)) \\
 \label{eqn:densityRecursionWithUnknown}
&= \dfrac{g^{\varphi^{-1} x'}_{\varphi^{-1} x}}{\alpha_{x}}  - 
   \dfrac{\gamma_x}{\alpha_{x}}.
\end{align}
To derive the first term on the right hand side, we have also used the chain rule for our parameterization (Eq. \ref{eqn:chainRuleForPhix}) and the scalar field $\gamma$ introduced in the previous subsection. Now, in Lemma \ref{lem:convergenceOfg}, we show that starting from any bounded scalar function $h^0$, and considering any bounded function $r$, the following recursion converges uniformly,
\begin{align}
\label{eqn:recursionh}
    h^{n+1}_{\varphi x} = \dfrac{h^n_x}{\alpha_{\varphi x}} + r_{\varphi x}.
\end{align}
Since the sequence of functions $\left\{g^{\varphi^n x}\right\}_{n=-\infty}^0$ in Eq. \ref{eqn:densityRecursionWithUnknown} satisfies this same recursion, the family of functions $\left\{g^x\right\}$ must indeed be a single function independent of $x.$ Hence, we omit the superscript $x,$ and simply denote the density gradient $g.$ Thus, Eq. \ref{eqn:densityRecursionWithUnknown} can be rewritten as follows, fixing some $\mu$-typical reference orbit $\left\{x_n\right\}$:
\begin{align}
\label{eqn:densityGradientComputation}
    g_{n+1} = \dfrac{g_n}{\alpha_{n+1}} - \dfrac{\gamma_{n+1}}{\alpha_{n+1}}.
 \end{align}
This iterative equation exponentially converges (Lemma \ref{lem:convergenceOfg}), with $n$, to the true value of $g_n := g(x_n),$ along almost every orbit. In practice, we start the computation with an arbitrary initialization. 

\subsection{Iterative computation of the scalar field $b$}
\label{sec:formulaforb}
Having completely prescribed a computation for $g,$ the only unknown we are left with in computing the unstable contribution (Eq. \ref{eqn:unstableContributionRegularizedExpression}) is the scalar field $b.$
We describe a procedure, involving $w$, and a new second-order tangent equation. This latter second-order tangent solution exponentially approaches the vector field representing the unstable derivative of $v.$ In order to derive this equation, recall that the vector field $v$ is the limit of a sequence $\left\{v^n\right\}$ described in Eq. \ref{eqn:regularizedTangentEquation}. The sequence $\left\{v^n\right\}$ is differentiable in the unstable direction, if we start with a differentiable vector field $v^0.$ In Lemma \ref{lem:existenceAndUniquenessOfy}, we show that the sequence of these derivatives, denoted $\left\{y^n\right\},$ converges uniformly. Since $\left\{v^n\right\}$ also converges uniformly,
the limit $y$ of the sequence $\left\{y^n\right\}$ is the unstable derivative of $v.$ 

Thus, it is valid to differentiate the regularized tangent equation (Eq. \ref{eqn:stableTangentEquation}) in the unstable direction. Taking this derivative, we obtain that the vector field $y$ satisfies the following equation, along a fixed $\mu-$typical orbit $\left\{x_n\right\}_{n\in\mathbb{Z}^+}$ 
\begin{align}
\notag 
    \alpha_{n+1} y_{n+1} &= (d\varphi)_n y_n + (d^2\varphi)_n(q_n, v_n) \\
    \label{eqn:secondOrderStableTangentEquation}
    &+  \alpha_{n+1} (d\chi)_{n+1}\: q_{n+1} - \alpha_{n+1} b_{n+1} q_{n+1} -\alpha_{n+1} a_{n+1} w_{n+1}.
\end{align}
As usual, we have used the chain rule for differentiating on unstable manifolds (Eq. \ref{eqn:chainRuleForPhix}) and the differentiability of $d\varphi$ and $\chi$ on $M$. We have suppressed the parameterization centers on the superscripts because the unstable derivatives $y$ and $b$ are limits of uniformly converging series, and hence we can invoke the same argument as in sections \ref{sec:formulaforw} and \ref{sec:formulaforg} to show their independence from the parameterization. Using Lemma \ref{lem:existenceAndUniquenessOfy}, we can argue that the iteration above converges to the true value of the vector $y_n$ as $n\to\infty,$ starting with an arbitrary value for $y_0.$ In order to compute the above recursion, we must know the value of $b$ along the trajectory. However, it is possible to close this system of equations for $y_{n+1}$ and $b_{n+1}$ by differentiating in the unstable direction, the definition of $a_{n+1}$. 

Recall that the regularized tangent solutions $v_n$ are orthogonal to $q_n$
and $a_n$ are chosen so as to enforce this orthogonality. From Eq. \ref{eqn:stableTangentEquation}, we get,
\begin{align} 
    a_{n+1} = q_{n+1} \cdot \left( (d\varphi)_n v_n + \chi_{n+1} \right).
\end{align}
Differentiating the above equation in the unstable direction,
\begin{align}
\notag 
    \alpha_{n+1} b_{n+1} &= \left((d^2\varphi)_n(v_n, q_n) + (d\varphi)_n y_n  + \alpha_{n+1} (d\chi_{n+1}q_{n+1})\right)\cdot q_{n+1} \\
    \label{eqn:dbn}
    &+ \alpha_{n+1} w_{n+1} \cdot \left((d\varphi)_n v_n + \chi_{n+1} \right).
\end{align}
By comparing Eq. \ref{eqn:dbn} with the inner product of Eq. \ref{eqn:secondOrderStableTangentEquation} with $q_{n+1}$, we see that
\begin{align}
\label{eqn:constraint}
y_{n+1} \cdot q_{n+1} = - w_{n+1}\cdot ((d\varphi)_n v_n + \chi_{n+1}).      
\end{align}
 By using the orthogonality $w_{n+1}\cdot q_{n+1} = 0$, which follows from Eq. \ref{eqn:computationOfw}, on Eq. \ref{eqn:stableTangentEquation}, we obtain $w_{n+1} \cdot v_{n+1} = w_{n+1} \cdot ((d\varphi)_n v_n + \chi_{n+1})$. Thus, we can further simplify Eq. \ref{eqn:constraint} to obtain the constraint,
\begin{align}
\label{eqn:constraintSimple}
    y_{n+1} \cdot q_{n+1} = - w_{n+1}\cdot v_{n+1}.
\end{align}
Thus, we have two equations (Eq. \ref{eqn:secondOrderStableTangentEquation} and Eq. \ref{eqn:constraintSimple}) from which the two unknowns $b_{n+1}$ and $y_{n+1}$ are solved for. Without loss of generality, we start the iteration in Eq. \ref{eqn:secondOrderStableTangentEquation} with $y_0 = 0 \in \mathbb{R}^m.$ Then, by Lemma \ref{lem:existenceAndUniquenessOfy} and Lemma \ref{lem:convergenceOfb}, this iterative procedure yields exponentially accurate values of both $y_n$ and $b_n$ along a trajectory.  

\subsection{Computation of the unstable contribution}
\label{sec:unstableContributionComputation}
We now assemble the components computed in the previous subsections together to form the unstable contribution, which we recall from Eq. \ref{eqn:unstableContributionRegularizedExpression},
\begin{align}
\label{eqn:unstableContributionRecalled}
    \langle J, \partial_s\mu_s\rangle^{\rm u} = - \lim_{K\to\infty}\sum_{k=0}^{K-1} \langle J\circ\varphi^k \left(a^{K-k} g + b^{K-k} \right), \mu\rangle.
\end{align}
Recall that, by increasing the run-up time $K',$ the scalar fields $a^k$ and $b^k$, $k \geq 0,$ become arbitrarily close in the supremum norm, to $a$ and $b$ respectively. In addition, we shall assume a strong form of exponential correlation decay (section \ref{sec:srb}) by which, for a H\"older continuous field $l$ with $\langle l, \mu \rangle = 0$, there exists a $c_l > 0$ such that $|\langle (J\circ\varphi^n) \: l , \mu\rangle| \leq c_l \lambda^n \|l\|,$ for all $n.$ Under this assumption, with $l = (a^{K-k} - a) g + (b^{K-k} - b),$ 
\begin{align}
   \left|\sum_{k=0}^{K-1} \left(\langle J\circ\varphi^k \left(a^{K-k} g + b^{K-k} \right), \mu\rangle -  \langle J\circ\varphi^k \left(a g + b \right), \mu\rangle\right)\right| \leq \|g\|c_3\: c_1\: K\:\lambda^K + c_4\:c_2\:K\: \lambda^K, 
\end{align}
because there exist constants $c_1, c_2 > 0$ such that $\|a^{K-k} - a\|\leq c_1 \lambda^{K-k}$ and $\|b^{K-k} - b\|\leq c_2 \lambda^{K-k}$ according to Lemmas \ref{lem:convergenceOfa} and \ref{lem:convergenceOfb} respectively. Letting $K \to \infty,$ we conclude that the unstable contribution may be computed as 
\begin{align}
\label{eqn:approximateUnstableContribution}
     \langle J, \partial_s\mu_s\rangle^{\rm u} = - \sum_{k=0}^{\infty} \langle J\circ\varphi^k \left(a g + b \right), \mu\rangle.
\end{align}
In practice, we need only a small number $K$ of terms in the series above to approximate the unstable contribution to within a given tolerance. We fix some $N$-length $\mu-$typical orbit $\left\{x_n\right\}$ along which we compute the first $K$ terms above,
\begin{align}
\label{eqn:unstableContributionErgodicAverageRecalled}
     \langle J, \partial_s\mu_s\rangle^{\rm u} \approx - \sum_{k=0}^{K-1} \dfrac{1}{N} \sum_{n=0}^{N-1} J_{n+k} \left(a_n g_n + b_n \right).
\end{align}
This computation, due to decay of correlations, converges as $N\to \infty$ and $K\to \infty$, in that order. Further, the error in approximating each term of Eq. \ref{eqn:unstableContributionRecalled} as an $N$-time ergodic average declines as ${\cal O}(1/\sqrt{N}),$ up to a factor iterated logarithmic in $N.$ A complete discussion of the error convergence of the unstable contribution computation is deferred until section \ref{sec:unstableContributionProof}, after the proofs of convergence of the individual computational components of Eq. \ref{eqn:unstableContributionErgodicAverageRecalled}. Here we remark that even if the equivalence between Eq. \ref{eqn:approximateUnstableContribution} and the unstable contribution (Eq. \ref{eqn:unstableContributionRecalled}) does not strictly hold, the error in approximating the unstable contribution (Eq. \ref{eqn:unstableContributionRecalled}) by Eq. \ref{eqn:unstableContributionErgodicAverageRecalled} becomes negligible for a sufficiently long run-up time $K'$ (section \ref{sec:algorithm}). This is because, for large $K'$, $a^k, 0\leq k \leq K-1$ are all approximately equal to $a$ and $b^k, 0\leq k \leq K-1$ are all approximately equal to $b,$ due to Lemmas \ref{lem:convergenceOfa} and \ref{lem:convergenceOfb} respectively. 

Now we summarize these individual components that lead to the computation of the unstable contribution by Eq. \ref{eqn:unstableContributionErgodicAverageRecalled}. Let $\left\{x_n\right\}_{n=-K'}^{N-1}$ be a sufficiently long finite-length $\mu$-typical orbit. In order to compute Eq. \ref{eqn:unstableContributionErgodicAverageRecalled}, we need the values of $a$, $b$ and $g$ along the orbit. First, we note that the (approximate) values of $a$ along the orbit are known from the regularized tangent solution (Eq. \ref{eqn:stableTangentEquation}),
\begin{align}
    a_{n+1} = q_{n+1}^T ( (d\varphi)_n \: v_n + \chi_{n+1}).
\end{align}
In section \ref{sec:formulaforg}, we derived a recursive equation for the density gradient function, 
\begin{align}
\label{eqn:gcomputation}
    g_{n+1} = \dfrac{g_n}{\alpha_{n+1}} - \dfrac{\gamma_{n+1}}{\alpha_{n+1}},
\end{align}
which we begin by arbitrarily setting $g_{-K'}= 0.$ The scalar field $\gamma,$ which denotes the unstable derivative of the expansion factor $\alpha,$ is evaluated along an orbit by solving for the vector field $w$ (section \ref{sec:formulaforw}). Summarizing this step here, the recursive formula 
\begin{align}
\label{eqn:wcomputation}
    w_{n+1} = \dfrac{1}{\alpha_{n+1}^2}\left( I - q_{n+1} q_{n+1}^T\right) \left( (d\varphi)_n w_n + (d^2\varphi)_n(q_n, q_n)\right) 
\end{align}
is again started with the arbitrary choice $w_{-K'} \in \mathbb{R}^m$; at each step of the recursion, the values of $\gamma$ along an orbit are set to
\begin{align}
\label{eqn:gammacomputation}
    \gamma_{n+1} = \dfrac{q_{n+1}^T}{\alpha_{n+1}} \left( (d\varphi)_n w_n + (d^2\varphi)_n(q_n, q_n)\right).
\end{align}
Finally, the values $b_n$ are obtained simultaneously with $y_n$ from Eq. \ref{eqn:secondOrderStableTangentEquation} and Eq. \ref{eqn:constraintSimple}, which are repeated here for clarity:
\begin{align}
\notag 
    y_{n+1} &= \dfrac{(d\varphi)_n y_n + (d^2\varphi)_n(q_n, v_n)}{ \alpha_{n+1}} \\
    \label{eqn:secondOrderStableTangentEquationRederived}
    &+  (d\chi)_{n+1}\: q_{n+1} - b_{n+1} q_{n+1} - a_{n+1} w_{n+1},
\end{align}
where $b_{n+1}$ is such that $y_{n+1}\cdot q_{n+1} = - v_{n+1}\cdot w_{n+1},$ at each $n.$ Using the values of $a_n,$ $b_n$ and $g_n$ obtained approximately, as outlined above, we compute an $N$-sample average of the $k$-lag correlation with $J$, in order to compute the $k$th term of the unstable contribution (Eq. \ref{eqn:unstableContributionErgodicAverageRecalled}). Putting this together with the stable contribution (Eq. \ref{eqn:stableContributionErgodicAverage}), we obtain the overall sensitivity. These steps are condensed into the S3 algorithm listed in section \ref{sec:algorithm}, with one simplification: we can eliminate the need to compute Eq. \ref{eqn:gcomputation} explicitly. To see this, let us consider Eq. \ref{eqn:wcomputation} without the projection,
\begin{align}
\label{eqn:pcomputation}
    p_{n+1} &= \dfrac{1}{\alpha_{n+1}^2}\left( (d\varphi)_n p_n + (d^2\varphi)_n(q_n, q_n)\right) \\
    \notag 
    p_{-K'} &= 0 \in \mathbb{R}^m.
\end{align}
Using these solutions $p_n$, Eq. \ref{eqn:gammacomputation} and Eq. \ref{eqn:gcomputation}, we find that $$a_{n+1} w_{n+1} = a_{n+1} g_{n+1}\: q_{n+1} + a_{n+1} p_{n+1}.$$ Substituting this relationship into Eq. \ref{eqn:secondOrderStableTangentEquationRederived} leads to 
\begin{align}
    \notag 
    y_{n+1} &= \dfrac{(d\varphi)_n y_n + (d^2\varphi)_n(q_n, v_n)}{ \alpha_{n+1}} +  (d\chi)_{n+1}\: q_{n+1}\\
    \label{eqn:secondOrderStableTangentEquationFinal}
    &- c_{n+1} q_{n+1} - a_{n+1} p_{n+1},
\end{align}
where $c_{n+1} := (a_{n+1} g_{n+1} + b_{n+1}).$ Notice that we only need the scalar field $c$, and not $b$ and $g$ explicitly, to compute the unstable contribution and since by iterating Eq. \ref{eqn:secondOrderStableTangentEquationFinal} we obtain $c_n$ directly, we avoid computing $g$ via Eq. \ref{eqn:gcomputation}. Since the above equation replaces the use of $w_n$ with $p_n,$ we do not need to solve for $w_n$ through Eq. \ref{eqn:wcomputation} either; rather, we just compute $p_n$ by iterating Eq. \ref{eqn:pcomputation}. Although we used $w_n$ in the constraint (Eq. \ref{eqn:constraint}) needed to solve for $b_n$, and now $c_n$, the constraint may be rewritten using $p_n$ as, 
\begin{align}
    \label{eqn:constraintpn}
    y_{n+1} \cdot q_{n+1} = - v_{n+1} \cdot w_{n+1} = - v_{n+1} \cdot p_{n+1}.
\end{align}
The second equality is true because $p_{n+1}$ is of the form $w_{n+1} + \delta_{n+1} q_{n+1},$ for some scalar sequence $\left\{\delta_{n+1}\right\},$ and $v_{n+1} \cdot q_{n+1} = 0$ by construction (Eq. \ref{eqn:stableTangentEquation}).
\section{Numerical results}
\label{sec:numerics}
\subsection{Perturbations of the Baker's map}
\label{sec:baker}
In order to validate the S3 algorithm, we consider perturbations of the standard Baker's map that are designed to elicit both stable and unstable contributions. Consider the following self-map $\varphi_s$ of the torus $\mathbb{T}^2$, where $s = [s_1, s_2, s_3, s_4]^T \in \mathbb{R}^4,$ 
\begin{align}
\label{eqn:perturbedBakerMap}
    \varphi_s([x^{(1)}, x^{(2)}]^T) = 
    \begin{bmatrix}
    2 x^{(1)} +  (s_1 + s_2 \sin(2x^{(2)})/2) \sin x^{(1)} - \lfloor x^{(1)}/\pi\rfloor 2\pi \\
    \dfrac{x^{(2)} + (s_4 + s_3 \sin x^{(1)}) \sin(2x^{(2)}) + \lfloor x^{(1)}/\pi\rfloor \pi}{2}
    \end{bmatrix}\; {\rm mod }\; 2\pi. 
\end{align}
The standard Baker's map is recovered at $s = 0 \in \mathbb{R}^4.$ We isolate the effect of each parameter by illustrating the action of perturbed maps with all except one parameter set to 0. Figure \ref{fig:bakersmaps1s2} elucidates the effect of $s_1$ and $s_2.$ The unperturbed Baker's map, as shown in Figure \ref{fig:bakersmaps1s2} (top right), uniformly expands in the horizontal direction ($\hat{x}^{(1)}$) and contracts in the vertical direction ($\hat{x}^{(2)}$). By contrast, when the parameter $s_1$ is non-zero (and the other parameters are set to 0), the expansion in the $\hat{x}^{(1)}$ direction depends on $x^{(1)},$ resulting in a non-uniformly expanded grid as shown on the bottom left of Figure \ref{fig:bakersmaps1s2}. On the bottom right of Figure \ref{fig:bakersmaps1s2}, we show the effect of the parameter $s_2,$ from which it is clear that the expansion in the $\hat{x}^{(1)}$ direction depends nonlinearly on the $x^{(2)}$ coordinate.

In Figure \ref{fig:bakersmaps3s4} (top right), the expansion and contraction by constant factors, in the $\hat{x}^{(1)}$ and $\hat{x}^{(2)}$ directions respectively, by the unperturbed map, are clearly seen from the uniform stretching and contraction of the horizontal strips on the top left. On the bottom row, the effect of the parameters $s_3$ and $s_4$, acting in isolation, are shown on the left and right respectively. From Eq. \ref{eqn:perturbedBakerMap}, it is clear that $s_3$ introduces a nonlinear variation with $x^{(1)}$ in the contraction along the $\hat{x}^{(2)}$ direction. This is indeed reflected in the image of the horizontal strips (top left of Figure \ref{fig:bakersmaps3s4}) under the perturbed map, which is shown on the bottom left of Figure \ref{fig:bakersmaps3s4}. Finally, the bottom right plot shows that the contraction in $\hat{x}^{(2)}$ is nonuniform with respect to the ${x}^{(2)}$ coordinate but uniform along the $x^{(1)}$ coordinate, as we would expect from Eq. \ref{eqn:perturbedBakerMap} with $s_4$ being the only non-zero parameter.
\begin{figure}
\centering
   \includegraphics[width=0.3\textwidth]{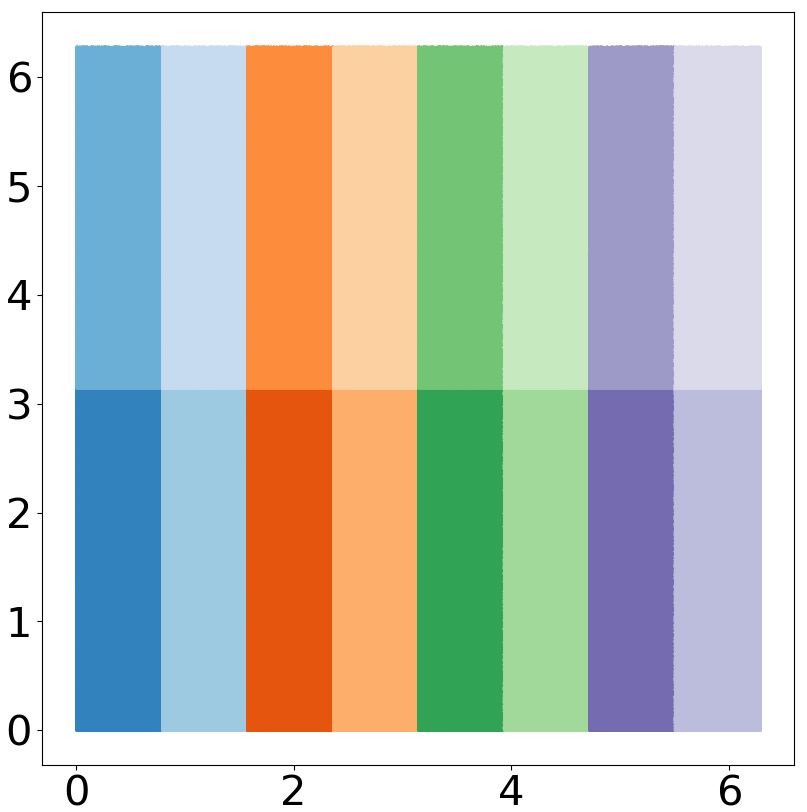}
  \includegraphics[width=0.3\textwidth]{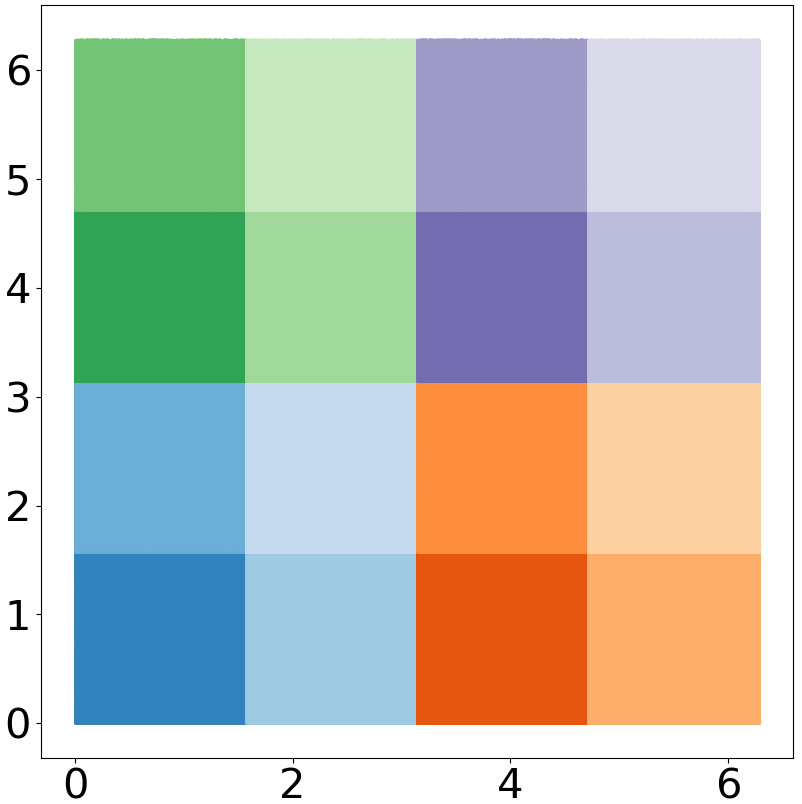}\\
  \includegraphics[width=0.3\textwidth]{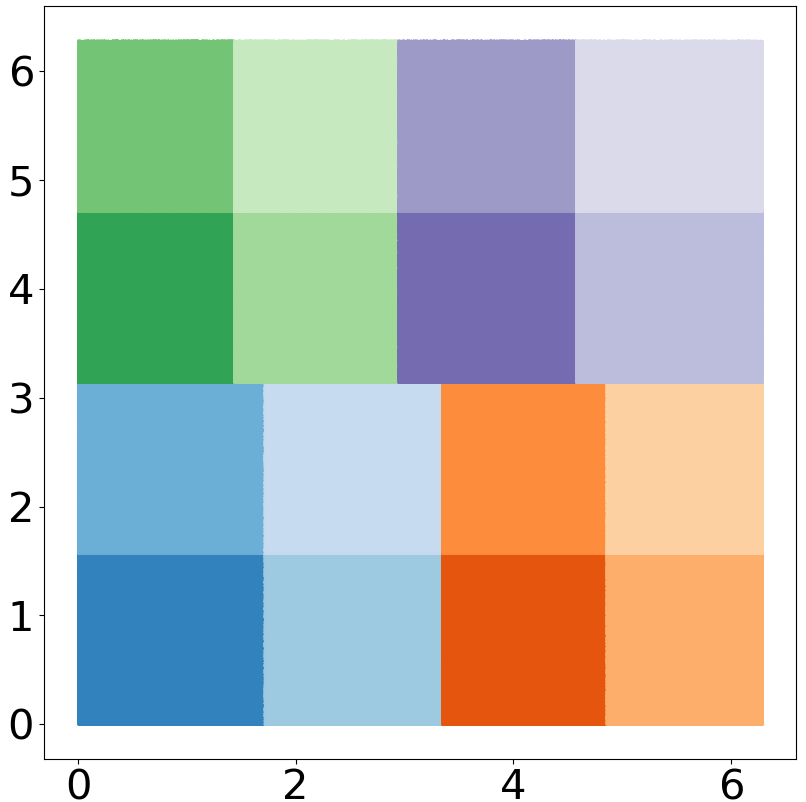}
  \includegraphics[width=0.3\textwidth]{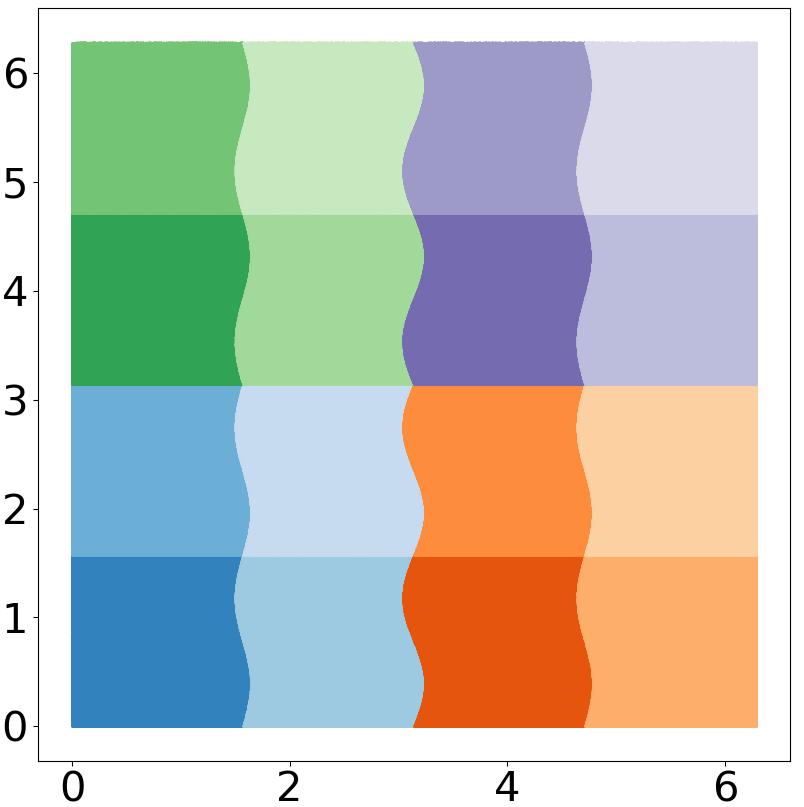}
    \caption{Top left: the domain $\mathbb{T}^2$ covered by rectangles. The other figures show the application of Baker's map at $s = 0 \in \mathbb{R}^4$ (top right), $s = [0.2,0,0,0]^T$ (bottom left) and $s = [0,0.2,0,0]$ (bottom right) on the gridded top left figure.}
    \label{fig:bakersmaps1s2}
\end{figure}
\begin{figure}
\centering
   \includegraphics[width=0.3\textwidth]{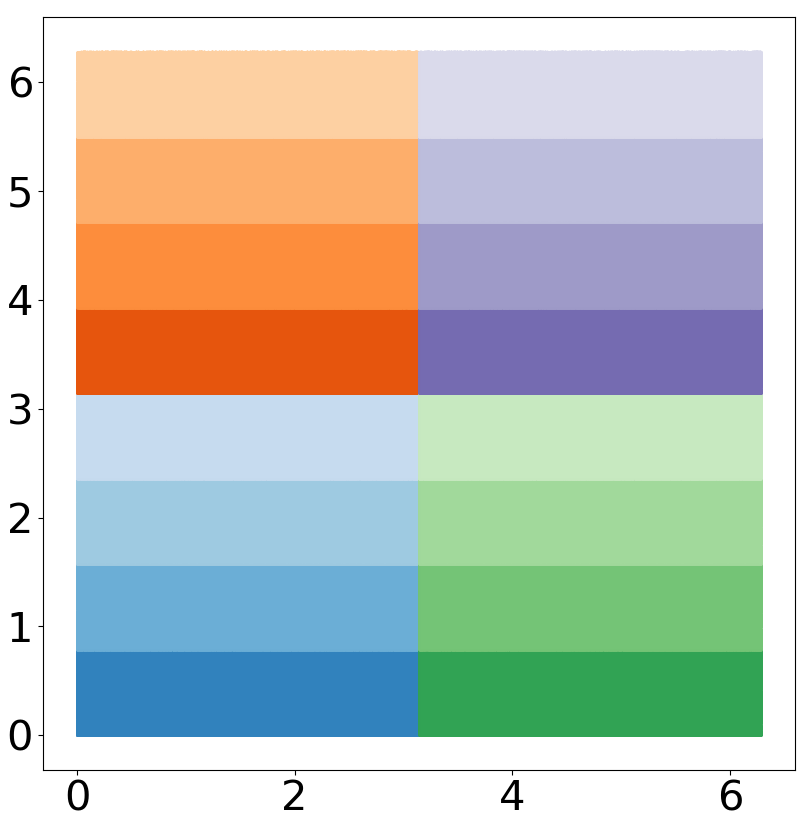}
  \includegraphics[width=0.3\textwidth]{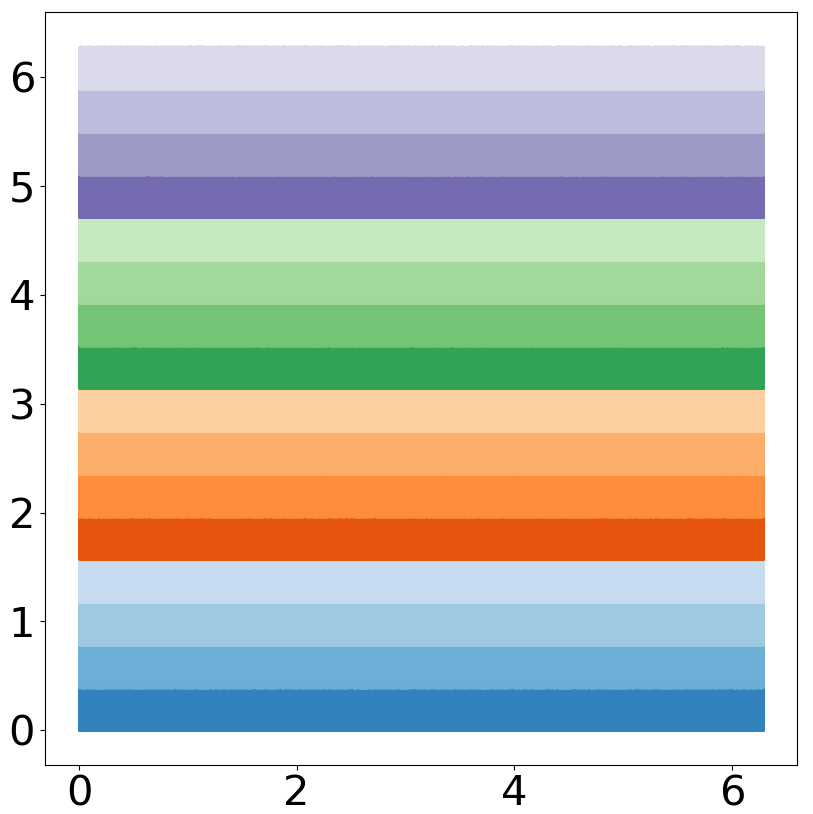} \\
  \includegraphics[width=0.3\textwidth]{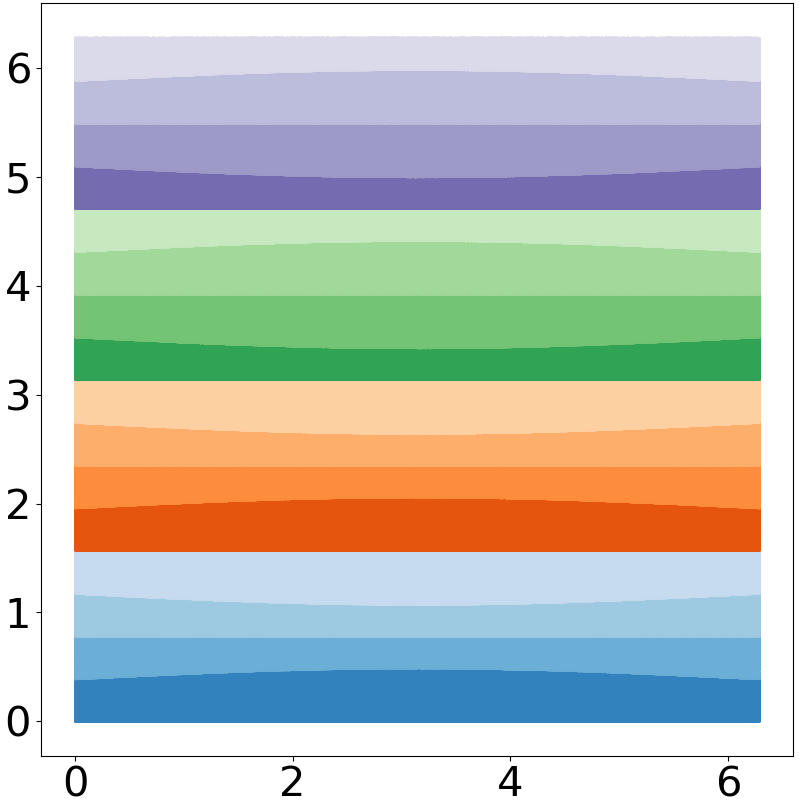}
  \includegraphics[width=0.3\textwidth]{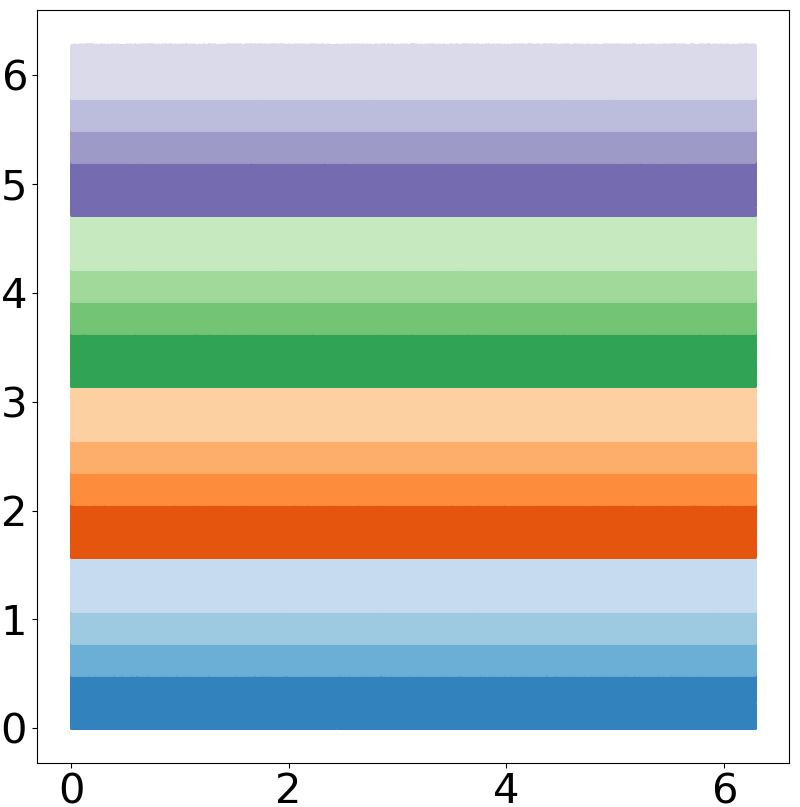}
    \caption{Top left: the domain $\mathbb{T}^2$ covered by rectangles. The other figures show the application of Baker's map at $s = 0 \in \mathbb{R}^4$ (top right), $s = [0,0,0.2,0]^T$ (bottom left) and $s = [0,0,0,0.2]$ (bottom right) on the gridded top left figure.}
    \label{fig:bakersmaps3s4}
\end{figure}
\subsubsection{Stable and unstable subspaces}
\label{sec:stableUnstableSubspacesOfBakersMap}
The standard Baker's map has uniform stable and unstable subspaces, which are aligned with $\hat{x}^{(2)}$ and $\hat{x}^{(1)}$ respectively, i.e., $E^u_x \equiv \hat{x}^{(1)} = {\rm span}\left\{[1,0]^T\right\}$ and $E^s_x \equiv \hat{x}^{(2)} \equiv {\rm span}\left\{[0,1]^T\right\}$ at all $x \in M$. The perturbations of $s_1$ and $s_4$ do not alter these uniform stable and unstable directions. It can be verified that the perturbation of $s_2$ alters the stable direction while retaining $\hat{x}^{(1)}$ as the unstable direction everywhere. The perturbed map with $s_3$ being the only non-zero parameter has a non-uniform unstable subspace that is not parallel to $\hat{x}^{(1)}$ everywhere. On the other hand, its stable subspaces are everywhere parallel to $\hat{x}^{(2)}.$    
\subsubsection{SRB measures of perturbed Baker's maps}
The perturbations of the Baker's map and the original map are all uniformly hyperbolic. Numerical approximations of the SRB measures of the perturbed maps associated to each parameter acting in isolation are shown in Figure \ref{fig:srbBaker}. Specifically, we plot, on a 400$\times$729 grid of $\mathbb{T}^2$, the empirical distributions computed using a long orbit (of length 1.6 trillion), say $\left\{x_n\right\}_{-K\leq n\leq N},$ where $x_0$ can be assumed to sample the SRB measure, for a large enough $K$. At each grid cell $A \subset (0,2\pi)^2$ we calculate the empirical probability, $\mu_{\rm emp}^N(A)=(1/N)\sum_{n=0}^{N-1} 1_A(x_n),$ where $1_A$, defined as $1_A(x) = 1$ when $x \in A$ and 0 otherwise, is the indicator function on the set $A.$ Since the SRB measure is physical, $\mu_{\rm emp}^N(A) \to \mu(A)$ as $N \to \infty.$ The colorbar in Figure \ref{fig:srbBaker} indicates the value $\mu_{\rm emp}^N(A)/{\rm mean}(\mu_{\rm emp}^N),$ where ${\rm mean}(\mu_{\rm emp}^N)$ is the sample mean over all grid cells. We verify that upon increasing the number of grid cells, $\mu_{\rm emp}^N(A)/({\rm area}(A))$ also increases. This numerical observation supports the lack of absolute continuity of $\mu$ with respect to Lebesgue. In particular, we observe that the empirical distributions are supported on Cantor-like sets in the 
$\hat{x}^{(2)}$ direction, which is the stable direction in all cases except the $s_2$ perturbation. That is, when the number of grid cells along $\hat{x}^{(2)}$ is $3^k$ we observe that about $2^k$ of them are not visited by the orbit.

To summarize, while the SRB measure of the unperturbed Baker's map ($s = 0 \in \mathbb{R}^4$) is the Lebesgue measure on $\mathbb{T}^2,$ the perturbed maps have SRB measures that are not absolutely continuous with respect to Lebesgue on $\mathbb{T}^2.$ This is visually observed in Figure \ref{fig:srbBaker}, where, in every case, the distribution appears to have fractal support in the $\hat{x}^{(2)}$ direction. 
\begin{figure}
\centering
    \includegraphics[width=0.45\textwidth]{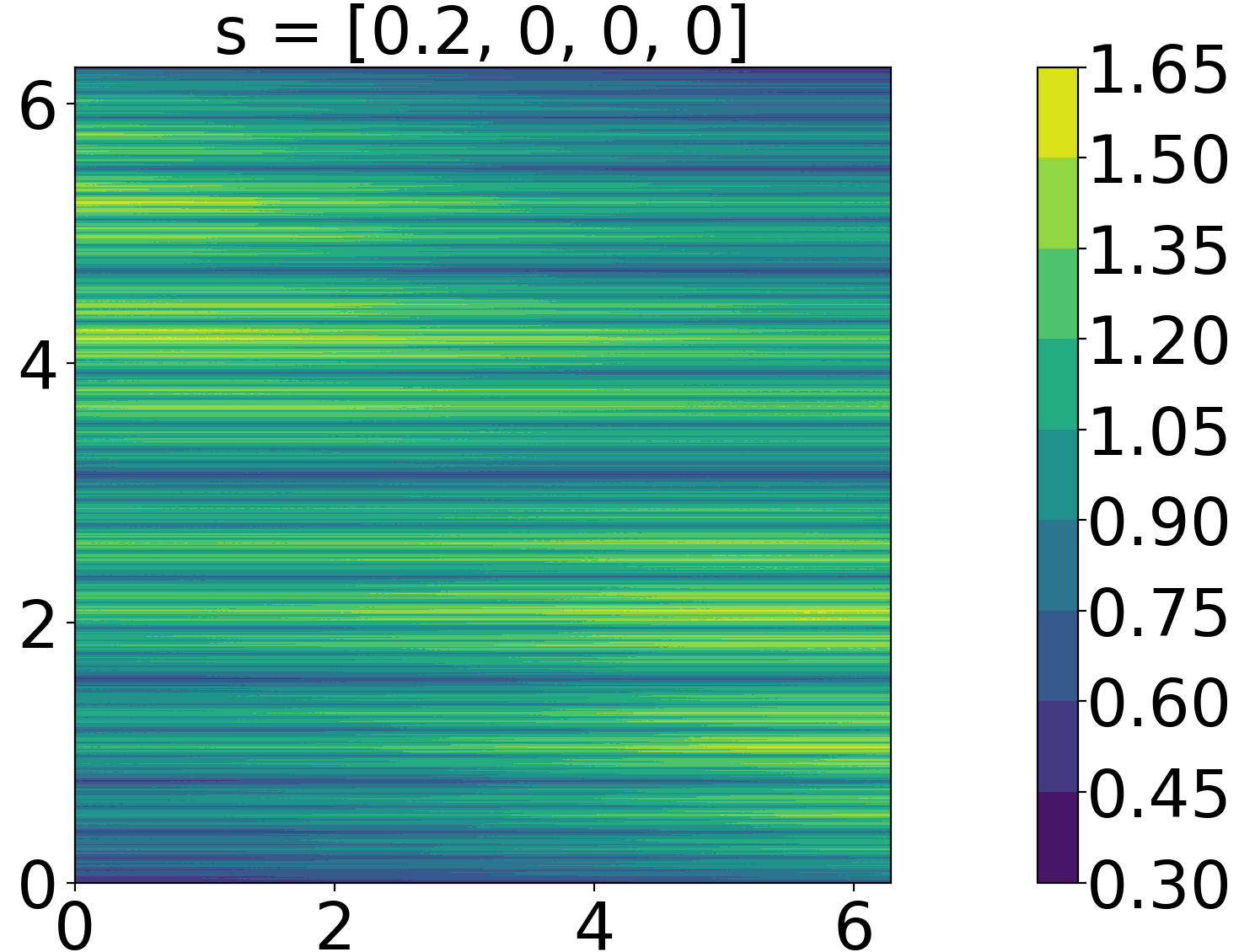}
    \includegraphics[width=0.45\textwidth]{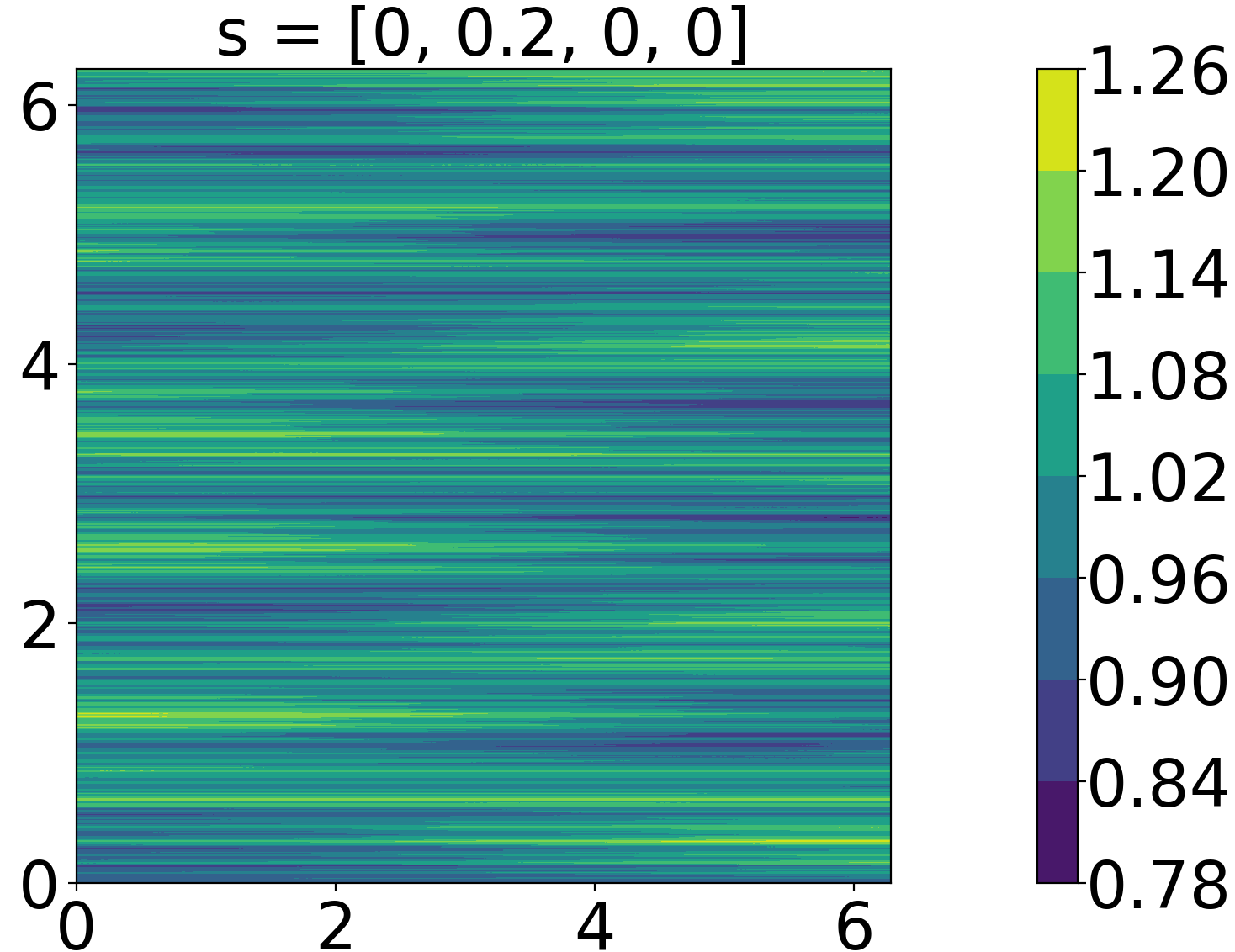}\\
     \includegraphics[width=0.45\textwidth]{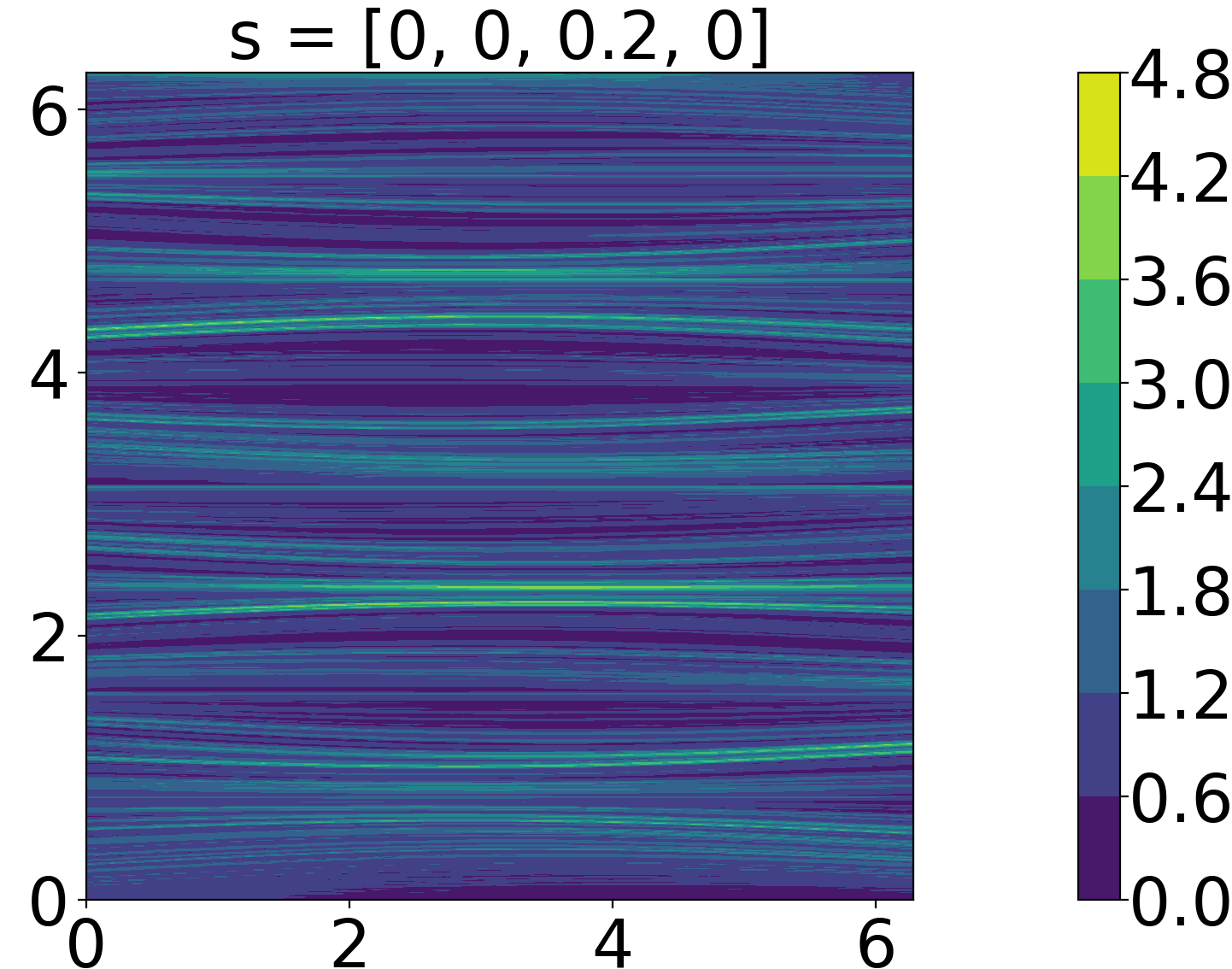}
   \includegraphics[width=0.45\textwidth]{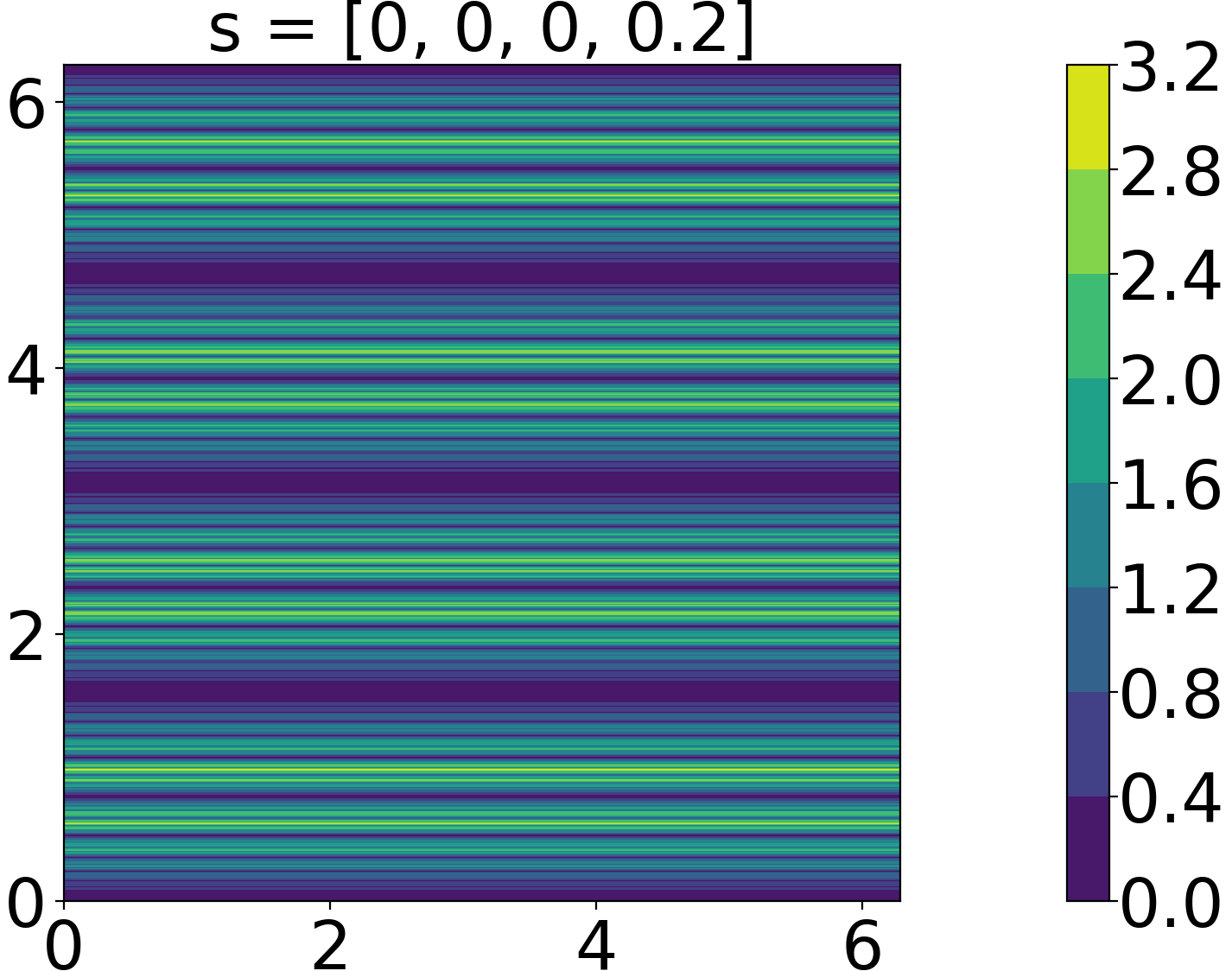}
\caption{Each plot shows a scaled SRB distribution achieved at the parameter values indicated on the title. A histogram of an orbit of length of 1.6 trillion is computed on a 400$\times$729 grid. The histogram values are scaled by their mean over all grid cells.}
\label{fig:srbBaker}
\end{figure}
The SRB distribution in each case is absolutely continuous on the unstable manifold, which is uniform and parallel to $\hat{x}^{(1)}$ in every case except the $s_3$ perturbation. In the perturbed map with $s_3$ being non-zero (bottom left of Figure \ref{fig:srbBaker}), the SRB measure appears smooth along its curved unstable manifold, and rough along its stable manifold, which is uniform and parallel to the $\hat{x}^{(2)}$ direction. As noted in section \ref{sec:stableUnstableSubspacesOfBakersMap}, while the unstable manifold of the perturbed map with $s_2 \neq 0$ is parallel to $\hat{x}^{(1)}$ everywhere, the stable manifold is curved around the $\hat{x}^{(2)}$ direction. This picture is consistent with smoothness of the map's SRB measure (top right of Figure \ref{fig:srbBaker}) in the $\hat{x}^{(1)}$ direction and the lack of smoothness in the vertical direction, which appears distinct from the $s_4$ case where the stable manifold is uniformly parallel to the $\hat{x}^{(2)}$ direction.  
\subsubsection{Sensitivities of a smooth objective function}
We make an arbitrary choice of a smooth objective function, $J([x^{(1)}, x^{(2)}]^T) = \cos{4x^{(2)}},$ and validate the parametric derivatives of its statistics computed by S3 in each perturbed Baker's map. The validation results are shown in Figures \ref{fig:stableUnstableSens} and \ref{fig:s3sens}.
\begin{figure}
\centering
    \includegraphics[width=0.49\textwidth]{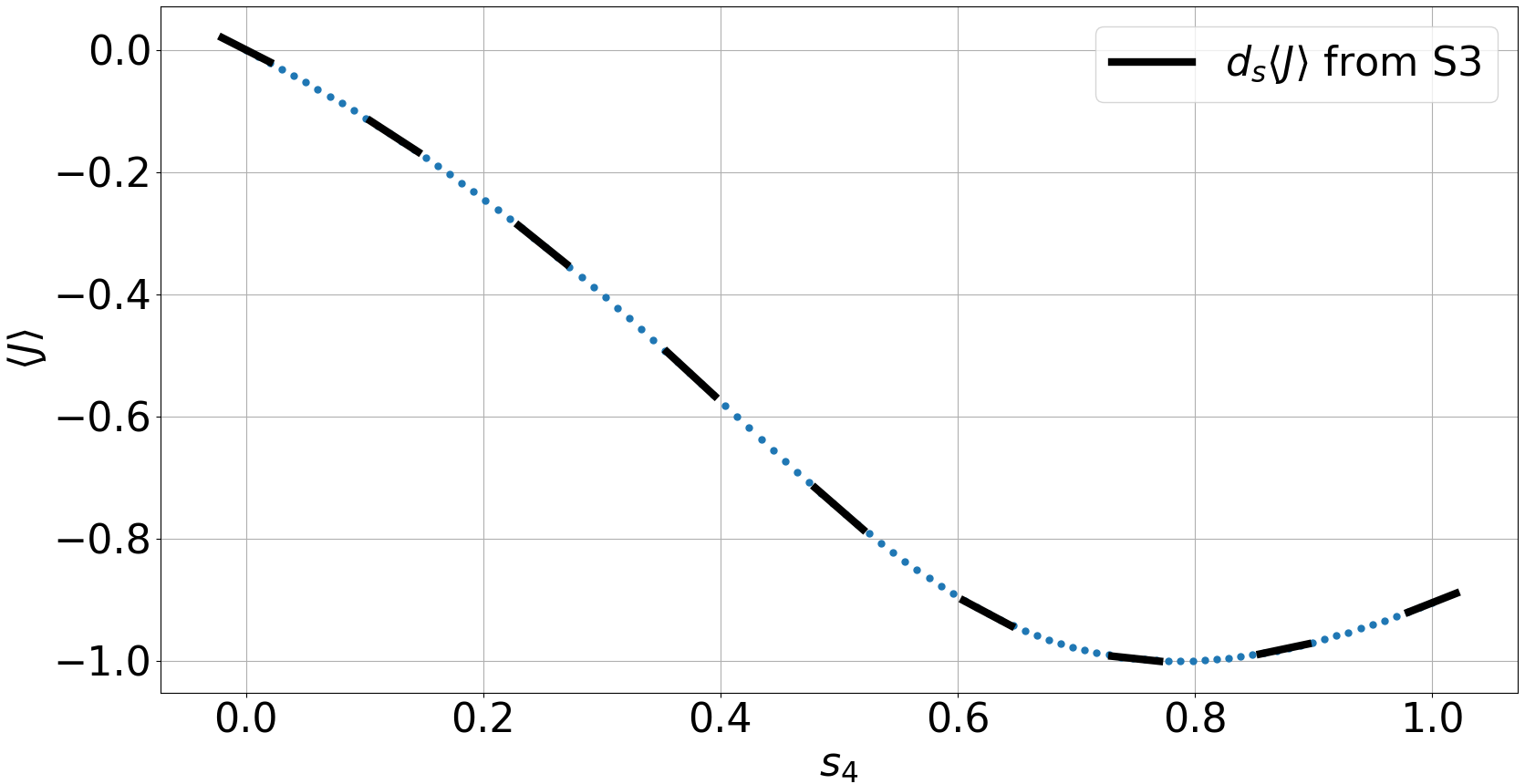}
    \includegraphics[width=0.49\textwidth]{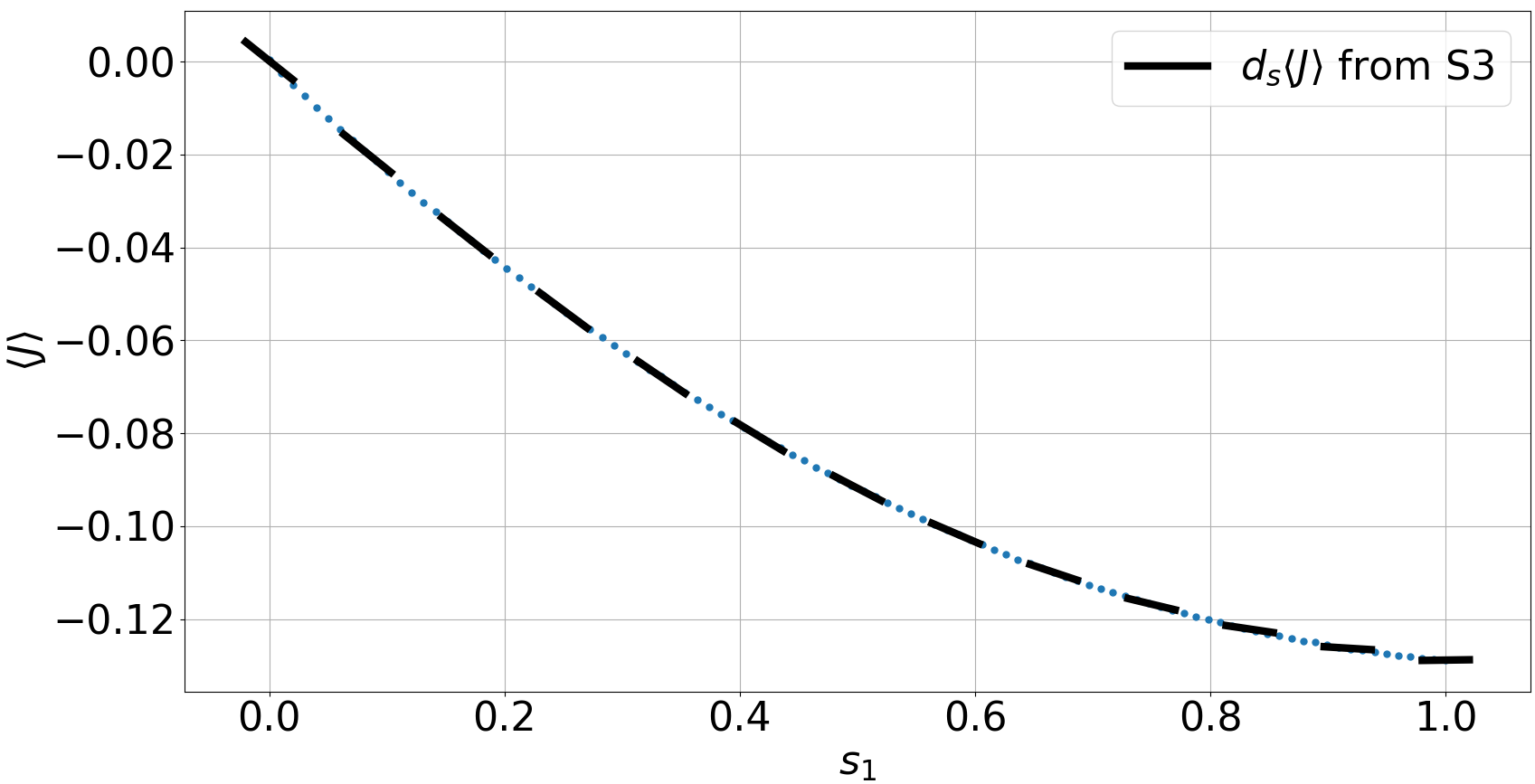}
    \caption{Ergodic average of the objective function $J = \cos(4x_2)$ as a function of $s_4$ (left) and $s_1$ (right); the other parameters are set to 0. Sensitivities from S3 are shown in black at select parameter values.}
    \label{fig:stableUnstableSens}
\end{figure}
In the left plot of Figure \ref{fig:stableUnstableSens}, only perturbations of $s_1$ are considered, while other parameters are held at 0. This leads to only an unstable contribution to the sensitivity because $\chi_{\varphi_s x} = [\sin{x^{(1)}}\;\;0]^T$ which is aligned with the unstable direction $\hat{x}^{(1)} = [1\:\: 0]^T$ at all $x.$ The ergodic/ensemble averages of $J$ are shown as a function of $s_1$ in blue; superimposed as black lines are the linearly extrapolated values of the S3 sensitivities computed at many different values of $s_1.$ We see that the derivatives of the $\langle J\rangle$ vs. $s_1$ are closely approximated by S3. This validates the computation of sensitivities to unstable perturbations by S3.
\begin{figure}
    \centering
    \includegraphics[width=0.49\textwidth]{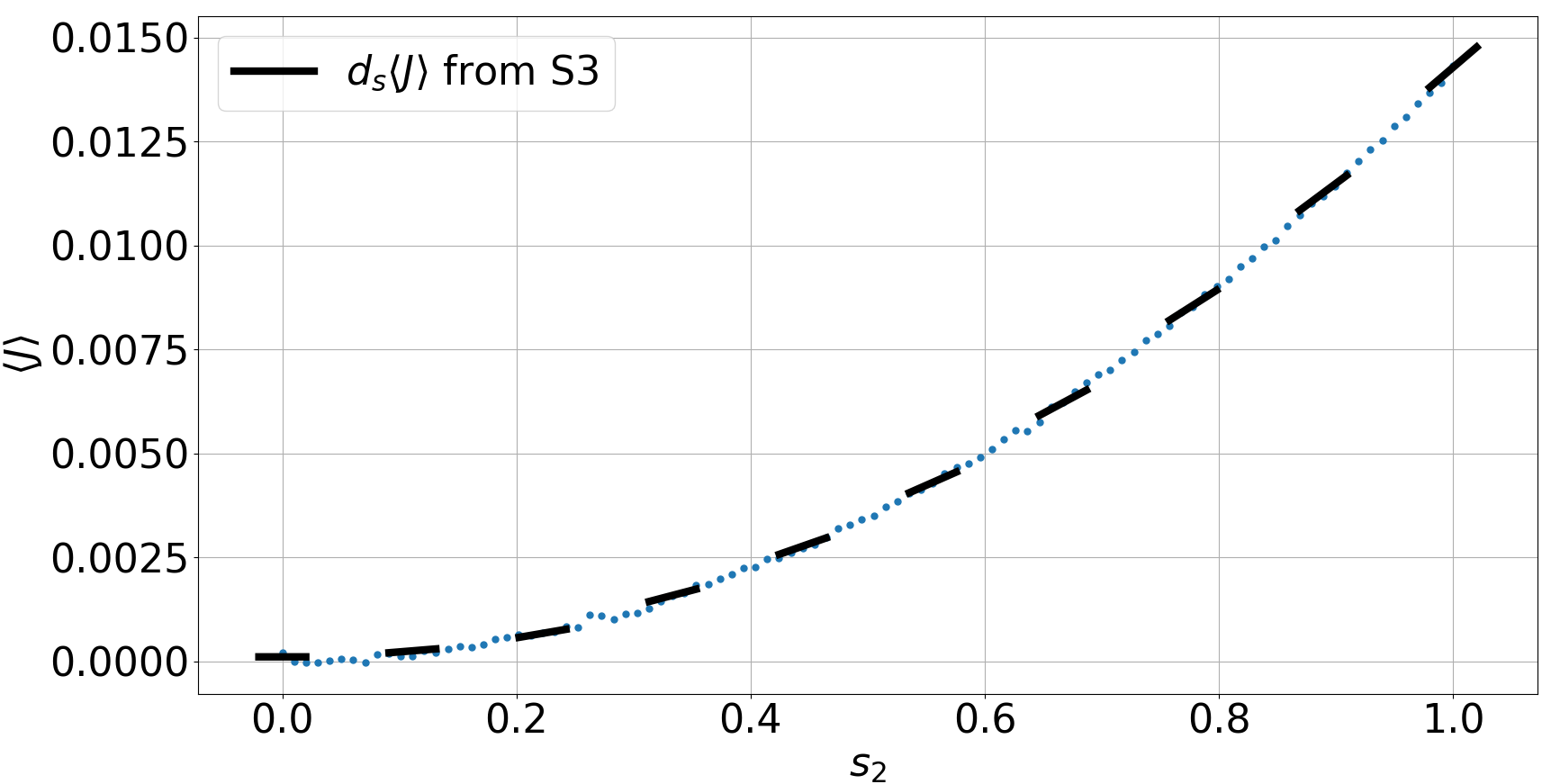}
    \includegraphics[width=0.49\textwidth]{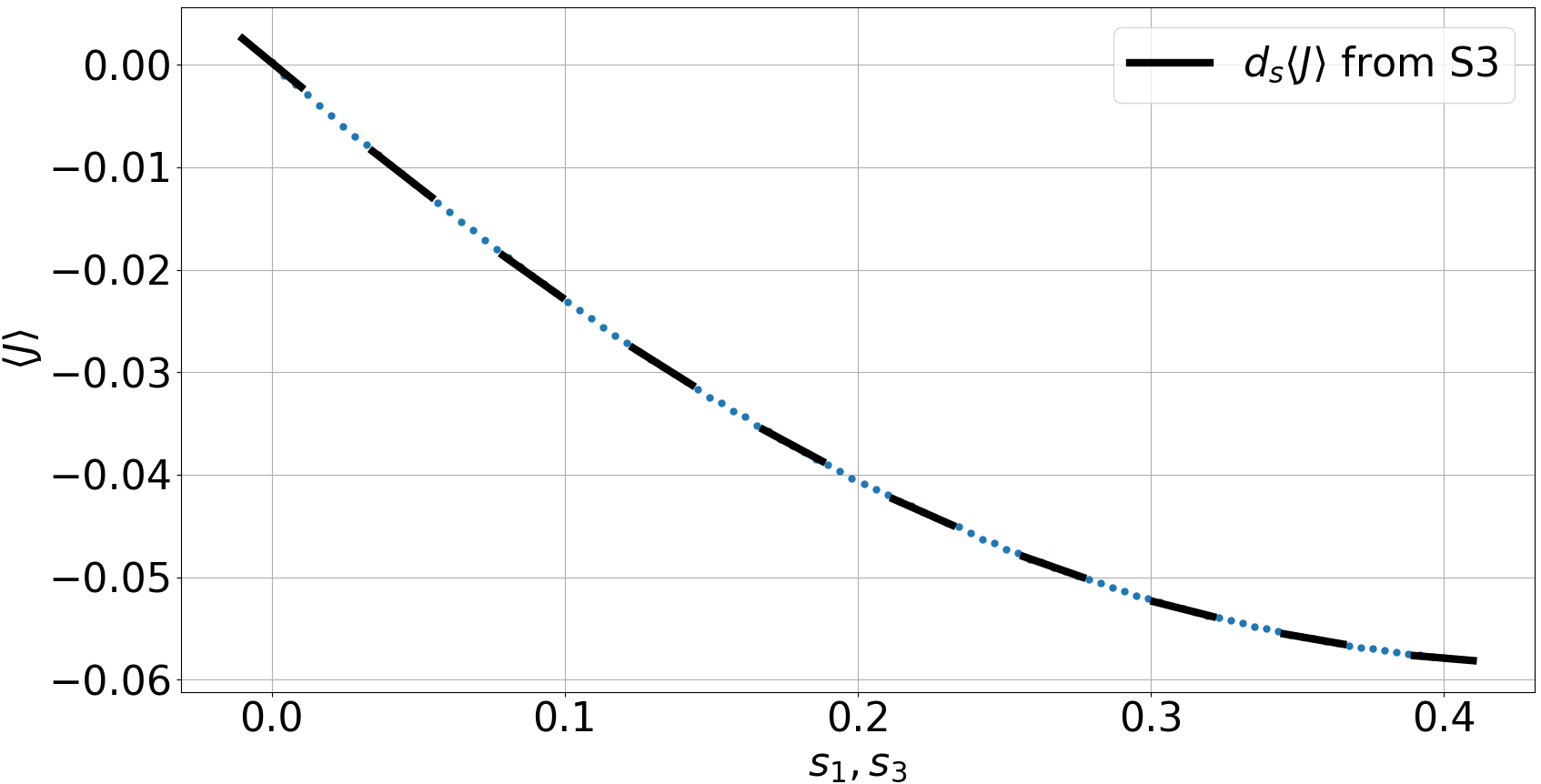}
    \caption{Ergodic average of the objective function $J = \cos(4x_2)$ as a function of $s_4$ (left) and $s_1 = s_3$ (right) for the Baker's maps in section \ref{sec:baker}; the other parameters are set to 0. Sensitivities from S3 are shown in black at select parameter values. }
    \label{fig:s3sens}
\end{figure}

Similarly, we show the validation of the S3 algorithm to stable perturbations, on the right plot of Figure \ref{fig:stableUnstableSens}. To generate these plots, ergodic/ensemble averages of $J$ are computed at perturbed maps where only $s_4$ is varied. In this case, $\chi$ is aligned with $\hat{x}^{(2)}$, which is the uniform stable direction. The S3 sensitivities, shown in black, are perfectly tangent to the response curves $\langle J\rangle$ vs. $s_4,$ as shown in Figure \ref{fig:stableUnstableSens} (right). 

To test S3 on maps with non-uniform stable and unstable directions, we apply the algorithm to compute linear response in maps with non-zero $s_2$ and $s_3$. At a non-zero $s_2$ parameter, when all other parameters are held at 0, the stable manifold is non-uniform, while the unstable manifold is uniform and tangent to $\hat{x}^{(1)}.$ The overall sensitivity contains only an unstable contribution, which is verified to be correct in Figure \ref{fig:s3sens} (left). As a final test, we put $s_1 = s_3$ and vary this parameter, which makes the unstable manifold curved and nonuniform. The response curves are shown (in blue) in Figure \ref{fig:s3sens}, on which the S3 sensitivities (black), which have both non-zero stable and unstable contributions, are plotted. Once again, the S3 derivatives are accurate over a range of parameter values. In each case, we chose $K = 11$ (terms summed in Ruelle's series) and $N = 500000$ (samples to compute ergodic averages), in the S3 algorithm; the blue points in each figure were ensemble averages over 160 million samples.
\subsection{Perturbations of the Solenoid attractor}
\label{sec:solenoid}
\begin{figure}
    \centering
    \includegraphics[width=0.32\textwidth]{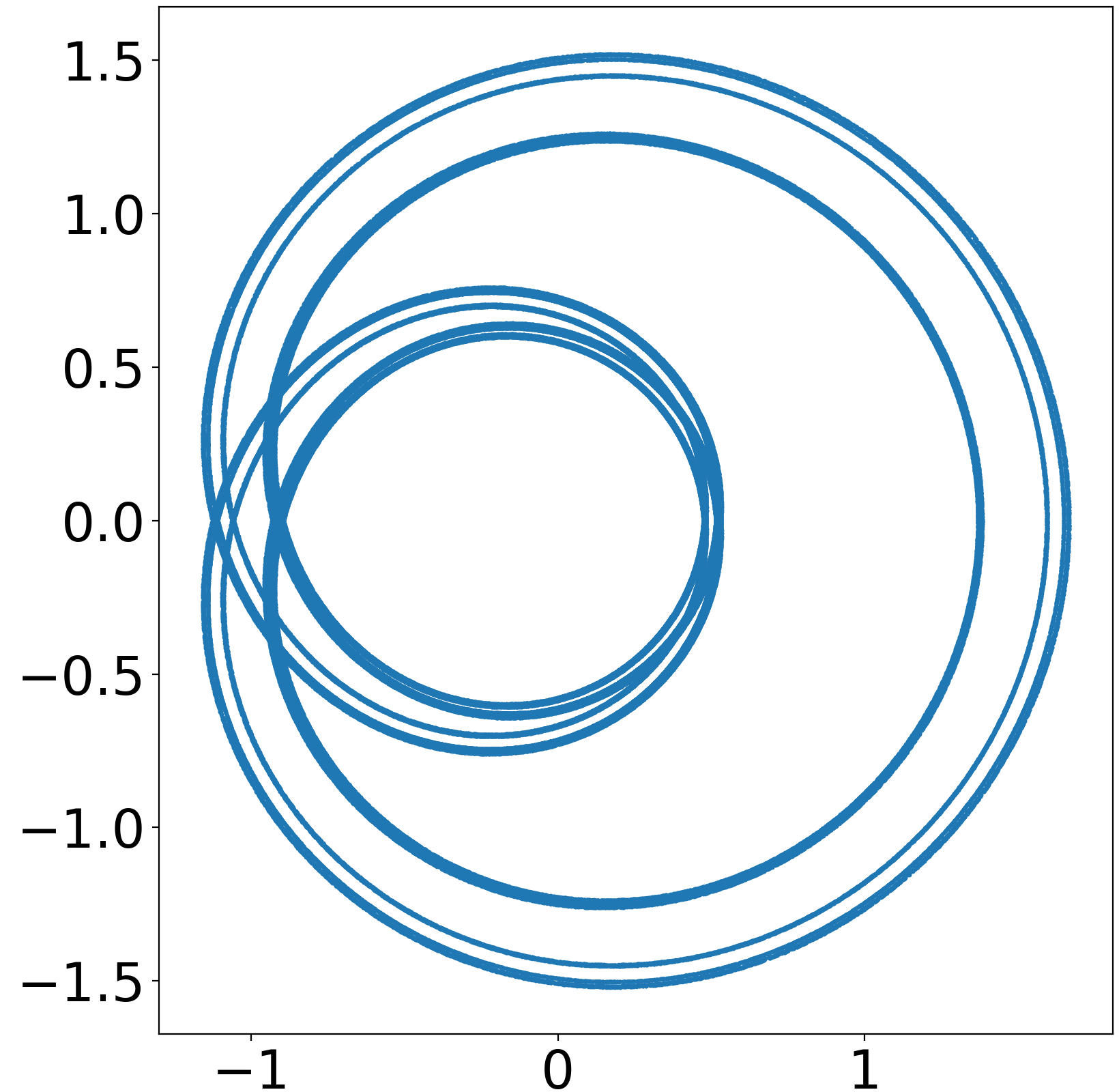}
    \includegraphics[width=0.32\textwidth]{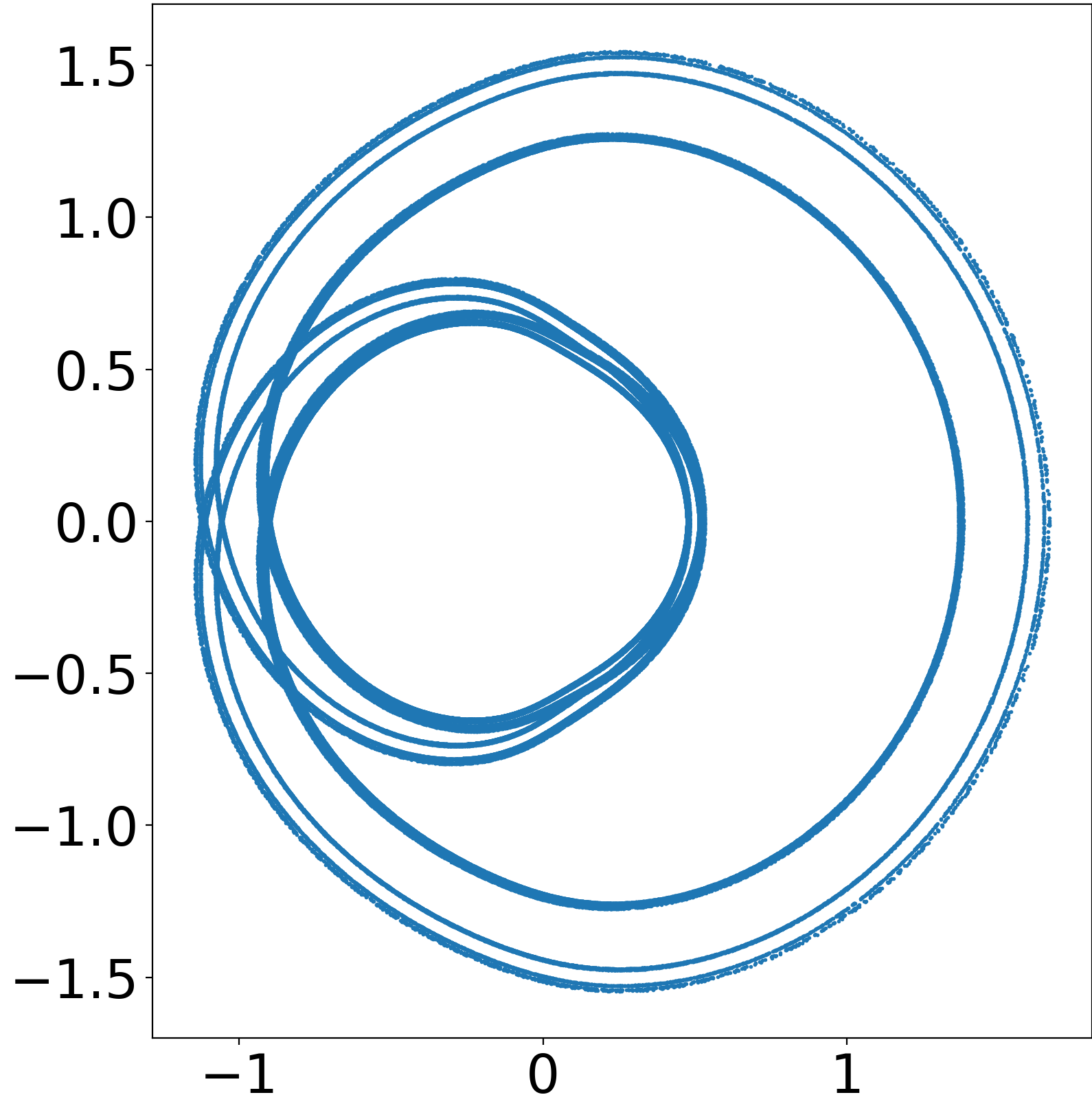}
    \includegraphics[width=0.32\textwidth]{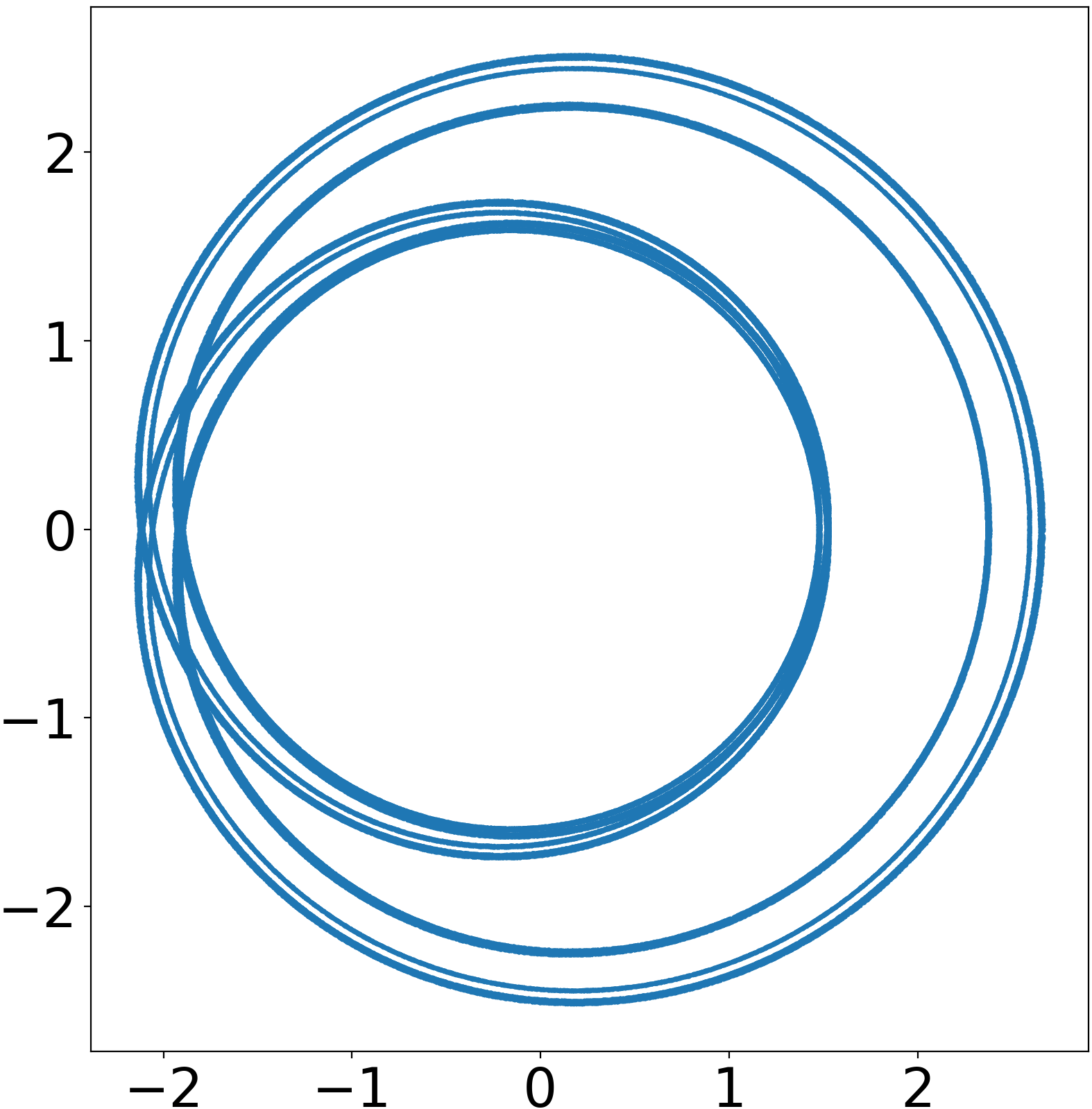}
    \caption{The projection of the Solenoid attractor on the  $x^{(1)}$-$x^{(2)}$ plane at $s = [1, 0]^T$ (left), $s = [1,1]^T$ (center) and $s = [2,0]^T$ (right) respectively.}
    \label{fig:solenoidAttractor}
\end{figure}
\begin{figure}
    \includegraphics[width=0.49\textwidth]{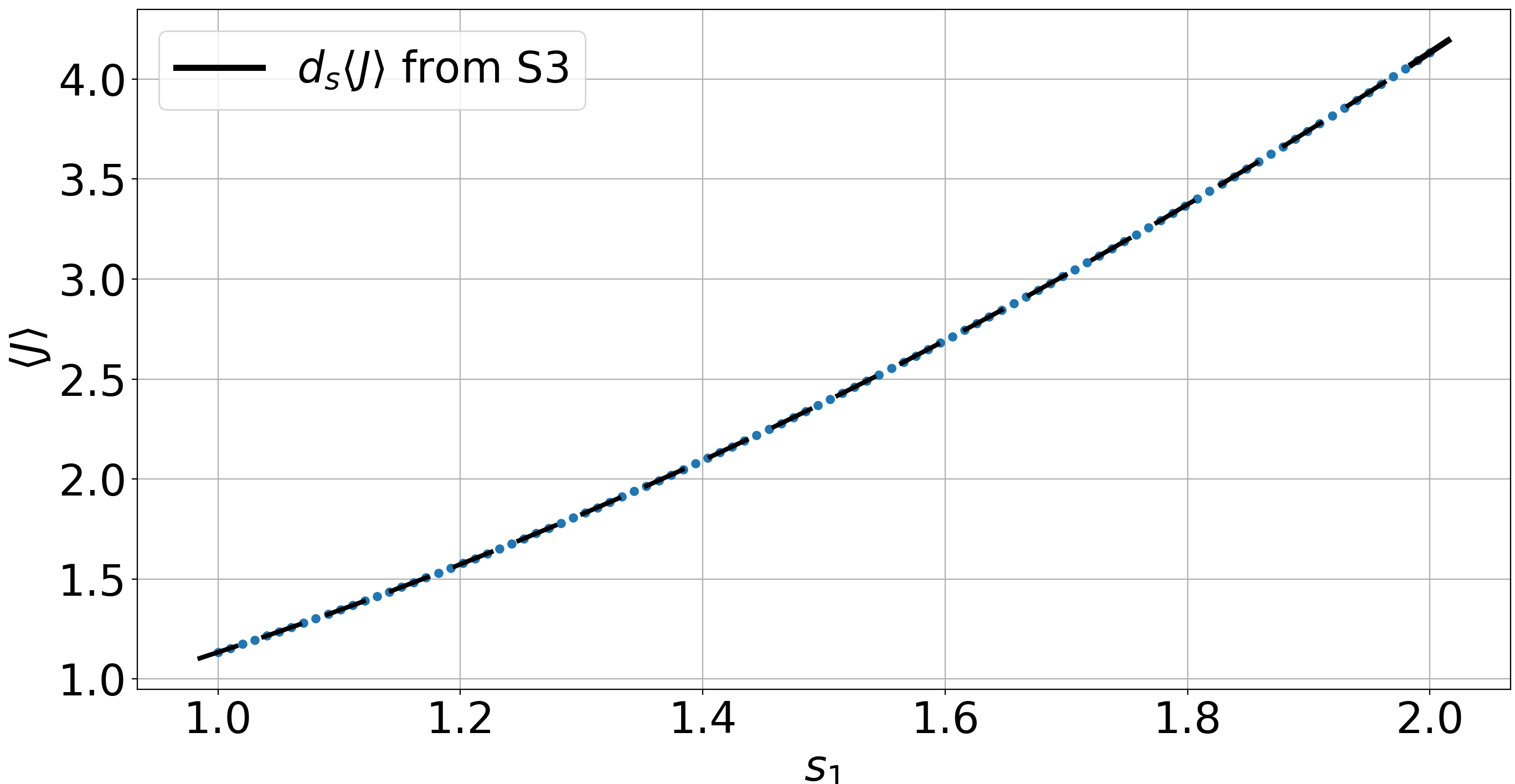}
    \includegraphics[width=0.49\textwidth]{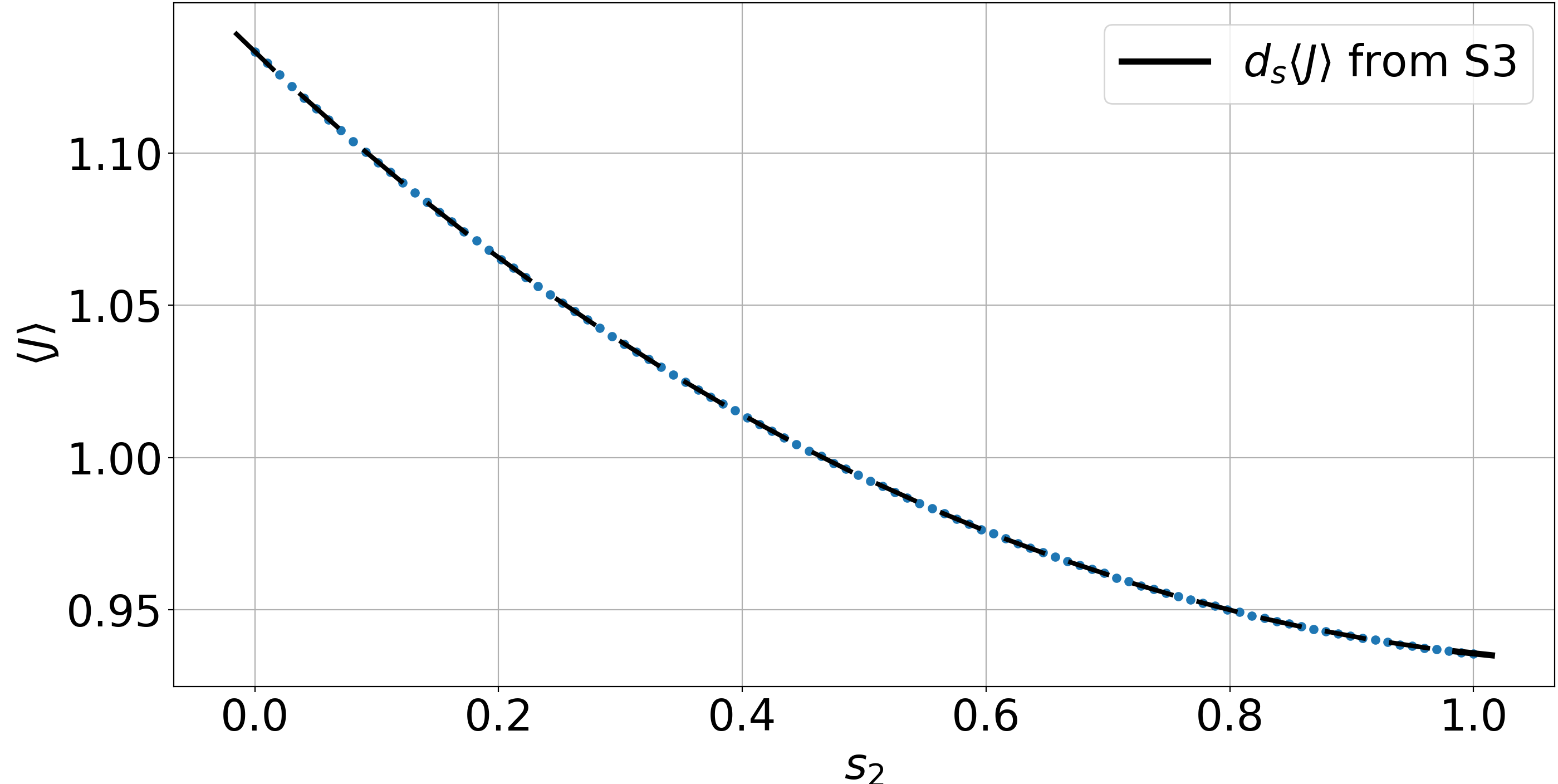}
    \caption{Linear response validation on the Solenoid attractor described in section \ref{sec:solenoid}. The sensitivities of $\langle J\rangle = \langle x^{(1)}.x^{(1)} + x^{(2)}.x^{(2)}\rangle$ with respect to the parameter $s_1$ (left) and $s_2$ (right) computed by S3 are shown as black lines at a number of parameter values on the $\langle J\rangle$ vs parameter curve. In each case, the other parameter is held constant at the reference value.}
    \label{fig:sol_sens}
\end{figure}

Next we consider the classical uniformly hyperbolic example of the Smale-Williams Solenoid map. We introduce a set of parameters $s = [s_1, s_2]^T,$ which takes the reference value $[1, 0]^T$. The perturbed Smale-Williams Solenoid map on the solid torus $\mathbb{T}^2$ is defined in cylindrical coordinates as,
\begin{align}
\label{eqn:solenoid}
    \varphi_s([{\rm r}, \theta, x^{(3)}]^T) = 
    \begin{bmatrix}
    s_1 + ({\rm r} - s_1)/4 + \cos \theta/2 \\
    \left( 2\theta + (s_2/4)\sin(4\theta) \right) \;{\rm mod}\; (2\pi)  \\
    x^{(3)}/4 + \sin\theta/2
    \end{bmatrix}.
\end{align}
The map can be expressed in Cartesian coordinates, denoted $[x^{(1)},x^{(2)},x^{(3)}]^T,$ by left and right compositions of $\varphi_s$ as in Eq. \ref{eqn:solenoid} with the following coordinate transformation and its inverse respectively: $x^{(1)} = {\rm r}\cos\theta,$ $x^{(2)} = {\rm r}\sin\theta.$ In Figure \ref{fig:solenoidAttractor}, we illustrate the effect of parameter variation on the Solenoid attractor, which is shown on the $x^{(1)}$-$x^{(2)}$ plane. 
\begin{figure}
    \centering
    \includegraphics[width=\textwidth]{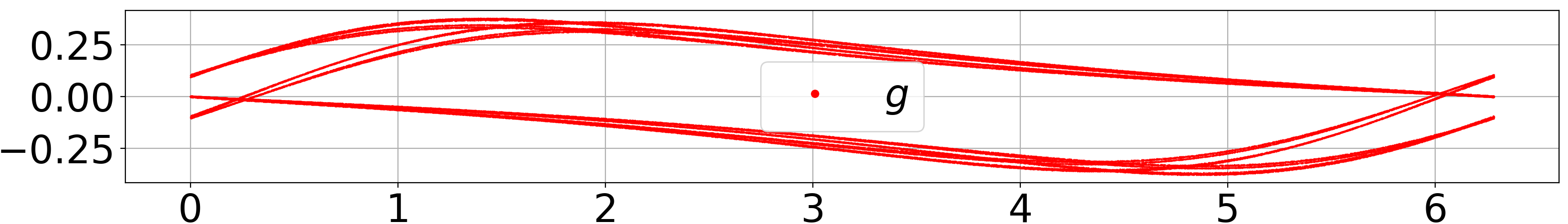}
    \includegraphics[width=\textwidth]{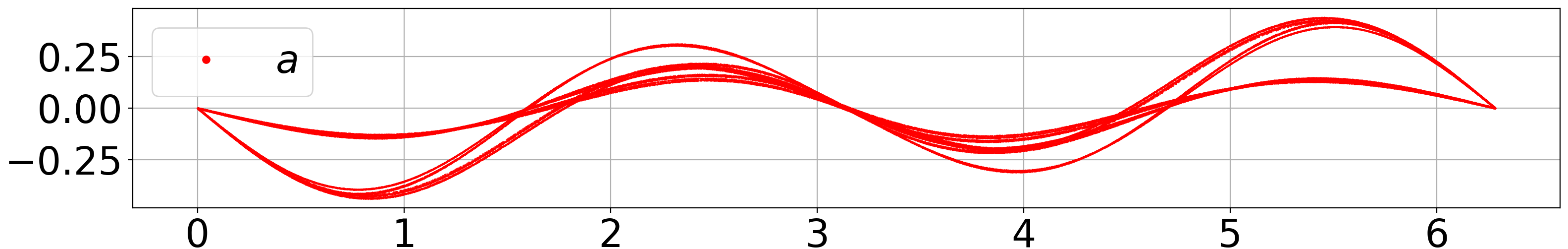}
    \includegraphics[width=\textwidth]{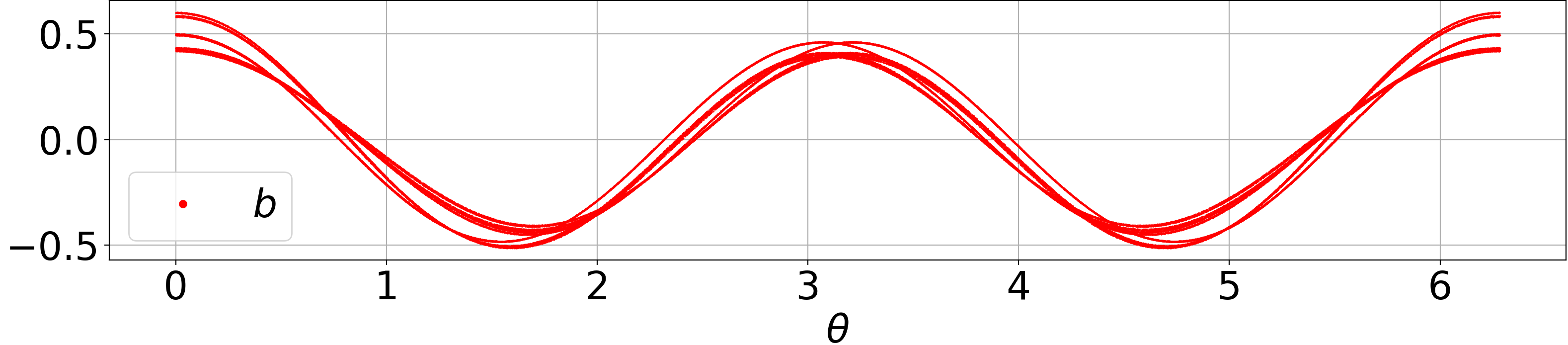}
    \caption{Three scalar fields computed for the unstable contribution: $g$ (top), $a$ (middle) and $b$ (bottom), as a function of $\theta.$ These functions were computed by the S3 algorithm (section \ref{sec:algorithm}) using an orbit of length  100,000 of the Solenoid map (section \ref{sec:solenoid}) with $s=[1, 0]^T.$}
    \label{fig:solenoid_metrics}
\end{figure}
The subfigures show the points on an orbit of length 2 million at the corresponding values of $s.$
On the left is the attractor of the unperturbed map at the standard value of $s=[1,0]^T.$ The center subfigure, at the value $s = [1,1]^T$, shows a change in the {\em shape} of the attractor due to an $s_2$ perturbation, while the rightmost subfigure indicates a change in the {\em size} of the attractor due to an $s_1$ perturbation. Qualitatively, we can also note change in the distribution of points on the perturbed attractor. For instance, the effect of an increase in $s_2$ is to reduce the probability of visiting the outer rim of the attractor. Both these types of qualitative changes, i.e., in the position and shape of, and in the distribution on the attractor are indicative of changes in the SRB measure and are captured precisely by the computation of Ruelle's formula. 
\begin{figure}
    \centering
    \includegraphics[width=0.49\textwidth]{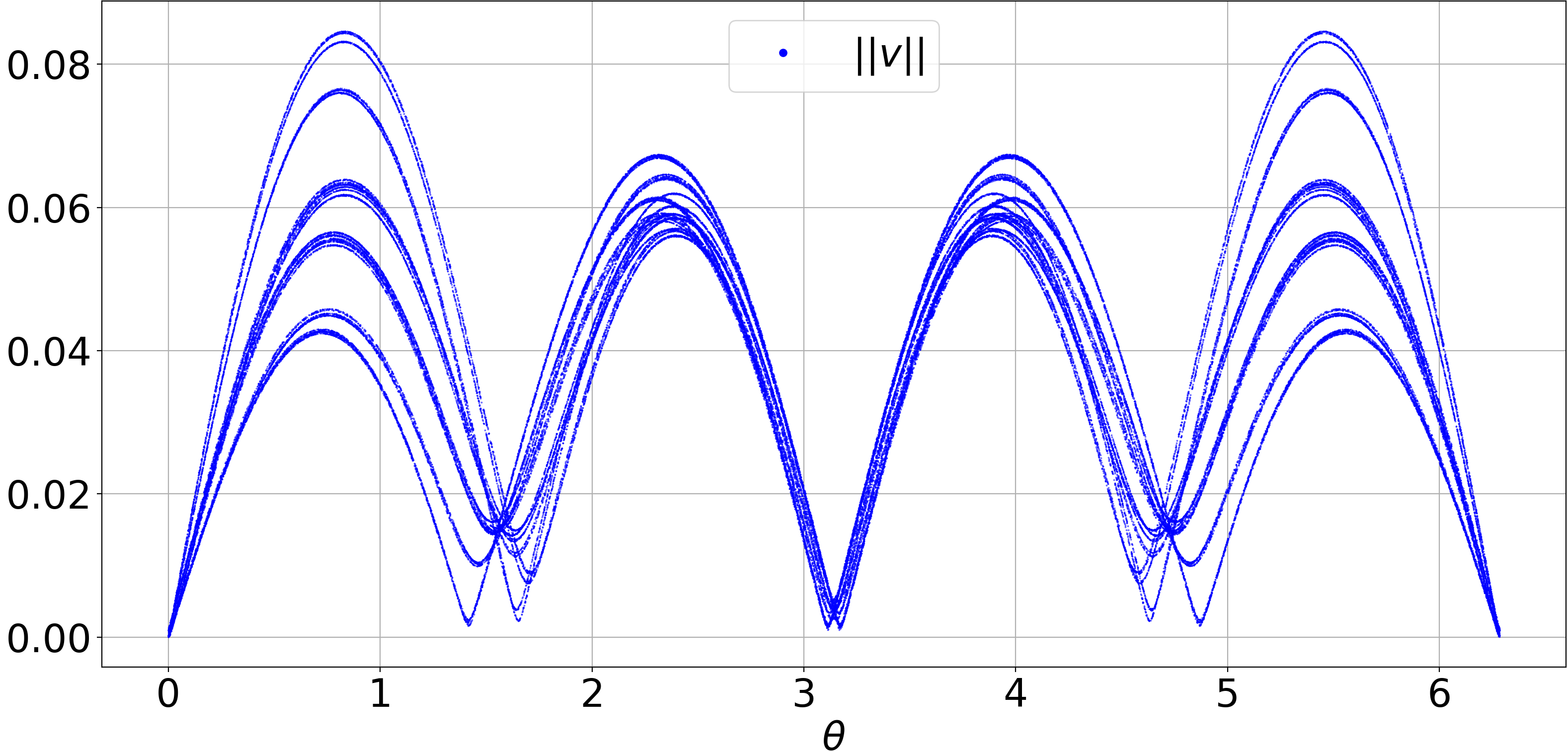}
    \includegraphics[width=0.49\textwidth]{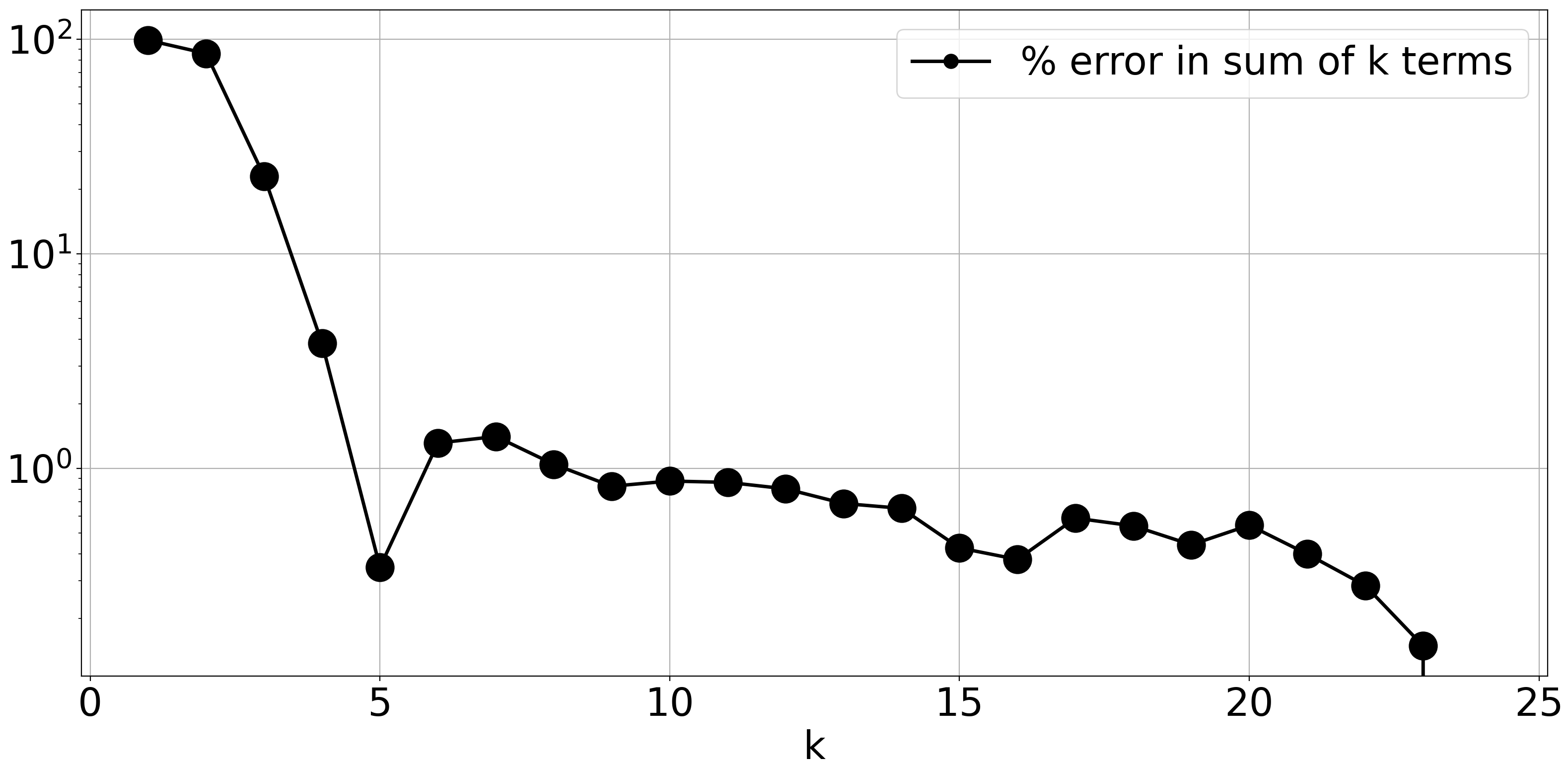}
    \caption{Left: the norm of the regularized tangent vector field computed by the S3 algorithm shown as a function of the $\theta$ coordinate. Right: percentage relative error in the unstable contribution when computed as a sum of the first $k$ terms, as a function of $k.$ The reference is taken as the unstable contribution computed with 24 terms. For both subfigures, the computations are performed along an orbit of length 100,000 of the Solenoid map (section \ref{sec:solenoid}) with $s=[1, 0]^T.$}
    \label{fig:solenoid_metric_1}
\end{figure}

By inspecting Eq. \ref{eqn:solenoid}, we can see that the ${\rm r}$ and $x^{(3)}$ coordinate directions are stable at each phase point. The unstable direction, however, is not aligned with $\hat{\theta}$ everywhere.  Thus, an $s_1$ perturbation corresponds to a stable vector field $\chi  \in E^s,$ and an $s_2$ perturbation has both stable and unstable components, albeit a small stable component. 

As done in \cite{nisha-s3}, had we split Ruelle's formula by decomposing the parameter perturbation into its stable and unstable components, the unstable contribution to the sensitivity is 0 for an $s_1$ perturbation. However, in the S3 algorithm, although the two terms of the split Ruelle's formula are called stable and unstable contributions, the perturbation vector field $\chi$ is not split along $E^s$ and $E^u$. Hence, with S3, it is not a priori clear whether either contribution is zero in either type of parameter perturbation although we expect that for the $s_1$ perturbation, the stable contribution dominates, and for the $s_2$, the unstable dominates.

\subsubsection{Performance of S3}
It is numerically verified that the S3 algorithm computes the correct sensitivity in both cases, as shown in Figure \ref{fig:sol_sens}. The objective function, $J$ is arbitrarily taken to be $J = {\rm r}^2 = x^{{(1)}^2} + x^{{(2)}^2},$ and the sensitivity of $\langle J\rangle$ is computed at a range of values of $s_1$ ($s_2$) shown on the left (right) of Figure \ref{fig:sol_sens}. The parameter that is not indicated on the horizontal axis is held constant at its reference value in each case. The computed S3 sensitivities closely match the slopes of the $\langle J\rangle$-vs-$s$ curves (shown as blue dots) that are computed from ergodic averages using orbits of length 160 million. The S3 sensitivities, which are linearly extrapolated and shown as black lines, are computed along orbits of length 3.2 million at each value of $s$. The source code to generate S3 results on the Solenoid map (as well as all other numerical results in this section) can be found at \cite{nisha-code}. 

\begin{figure}
    \centering
    \includegraphics[width=\textwidth]{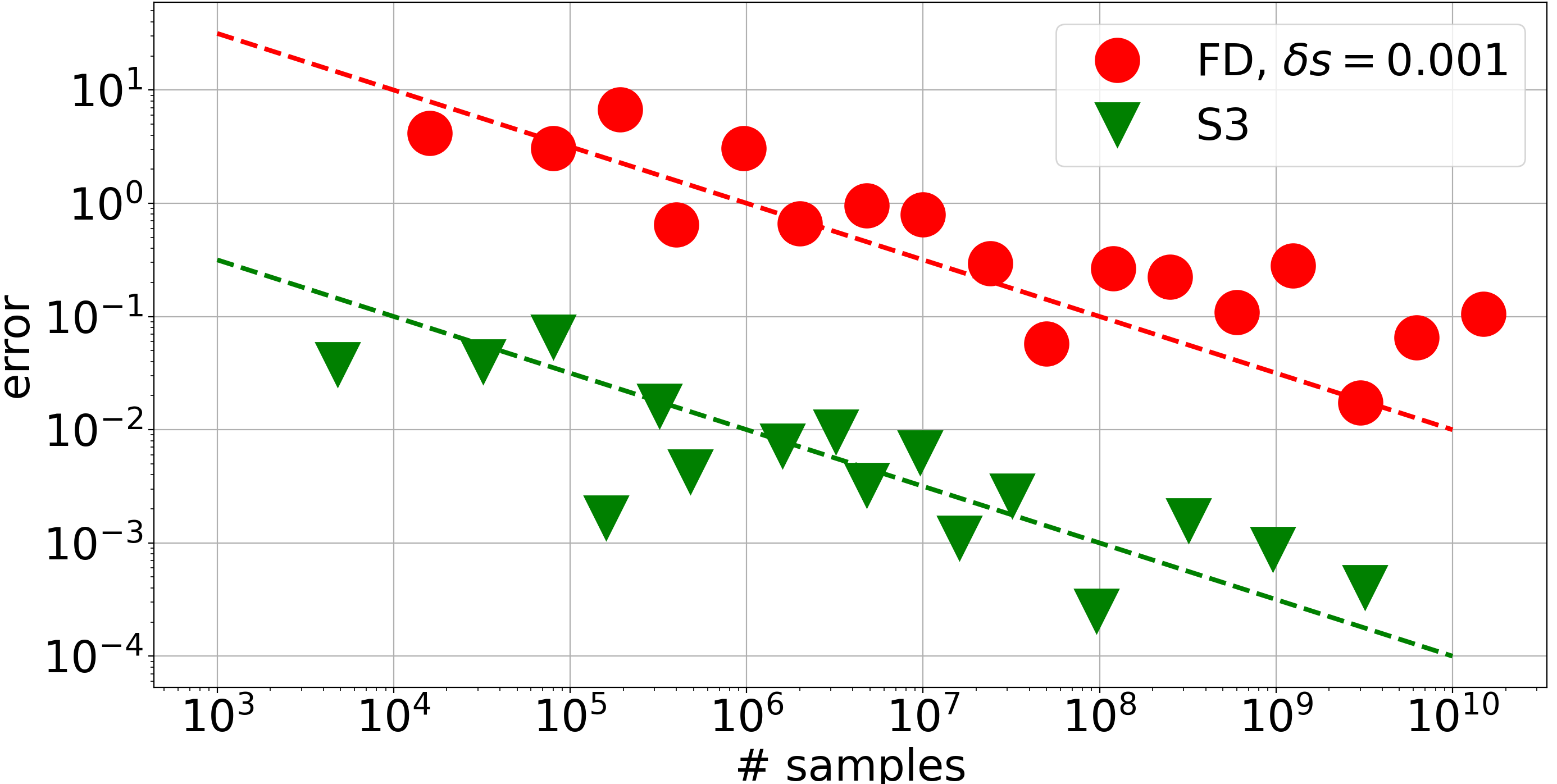}
    \caption{Comparison of relative error in the sensitivity of $\langle J\rangle$ to $s_2$ at $s= [1,0]$ of the Solenoid map (section \ref{sec:solenoid}) as a function of the number of samples, between S3 and finite difference. The relative error is computed with respect to the true value, which is taken to be the S3 derivative obtained with $N = 9.6$ billion samples. The finite difference in the parameter $s_2$ used to calculate the derivatives is set to 0.001 in order to obtain the red dots. The green triangles are the derivatives computed using the S3 algorithm described in section \ref{sec:algorithm}. The dotted lines have a slope of -0.5.}
    \label{fig:errorVsCost}
\end{figure}
Having validated the S3 sensitivities on the Solenoid attractor, given the low dimensionality of this problem, we can further extract and visualize intermediate quantities computed by the S3 algorithm. To generate the remaining results of this section, $s_1$ is fixed at 1, and the derivative of $\langle J\rangle = \langle x^{(1)}.x^{(1)} + x^{(2)}.x^{(2)}\rangle$ with respect to $s_2$ is computed at $s = [1,0]^T$. 

One of our main results (Theorem \ref{thm:thmS3}) asserts the differentiability of the scalar field $a$ and the vector field $v$ in the unstable direction (see section \ref{sec:proofOfTheorem1} for the proof), and that $g,$ $a$, $b$ are H\"older continuous fields. Another main result (Theorem \ref{thm:thmS3Computation}) proves the exponential convergence of the recursive computations of the functions $v,$ $a,$ $b$ and $g$ (see section  \ref{sec:proofOfTheorem2} for the proof). While numerical verification of these analytically proved results is beyond the scope of this paper, we at least confirm the boundedness of the quantities $a,$ $b$, $g$ and $\|v\|$ on a long orbit. In Figure \ref{fig:solenoid_metrics}, we show the functions $g$, $a$ and $b$ computed by the S3 algorithm, which are associated with the unstable contribution. We plot the functions as a function of $\theta,$ which is close but not exactly parallel to the unstable direction at each point. Each subfigure shows the values of the corresponding functions at each point of an orbit of length 100,000 with $s = [1,0]^T.$ The S3 algorithm is run to compute the derivative $d_{s_2}\langle J\rangle.$

On the left hand side of Figure \ref{fig:solenoid_metric_1}, we plot the norm of the regularized tangent vector field $\|v\|$ also as a function of the $\theta$ coordinate. Reassuringly, the norm of the regularized tangent vector field is always bounded. Further, as we noted in the previous subsection, the unstable contribution is the majority of the derivative in the case of the $s_2$ perturbation considered here. We confirm numerically that the unstable contribution dominates; the stable contribution is less than 0.01\% of the overall sensitivity.

Figure \ref{fig:solenoid_metric_1} (right) provides numerical evidence for the rapid convergence (due to correlation decay) of Eq. \ref{eqn:unstableContributionErgodicAverageRecalled} -- the unstable contribution. We show the percentage relative error in the unstable contribution computed with the first $k$ terms when compared to the baseline value calculated with 24 terms. Clearly, sums of terms in the series beyond 8 terms incur an error of less than 1\% compared to the sum of the first 24 terms. In fact, each term for $k > 6$ evaluates to less than 1\% of the overall sensitivity (calculated with 24 terms of the series).  

In Figure \ref{fig:errorVsCost}, a comparison of the error convergence of S3 against a na\"ive finite difference calculation is shown using the Solenoid map. The finite difference approximation of the same linear response calculated with $N$ samples is as follows, 
\begin{align}
\label{eqn:centeredDifference}
    \langle J, \partial_{s_2}\mu_s\rangle_{\rm FD} = \dfrac{\langle J\rangle_N([s_1,s_2 + \delta s]) - \langle J\rangle_N([s_1, s_2 - \delta s])}{2\delta s},
\end{align}
where $\langle J\rangle_N(s)$ is an $N$-sample Monte Carlo estimate of $\langle J\rangle(s).$ Such an estimate can be computed either as an ergodic average along a single trajectory initialized Lebesgue almost everywhere, or as a sample average at samples according to $\mu$ on the attractor.   To generate these plots efficiently, we do not compute ergodic averages along a single trajectory, as such a computation is a serial operation. Rather we sample average along multiple short trajectories in parallel, where each trajectory is initialized with a sample according to $\mu$, which is achieved after a sufficiently long run-up time, starting from Lebesgue a.e. We observe in Figure \ref{fig:errorVsCost} that the numerator of Eq. \ref{eqn:centeredDifference}  is overwhelmed by statistical noise. As a result of the law of the iterated logarithm \cite{denker_philipp_1984} that applies to $\langle J\rangle_N$ (asymptotically), the errors (ignoring the iterated logarithmic factor) in the finite difference decay as ${\cal O}(1/\sqrt{N}).$ This is confirmed by the finite difference data points being well-approximated by a line with slope -0.5 (shown as a dotted red line). However, the central difference approximation grows as ${\cal O}(1/\delta s).$ As a result, the finite difference derivatives have much larger errors -- in this case, two orders of magnitude larger errors -- than the S3 derivatives (which are shown as green triangles), for the same number of samples. For instance, to achieve a 10\% relative error, a central finite difference with $\delta s = 0.001$ requires more than a billion samples while the S3 algorithm shows less than 10\% error even at 10,000 samples. 

Note that the true value to compute relative error is taken to be the S3 derivative obtained from a $\sim $ 10 billion sample computation, which is verified against the slope of the response curve $\langle J\rangle$-vs-$s_2$ at $s_2 = 0.$
Finally, we note that the error convergence in S3 is as predicted by the analysis in sections \ref{sec:proofOfTheorem1} and \ref{sec:proofOfTheorem2}. In particular, in sections \ref{sec:convergenceOfTheStableContribution} and \ref{sec:unstableContributionProof} respectively, we show that the error in the stable contribution and that in the unstable contribution diminish at the rate of ${\cal O}(\sqrt{\log\log N}/\sqrt{N}),$ with $N$ samples. Consistent with these results, the dotted green line, which has a slope equal to -0.5, well approximates the S3 data points in Figure \ref{fig:errorVsCost} (we may ignore the $\sqrt{\log\log N}$ factor, which only has a negligible effect on the slope).

\section{Proof of Theorem \ref{thm:thmS3}}
\label{sec:proofOfTheorem1}
In this section, we prove Theorem \ref{thm:thmS3}, which establishes that the S3 decomposition of Ruelle's formula into stable and unstable contributions (Eq. \ref{eqn:splitRuellesFormula2}) is well-defined. We show that the S3 decomposition exists and is differentiable on the unstable manifold. In particular, we prove that the regularized perturbation field, $v,$ that is the limit of the sequence $\left\{ v^n\right\}$ in Eq. \ref{eqn:stableTangentVectorField}, exists. This proves Theorem \ref{thm:thmS3}-1. The second statement, Theorem \ref{thm:thmS3}-2., is shown by proving that this vector field $v$ is differentiable on the unstable manifold. We then prove that the differentiability in the unstable direction of $v$ implies that of the scalar field $a$, and hence Theorem \ref{thm:thmS3}-2. is proved. First we start by proving the existence and uniqueness of $v.$ Before we begin, we establish some notation for oblique projection operators on vector fields, and describe their properties that we use in the proofs.

\begin{notation}
Let $\stableProjection_x: T_x M \to T_x M$ denote the linear operator that gives the component of a tangent vector in 
$E^s_x,$ in a direct sum of components along $E^u_x$ and $E^s_x$. If $v_x = v_x^u + v_x^s$ where $v_x^u\in E^u_x$ and
$v_x^s\in E^s_x$, then, $\stableProjection_x\: v_x := v_x^s$. 
Applying $\stableProjection_x$ to every point on $M$ would result in a linear operator on vector
fields of $M$, hereafter denoted $\stableProjection$.
\end{notation}

\begin{remark}
\label{rmk:operatorNormOfS}
Note that $\unOrthProjection_x$ is an orthogonal projector.  Its norm is therefore 1.
In contrast, the operator $\stableProjection_x$ is not an orthogonal projector; $\stableProjection_x$ is uniformly bounded over $M$, i.e.,
$S:=\sup_{x\in M}\norm{\stableProjection_x}$ is finite.
\end{remark}

\begin{remark}
It follows from their definition that both operators are idempotent, i.e.,
$\unOrthProjection^2 = \unOrthProjection$ and $\stableProjection^2 = \stableProjection$.
\end{remark}

\begin{remark}
\label{rmk:stableProjectionProperty}
Due to the covariance of $E^s$ and $E^u,$ i.e.,
$d\varphi^n_x\: E^s_x = E^s_{\varphi^n x}$ and
$d\varphi^n_x\: E^u_x = E^u_{\varphi^n x},$
the operator $\stableProjection$ satisfies
$d\varphi_x^n \: \stableProjection_x = \stableProjection_{\varphi^n x} \: d\varphi^n_x.$
Operating on vector fields on $M$, this equality leads to
$d\varphi^n \: \stableProjection = \stableProjection \: d\varphi^n.$
\end{remark}

\begin{remark}
\label{rmk:pvispsv}
At every $x \in M,$ both $\stableProjection_x$ and $\unOrthProjection_x$ map $E^u_x$ to 0 by their definitions.
Also by definition, for any $v_x\in T_x M$,
both $v_x - \stableProjection_x\: v_x$ and $v_x - \unOrthProjection_x\: v_x$ are in $E^u_x$.
We thus have $\unOrthProjection_x\: v_x = \unOrthProjection_x \: \stableProjection_x\: v_x$ and
$\stableProjection_x\: v_x = \stableProjection_x \: \unOrthProjection_x\: v_x$. 
The equivalent expressions, as operators on vector fields, are
$\unOrthProjection \: \stableProjection = \unOrthProjection$ and $\stableProjection \: \unOrthProjection = \stableProjection$.
\end{remark}

We now prove that a regularized perturbation field $v$ exists, which is the limit of the tangent equation solved with repeated projections (section \ref{sec:stableContribution}). We note that, in all the proofs that follow, the constants $C$ and $\lambda$ refer to the eponymous constants in the definition of uniform hyperbolicity (section \ref{sec:background}), and the constant $S$ refers to the norm of the stable projection operator $\stableProjection$, as defined above (Remark \ref{rmk:operatorNormOfS}); all other constants, such as $A, c,$ etc. may vary from line to line.

\subsection{Existence and uniqueness of a regularized perturbation field}
\begin{lemma}
\label{lem:existenceAndUniquenessOfv}
For any bounded vector field $\chi: M \to \mathbb{R}^m$, there exists a unique bounded vector field $v$
on $M$ that satisfies 
\begin{align} \label{eqn:stable_tangent}
v = \unOrthProjection\big(d\varphi\,v + \chi\big)
\end{align}
\end{lemma}
Note that because $\unOrthProjection^2 = \unOrthProjection$, Eq. \ref{eqn:stable_tangent} implies that
$\unOrthProjection v = v$.  In other words, at every $x\in M$, $v_x$ is orthogonal to
the one-dimensional unstable subspace $E^u_x$.  Also, a vector field $\chi$ being bounded
means that $\sup_{x\in M}\|\chi_x\|$ is finite. 

\begin{proof}
We first prove existence.  Let $v^0=0$ be an arbitrary bounded vector field on $M.$
Let $v^{k+1} = \unOrthProjection\big(d\varphi\,v^k + \chi\big)$ for $k=0,1,\ldots$
We will show that for each $x\in M$, $\{v^k_x, k=0,1,\ldots\}$ is a Cauchy sequence
and thus converges to a limit, namely $v_x$.

From the definition of $v^k$ and the linearity of the operators involved, we have
$v^{k+1} - v^k = \unOrthProjection\: d\varphi (v^k - v^{k-1}) = (\unOrthProjection\: d\varphi)^k (v^1 - v^0).$
Here, we can use the relations between $\unOrthProjection, \stableProjection,$ and $d\varphi$
(See Remarks \ref{rmk:stableProjectionProperty} and \ref{rmk:pvispsv}) to simplify the
linear operator as follows. Since $\unOrthProjection = \unOrthProjection\stableProjection,$ $(\unOrthProjection \: d\varphi)^k = (\unOrthProjection \:\stableProjection\: d\varphi)^k = (\unOrthProjection\: d\varphi \:\stableProjection)^k,$ where the second equality follows from Remark \ref{rmk:stableProjectionProperty}. Now using $\stableProjection \unOrthProjection = \stableProjection$ (Remark \ref{rmk:pvispsv}), $ (\unOrthProjection \: d\varphi)^k = (\unOrthProjection\: d\varphi \:\stableProjection)^k = \unOrthProjection\: d\varphi\: (d\varphi \: \stableProjection)^{k-1}$, and finally using Remark \ref{rmk:stableProjectionProperty},  $(\unOrthProjection \: d\varphi)^k = \unOrthProjection \: d\varphi\: (d\varphi \: \stableProjection)^{k-1} = \unOrthProjection \:d\varphi^k\: \stableProjection.$

From inequality \ref{eqn:stable} and the uniform boundedness of $\|\stableProjection_x\|$ by $S$ (Remark \ref{rmk:operatorNormOfS}),
$\|d\varphi_x^k\:\stableProjection_x\| \leq C \: \lambda^k\:S.$  Together with $\|\unOrthProjection_x\|=1$, we have a
uniform bound
\begin{align} \label{eqn:operator_norm_k}
    \|(\unOrthProjection\: d\varphi)_x^k\| = \|\unOrthProjection_{\varphi^k x}\: d\varphi^k_x\:\stableProjection_x\| \leq C \: \lambda^k \: S\;, \quad \forall\; x \in M.
\end{align}
Thus,
\begin{align}
\label{eqn:vnminusvnminus1}
    \norm{v^{k+1}_{x} - v^k_x} = \norm{(\unOrthProjection\: d\varphi)^k_{\varphi^{-k} x} (v^1_{\varphi^{-k} x} - v^0_{\varphi^{-k} x})}\leq C \: \lambda^k\: S\: \norm{v^1_{\varphi^{-k} x} - v^0_{\varphi^{-k} x}}.
\end{align}
Because $v^0$ is bounded by definition, and $v^1$ is bounded due to the boundedness 
of $v^0, \chi,$ and $d\varphi$, there exists an $A > 0$ such that $C\: S\:\norm{v^1_{x} - v^0_{x}} < A$
for all $x\in M$.  Thus, for all $k\ge 0$, and $x \in M,$
\begin{align}
    \norm{v^{k+1}_{x} - v^k_x} \leq A \: \lambda^k.
\end{align}
And, for all $m > n \geq 0,$
\begin{align} \label{eqn:norm_diff_vmn}
    \norm{v^m_x - v^n_x} & \leq \sum_{k=n}^{m-1} \norm{v^{k+1}_{x} - v^k_{x}}
    \leq A\:\sum_{k=n}^{m-1} \lambda^k  < A \: \frac{\lambda^n}{1-\lambda}.
\end{align}
Then, given $\epsilon > 0$, choosing $N > \lfloor \log|A/(\epsilon(1-\lambda))|/\log|1/\lambda| \rfloor,$ for all 
such that $m > n\ge N,$
\begin{align}
    \norm{v^m_x - v^n_x} \leq \epsilon, \:\: \forall\: x \in M.
\end{align}
Thus $\left\{v^n_x\right\}$ is a \emph{uniform} Cauchy sequence, and thus the sequence $\left\{ v^n \right\}$ converges uniformly on $M.$

Equation \ref{eqn:norm_diff_vmn}, applied to $n=0$, also implies that $\norm{v^m_x} \leq \norm{v^0_x} + \frac{A}{1-\lambda}$, for all $x$ and $m$.  The limit
of the Cauchy sequence thus uniformly satisfies the same bound.

We now show this limit, defined by the vector field $v$ as $v_x := \lim_{n\to \infty} v^n_x$, is \textbf{unique} in satisfying Eq. \ref{eqn:stable_tangent}.
Suppose $\Delta v$ is a bounded vector field such that $v+\Delta v$
also satisfies Eq. \ref{eqn:stable_tangent}.  Without loss of generality, let
$\Delta v$ be bounded by 1, i.e., $\|\Delta v_x\| \leq 1$ for all $x\in M$. Due to the linearity of the operators involved,
$\Delta v = \unOrthProjection\:d\varphi\,\Delta v$, and by iteration,
$\Delta v = (\unOrthProjection\:d\varphi)^k\,\Delta v$ for any $k\in \mathbb{N}$.  Using Eq. \ref{eqn:operator_norm_k},
$\|(\Delta v)_x\| \le C\:\lambda^k\:S \:\|(\Delta v)_{\varphi^{-k}x}\| \le C\:\lambda^k\:S$ for all $x\in M$.
Because $\lambda < 1$ and this inequality holds for all $k$, $\|\Delta v_x\|$ must be 0.
Thus $\Delta v=0$.  The uniqueness of $v$ follows.
\end{proof}
We remark that, without loss of generality, starting from $v^0 = 0,$ we obtain the following simple expressions for the regularized tangent vector field $v,$ 
\begin{align}
    v = \sum_{n=0}^\infty (\unOrthProjection\: d\varphi)^n \unOrthProjection \chi = \sum_{n=0}^\infty \unOrthProjection\: d\varphi^n \stableProjection \unOrthProjection \chi = \sum_{n=0}^\infty \unOrthProjection\: d\varphi^n \stableProjection \chi,
\end{align}
where we have used Remark \ref{rmk:pvispsv} to obtain the third equality.
This establishes the existence of the regularized perturbation field $v$, and completes the proof of Theorem \ref{thm:thmS3}-1. Having shown that $v$ exists, we now show that it is differentiable in the unstable direction. This differentiability is used to prove Theorem \ref{thm:thmS3}-2. Before this, we first show that $q$ is differentiable in its own direction. Although it is known that self-derivatives of the unstable/stable subspaces exist (see, for instance, Remark after Lemma 19.1.7 of \cite{katok}), we include a proof here for completion.
\subsection{Existence of self-derivative of the unstable direction}

We refer to as the unstable self-derivative, the vector field $w := d^2_{\xi} \Phi^x.$ We now show its existence.  
\begin{lemma}
\label{lem:existenceAndUniquenessOfw}
    There exists a unique, bounded vector field $w$ that satisfies
    \begin{align}
    \label{eqn:definitionOfw}
        w = \unOrthProjection \dfrac{d^2 \varphi (q, q) + d\varphi\: w}{\alpha^2},
    \end{align}
where $\alpha := \|d\varphi\: q\|.$
\end{lemma}
Note that, since $w = \unOrthProjection w,$ $w$ is orthogonal to $q.$ 
\begin{proof}
Consider a sequence of vector fields $\left\{w^n\right\}$ that satisfies the recurrence relation
\begin{align}
        \label{eqn:recursionForw}
        w^{n+1} = \unOrthProjection \dfrac{d^2 \varphi (q, q) + d\varphi\: w^n}{\alpha^2}
    \end{align}
We will now show that $w^n$ is a uniformly Cauchy sequence and hence converges uniformly.
From the above recurrence relation,
\begin{align}
    w^{n+1} - w^n &= \dfrac{1}{\alpha^2}\:\unOrthProjection \:  d\varphi \: (w^n - w^{n-1}),
\end{align}
iterating which, gives, for all $n \in \mathbb{N}$
\begin{align}
    w^{n+1} - w^n &= \dfrac{1}{\prod_{k=0}^{n-1} \alpha^2_{\varphi^{-k} \cdot}}\:\left(\unOrthProjection \:  d\varphi\right)^n \: (w^{1} - w^{0}).
\end{align}
Then, using i) $(\unOrthProjection\: d\varphi)^n = \unOrthProjection \: d\varphi^n\: \stableProjection$, which is shown in Lemma \ref{lem:existenceAndUniquenessOfv}, and ii) $\|\unOrthProjection_x\| = 1,$ at all $x \in M,$
\begin{align}
     \|w^{n+1}_x - w^n_x\| &\leq \dfrac{1}{\prod_{k=0}^{n-1} \alpha^2_{\varphi^{-k} x}}\: \|d\varphi^n_{\varphi^{-n}x} \:\stableProjection_{\varphi^{-n} x} (w^1_{\varphi^{-n} x} - w^{0}_{\varphi^{-n} x})\|. 
\end{align}
Now using i)$\prod_{k=0}^{n-1} \alpha_{\varphi^{-k} x} \geq (1/C) \lambda^{-n}$, and ii) $\norm{d\varphi^n_x \stableProjection_x}\leq C \: \lambda^n\: S,$ for all $x \in M,$ both of which follow from the definition of uniform hyperbolicity,
\begin{align}
    \|w^{n+1}_x - w^n_x\| &\leq 
   C^3\: \lambda^{3n} \: S\: \|w^1_{\varphi^{-n} x} - w^{0}_{\varphi^{-n} x}\|.
\end{align}
Clearly, since the map from $w^0$ to $w^1$ is bounded, and by assumption $w^0$ is bounded, there exists a constant $A$ such that $A := \sup_{x\in M} \| w^1_x - w^0_x\|.$
Hence, for any $m > n \geq 0,$
\begin{align}
    \| w^m_x  - w^n_x \| \leq A \: S\: C^3 \sum_{k=n}^{m-1} \lambda^{3k} \leq A \: S\: C^3 \dfrac{\lambda^{3n}}{1-\lambda^3}. 
\end{align}
Since the above inequality holds for all $x,$ the sequence $w^n_x$ is uniformly Cauchy and hence converges uniformly. Let $w := \lim_{n\to\infty} w^n.$ To show that $w$ is unique, suppose $w^n$ and $\tilde{w}^n$ are two different sequences that both satisfy Eq. \ref{eqn:recursionForw}. Then, at every $x \in M,$ 
\begin{align}
     \|w^{n}_x - \tilde{w}^{n}_x\| &\leq \dfrac{1}{\alpha^2_{x}}\: \|d\varphi_{\varphi^{-1}x} \:\stableProjection_{\varphi^{-1} x}\: (w^{n-1}_{\varphi^{-1} x} - \tilde{w}^{n-1}_{\varphi^{-1} x})\| \\
     &\leq A' \lambda^{3n},
\end{align}
where the second inequality is obtained by recursively applying the first $n$ times and then applying the definition of uniform hyperbolicity as done previously. Taking the limit $n \to \infty$ on both sides, we obtain that at every $x \in M,$ $\|w_x - \tilde{w}_x\| = 0.$ Thus, $w$ is a unique vector field independent of $w^0$. 
\end{proof}
We note that the vector field $w$ that satisfies Eq. \ref{eqn:definitionOfw} is also $\partial_\xi q.$ This relationship can be derived by differentiating with respect to $\xi$ the following equation that expresses the definition of the unstable CLV, $q$:
\begin{align}
    \alpha_{\varphi x} \: q_{\varphi x} =  d\varphi_x \: q_x.
\end{align}
We derive this relationship in detail in section \ref{sec:formulaforw}. As described in section \ref{sec:unstableContribution}, the practical computation of the vector field $w$ involves repeated application of Eq. \ref{eqn:recursionForw}, by choosing, without loss of generality, $w^0_x = 0$, at all $x.$ Choosing an initial phase point $x$ sampled according to $\mu,$ one then obtains the values  $w^n_{\varphi^n x}$ along the orbit of $x.$ As $n$ increases, these values exponentially approach the true value of $w$ along that orbit, as the proof of Lemma \ref{lem:existenceAndUniquenessOfw} shows. The convergence of the numerical procedure to compute $w,$ which we have shown here, is needed to prove the convergence of the method (in the S3 algorithm, section \ref{sec:algorithm}) to compute the unstable derivative of $v.$ We focus on this convergence next, which ultimately establishes the differentiability of the S3 decomposition.

\subsection{Convergence of the unstable projections of the regularized tangent solutions}
 The S3 decomposition of the perturbation field $\chi$ into $a\: q$ and $\chi - aq$ is differentiable on the unstable manifolds, if $y := \partial_\xi v$ exists. To see why, we first note that the scalar field $a$ is the limit of the iterative projections of the sequence $v^n,$ which in practice is obtained by solving the regularized tangent equation (Eq. \ref{eqn:stableTangentEquation}). We formally establish the existence of $a$ by showing that the iterative procedure used to obtain it converges.

\begin{lemma}
\label{lem:convergenceOfa}
Let $\left\{a^n\right\}, \: n \in \mathbb{Z}^+,$ be a sequence of scalar fields determined by a sequence of vector fields, $\left\{v^n\right\},$ that satisfies i) $v^n_x \cdot q_x = 0,\;\;\forall \: x \in M,\; n \in \mathbb{Z}^+$, and,
\begin{align}
\label{eqn:alternativeStableTangentEquation}
    ii) \; v^{n+1} &= (d\varphi)\: v^n + \chi - a^{n+1} q, \;\; n\in \mathbb{Z}^+.
\end{align}
Then, $\left\{a^n\right\}$ converges uniformly. 
\end{lemma}
Note that i) and ii) above constitute the iterative procedure to compute the regularized tangent solution in the S3 algorithm (section \ref{sec:algorithm}). The sequence $\left\{v^n\right\}$ is identical to the sequence defined in Lemma \ref{lem:existenceAndUniquenessOfv}.
\begin{proof}
The scalar fields $a^{n+1}$ are the projections of $d\varphi\: v^n + \chi$ on $q$, 
\begin{align}
\label{eqn:orthogonalityToGivean}
    a^{n+1} = q^T (d\varphi\: v^n + \chi).
\end{align}
As shown in Lemma \ref{lem:existenceAndUniquenessOfv}, the sequence $v^n$ converges uniformly. Hence, also using the fact that $\varphi \in \mathcal{C}^2(M)$, for every $\epsilon > 0,$ there exists an $N \in \mathbb{N}$ such that for all $m,n \ge N,$ 
\begin{align}
    \|v^m_x - v^n_x \| < \dfrac{\epsilon}{\sup_{x\in M} \|d\varphi_x\|}, \;\;\forall\; x\in M.
\end{align}
Hence, for all $m, n \geq N,$ and for all $x \in M,$
\begin{align}
\notag
    |a^{m+1}_x - a^{n+1}_x| &\leq \|q_x\| \|d\varphi_x\| \|v^m_x - v^n_x\| \\
    &\leq \left(\sup_{x\in M} \|d\varphi_x\|\right) \|v^m_x - v^n_x\| \leq \epsilon.
\end{align}
Thus, $\left\{a^n\right\}$ converges uniformly.
\end{proof}
We define the limit of the sequence $\left\{a^n\right\}$ by taking the limit as $n\to\infty$ of Eq. \ref{eqn:orthogonalityToGivean}, 
\begin{align}
    a := \lim_{n\to\infty} a^n = q^T(d\varphi\: v + \chi),
\end{align}
where $v := \lim_{n\to\infty}v^n$ is as defined in Lemma \ref{lem:existenceAndUniquenessOfv}. In other words, we have the following relationship between $a$ and $v$,
\begin{align}
\label{eqn:avrelation}
    \chi - a q = v - d\varphi \: v.
\end{align}
As we have shown in Lemma \ref{lem:existenceAndUniquenessOfv}, the limit $v$ is a unique vector field, independent of the initial condition for the iteration (Eq. \ref{eqn:stableTangentEquation} or equivalently, Eq. \ref{eqn:alternativeStableTangentEquation}) provided that $v^0$ is bounded. 

\subsection{Route to showing existence and differentiability of the S3 decomposition}
\label{sec:route}
So far, we have shown that a bounded vector field $v$ exists that is orthogonal to the unstable manifold, and the asymptotic solution of a regularized tangent equation (Eq. \ref{eqn:stableTangentEquation}). We showed the existence of the stable contribution by proving that the scalar field $a$, which represents the component in the unstable direction in decomposing $\chi$, is related to $v$ (Eq. \ref{eqn:avrelation}). Further, we established that $a$ is achieved in the same iterative procedure used for $v.$ 

In order to complete the proof of Theorem \ref{thm:thmS3}, we must show that the S3 decomposition of $\chi$ is differentiable on the unstable manifold. That is, we must show that the scalar field $a$ is differentiable on the unstable manifold. Lemma \ref{lem:existenceAndUniquenessOfw} is a result toward this purpose. Using Eq. \ref{eqn:avrelation}, we can see that if $v,$ and hence $d\varphi \: v$ are differentiable with respect to $\xi$, and since $\chi$ is differentiable, by assumption, in all directions, we can then conclude that $aq$ is differentiable in the unstable direction by Eq. \ref{eqn:avrelation}.

Further, by Lemma \ref{lem:existenceAndUniquenessOfw}, the derivative $w := \partial_\xi q$ exists. Thus, the scalar field $a$ is differentiable on the unstable manifold, if $v$ is. 
Therefore, it remains to show the existence of $\partial_\xi  v,$ in order to complete the proof of Theorem \ref{thm:thmS3}-1 and 2. 

\subsection{Differentiability of $v$ in the unstable direction}
\begin{proposition}
\label{prop:dvdxi}
The regularized perturbation field $v$ is differentiable on the unstable manifold, i.e., $\partial_\xi v$ exists.
\end{proposition}
We prove this proposition using the following lemma.
\begin{lemma}
\label{lem:existenceAndUniquenessOfy}
Let $\left\{ \zeta^n \right\}$ be a uniformly exponentially converging sequence of vector fields. That is, $\left\{\zeta^n\right\}$ is uniformly converging such that there exists an $A > 0$ for which
$\| \zeta^{n+1}_x - \zeta^n_x \| \leq A \: \lambda^n,$ for all $x \in M,$
and $n \in \mathbb{Z}^+.$ Given a bounded vector field $y^0$, the sequence $\left\{y^n\right\}$ that satisfies the following recurrence,
\begin{align}
    \label{eqn:recurrenceFory}
    y^{n+1} = \dfrac{1}{\alpha}\unOrthProjection\: d\varphi\: y^n + \zeta^{n+1},\;\;\forall\;\; n\in \mathbb{Z}^+,
\end{align}
converges uniformly to a unique vector field, $y := \lim_{n\to\infty} y^n$.
\end{lemma}
\begin{proof}
Using the recurrence relation, for all $n \in \mathbb{N},$
\begin{align}
    y^{n+1} - y^n = \dfrac{1}{\alpha}\unOrthProjection\: d\varphi\: (y^n - y^{n-1})  + (\zeta^{n+1} - \zeta^{n}).
\end{align}
Applying this equation iteratively,
\begin{align}
    y^{n+1} - y^n &= \dfrac{1}{\prod_{k=0}^{n-1} \alpha\circ\varphi^{-k}} (\unOrthProjection \: d\varphi)^n (y^1 - y^0) \\
    &+ 
    \sum_{k=0}^{n-1} \dfrac{1}{\prod_{j=0}^{k-1} \alpha\circ\varphi^{-j}}(\unOrthProjection \: d\varphi)^k (\zeta^{n-k+1} - \zeta^{n-k}).
\end{align}
Using i) the relation $(\unOrthProjection \: d\varphi)^n = \unOrthProjection \: d\varphi^n \: \stableProjection$ -- which is derived in Lemma \ref{lem:existenceAndUniquenessOfv} using Remarks \ref{rmk:stableProjectionProperty} and \ref{rmk:pvispsv} --, ii) $\|\unOrthProjection_x\| = 1$, and iii) $\prod_{k=0}^{n-1} \alpha_{\varphi^{-k} x} \geq (1/C) \lambda^{-n}$ at any $x \in M,$
\begin{align}
\notag 
    \|y^{n+1}_x - y^n_x\| &\leq  C^2 \:\lambda^{2n}\: S\: \| y^1_{\varphi^{-n} x} - y^0_{\varphi^{-n}x}\| \\
    \label{eqn:ynminuxynminus1}
    &+ 
    C^2 \: S\: \sum_{k=0}^{n-1} \lambda^{2k} \|\zeta^{n-k+1}_{\varphi^{-k}x} - \zeta^{n-k}_{\varphi^{-k} x}\|.
\end{align}
Now since $\left\{\zeta^n\right\}$ is uniformly exponentially converging,
\begin{align}
\label{eqn:exponentialConvergenceofzeta}
    \|\zeta^{n-k+1}_{\varphi^{-k}x} - \zeta^{n-k}_{\varphi^{-k} x}\|
    \leq A \: \lambda^{n-k}.   
\end{align}
Further, since $y^0,$ $\zeta^1$ are bounded vector fields, and $\unOrthProjection \: d\varphi$ is bounded, $y^1$ is a bounded vector field. 
Hence, there exists some constant $A_1$ such that for any $x \in M,$
\begin{align}
\label{eqn:boundednessofy1}
    \|y^1_{\varphi^{-n} x} - y^0_{\varphi^{-n}x}\| < A_1.
\end{align}
Using both the above relationships ( \cref{eqn:exponentialConvergenceofzeta} and \cref{eqn:boundednessofy1}) in \cref{eqn:ynminuxynminus1},
\begin{align}
\notag 
    \|y^{n+1}_x - y^n_x\| &\leq  C^2\: \lambda^{2n}\: S\: A_1 + 
     A\: C^2\: \: S\: \lambda^{n}/(1-\lambda) \leq  A_2 \lambda^n.
\end{align}
Hence, for $m \geq n > 0,$ and all $x\in M,$
\begin{align}
    \|y^{m}_x - y^n_x\| &\leq  \sum_{k=n}^{m-1} A_2 \lambda^k < A_2  \dfrac{\lambda^n}{1-\lambda}.
\end{align}
Thus, $\left\{ y^n\right\}$ is uniformly Cauchy and converges uniformly. 

To see that the limit $y := \lim_{n\to\infty} y^n$ is unique, let $y^n$ and $\tilde{y}^n$ be two different bounded sequences that satisfy Eq. \ref{eqn:recurrenceFory}, and converge to $y$ and $\tilde{y}$ respectively. Then,
\begin{align}
    y^{n} - \tilde{y}^{n} = \dfrac{1}{\alpha}\unOrthProjection \: d\varphi\: (y^{n-1} - \tilde{y}^{n-1}),
\end{align}
which can be applied recursively to yield,
\begin{align}
    y^{n} - \tilde{y}^{n} = \dfrac{1}{\prod_{k=0}^{n-1} \alpha\circ\varphi^{-k}} (\unOrthProjection \: d\varphi\:)^n (y^0 - \tilde{y}^0) = \dfrac{1}{\prod_{k=0}^{n-1} \alpha\circ\varphi^{-k}} \unOrthProjection\: d\varphi^n \: \stableProjection (y^0 - \tilde{y}^0).
\end{align}
Since both $y^0$ and $\tilde{y}^0$ are bounded, the operator norm of $\stableProjection_x$ is uniformly bounded above, and from the definition of uniform hyperbolicity, there exists some $c > 0$ such that for all $n,$ and all $x \in M$, we have 
\begin{align}
     \|\tilde{y}^{n}_x - y^{n}_x\| \leq c\: \lambda^{2n} .
\end{align}
Thus, $\lim_{n\to\infty} (\tilde{y}^n_x - y^n_x) = 0 \in \mathbb{R}^m$ for every $x \in M.$ Hence, the limit $y$ is unique.
\qed
\end{proof}

Now we prove Proposition \ref{prop:dvdxi}. We set 
\begin{align}
    \label{eqn:zeta}
    \zeta^{n+1} = \unOrthProjection \big(\dfrac{d^2\varphi(q, v^n)}{\alpha} + d\chi\: q \big) + (\partial_\xi \unOrthProjection)(d\varphi\: v^n + \chi),
\end{align}
where the sequence $\left\{v^n\right\}$ is as defined in Lemma \ref{lem:existenceAndUniquenessOfv}; $d^2\varphi_x(\cdot, \cdot):T_xM\times T_xM \to T_{\varphi x}M$ is a bilinear form representing the second-order derivative of $\varphi$ at $x \in M$. In order to apply Lemma \ref{lem:existenceAndUniquenessOfy}, we must show that $\left\{\zeta^n\right\}$ is a bounded, uniformly exponentially converging sequence. First we note that the projection operator $\unOrthProjection$ is differentiable in the unstable direction. Using the matrix representation of $\unOrthProjection,$ and Lemma \ref{lem:existenceAndUniquenessOfw}, its derivative is given by 
\begin{align}
    \partial_{\xi}\unOrthProjection = - \partial_\xi (qq^T) = 
    - (w q^T + q w^T),
\end{align}
from which we see that the operator norm of $\partial_\xi \unOrthProjection$ is bounded on $M$. Because $d\varphi \in C^2(M)$ by assumption, $\|d\varphi_x\|$ and $\|d^2\varphi_x\|$ are also bounded on $M$; $\chi$ and $d\chi$ are bounded by assumption. Now, the sequence $\left\{v^n\right\}$ is uniformly bounded on $M$ and $\mathbb{Z}^+$ ($\because \: \left\{v^n\right\}$ is bounded and uniformly convergent, by Lemma \ref{lem:existenceAndUniquenessOfv}). Thus,  $\left\{\zeta^n\right\}$ defined in Eq. \ref{eqn:zeta} is a uniformly bounded sequence.

Now to show that $\left\{\zeta^n\right\}$ is uniformly 
exponentially converging, we use its definition, 
Eq. \ref{eqn:zeta}, to get
\begin{align}
    \zeta^{n+1} - \zeta^{n} =  \partial_\xi \unOrthProjection\: d\varphi\: (v^{n+1} - v^{n}) + \dfrac{\unOrthProjection}{\alpha}\: d^2\varphi\left( q, (v^{n+1} - v^{n})\right).
\end{align}
Since $\partial_\xi \unOrthProjection$, $d\varphi$ and $d^2\varphi$ are bounded and $\|\unOrthProjection\| = 1$, and using Eq. \ref{eqn:vnminusvnminus1}, there exists some $c > 0$ such that, for all $x \in M,$
\begin{align}
    \|\zeta^{n+1}_x - \zeta^{n}_x\| \leq  \left(\|(\partial_\xi \unOrthProjection)_x\|\: \norm{d\varphi_{\varphi^{-1}x}} + \frac{\|d^2\varphi_x(q_x,\cdot)\|}{\alpha}\right)\: A\: \lambda^n := c\: \lambda^n.
\end{align}
Thus, using the same argument following Eq. \ref{eqn:vnminusvnminus1} in Lemma \ref{lem:existenceAndUniquenessOfv}, we can show that $\left\{\zeta^n\right\}$ converges uniformly. Further, the limit $\zeta := \lim_{n\to \infty} \zeta^n$ is unique, and defined by taking the limit of Eq. \ref{eqn:zeta},
\begin{align}
\label{eqn:zetainf}
    \zeta =  \unOrthProjection \big(\dfrac{d^2\varphi(q, v)}{\alpha} + d\chi\cdot q \big) + (\partial_\xi \unOrthProjection)(d\varphi\: v + \chi),
\end{align}
Therefore, $\left\{\zeta^n\right\}$, as defined in Eq. \ref{eqn:zeta}, is a uniformly bounded, uniformly exponentially converging sequence of vector fields. Now applying Lemma \ref{lem:existenceAndUniquenessOfy}, there exists a unique, bounded vector field $y$ that satisfies
\begin{align}
    \label{eqn:stable_tangent_derivative}
    y = \dfrac{\unOrthProjection\: d\varphi\: y}{\alpha} + \zeta.
\end{align}
Differentiating with respect to $\xi$ the constraint, Eq. \ref{eqn:stable_tangent}, which is satisfied by the regularized perturbation field $v$, we see that if the derivative $\partial_\xi v$ exists, it must satisfy Eq. \ref{eqn:stable_tangent_derivative} (replacing $y$), with $\zeta$ defined in Eq. \ref{eqn:zetainf}. Since a vector field $y$ that satisfies Eq. \ref{eqn:stable_tangent_derivative} exists and is unique (Lemma \ref{lem:existenceAndUniquenessOfy}), the regularized perturbation field $v$ is indeed differentiable.  This completes the proof of \cref{prop:dvdxi}. Further, the computation of its derivative, using the sequence of vector fields $\left\{y^n\right\}$ used in the proof of Lemma \ref{lem:existenceAndUniquenessOfy}, converges exponentially along almost every trajectory. Note that the evaluation of $y$ along a trajectory, with $\zeta$ defined in Eq. \ref{eqn:zetainf}, is precisely that derived in section \ref{sec:formulaforb}.  

Recall the argument (section \ref{sec:route}) that the differentiability (with respect to $\xi$) of $v$ implies that of $a.$ This completes the proof of Theorem \ref{thm:thmS3}-2., and hence the first two parts of Theorem \ref{thm:thmS3}. To complete the proof of Theorem \ref{thm:thmS3}-3, we provide a constructive proof of the unstable derivative $b = \partial_\xi a.$ In a similar vein to Lemmas \ref{lem:existenceAndUniquenessOfv}, \ref{lem:existenceAndUniquenessOfw} and \ref{lem:existenceAndUniquenessOfy}, this proof shows that a trajectory-based computation adopted in S3 converges exponentially, in this case, to the true values of $b$. 
\subsection{Convergence of the derivative of $\left\{a^n\right\}$}
We now show that the iterative algorithm for the scalar derivative $b^n := \partial_\xi a^n$ converges uniformly. Expanding Eq. \ref{eqn:zeta} to make the appearance of $b^n := \partial_{\xi} a^n$ explicit, by using the definition of $a^n$ (Lemma \ref{lem:convergenceOfa}),
\begin{align}
\notag
    y^{n+1} 
    &= \dfrac{d\varphi\: y^n}{\alpha} + \dfrac{d^2\varphi(v^n, q)}{\alpha}\: + d\chi\: q - a^{n+1}w - 
    b^{n+1}q.
\end{align}
Taking inner product with $q$ and using $w_x \cdot q_x = 0$ at all $x\in M$ (Lemma \ref{lem:existenceAndUniquenessOfw}),
\begin{align}
\label{eqn:bn}
    b^{n+1} &= \left(\dfrac{d\varphi\: y^n}{\alpha} - y^{n+1} + \dfrac{d^2\varphi(v^n,q)}{\alpha} + d\chi \: q\right)\cdot q.
\end{align}
\begin{lemma}
\label{lem:convergenceOfb}
The sequence $\left\{ b^n\right\}$ defined in Eq. \ref{eqn:bn} converges uniformly
where i) the sequence $\left\{y^n\right\}$ satisfies Eq. \ref{eqn:recurrenceFory}, and, ii) $\left\{v^n\right\}$ satisfies Eq. \ref{eqn:regularizedTangentEquation}.
\end{lemma}
\begin{proof}
We prove that the sequence of vector fields $\left\{\tau^n\right\}$ given by
\begin{align}
    \tau^{n+1} = \dfrac{d\varphi\: y^n}{\alpha} - y^{n+1} + \dfrac{d^2\varphi(v^n,q)}{\alpha} + d\chi \: q.
\end{align}
converges uniformly. From the above definition, and using the linearity of $d^2\varphi(\cdot,q)$, at any $x \in M,$ and all $n> m\ge 0,$
\begin{align}
\label{eqn:taun}
    \|\tau^{n+1}_x - \tau^{m+1}_x\| &\leq \dfrac{\|d\varphi_x\|}{\alpha_x}\: \|y^n_x - y^m_x\| + \|y^{n+1}_x-y^{m+1}_x\| + \dfrac{\|d^2\varphi_x(v^n_x - v^{m}_x, q_x)\|}{\alpha_x}.
\end{align}
Recall that $\varphi \in C^3(M)$ and hence $\|d\varphi_x\|$ and $\|d^2\varphi_x(\cdot, \cdot)\|$ are both uniformly bounded; from uniform hyperbolicity, $1/\alpha_x \leq C\lambda,$ at all $x \in M$. Let $\epsilon > 0$ be given. Using Lemma \ref{lem:existenceAndUniquenessOfy}, there exists an $N_y \in \mathbb{N}$ such that for all $x \in M$ and $m,n \geq N_y,$ $\|y^n_x - y^m_x\| < \epsilon/(3\: C \: \lambda\: \sup_{x\in M} \|d\varphi_x\|).$ Similarly, using Lemma \ref{lem:existenceAndUniquenessOfv}, 
there exists an $N_y \in \mathbb{N}$ such that for all $x \in M$ and $m,n \geq N_v,$ $\|v^n_x - v^m_x\| < \epsilon/(3\: C\: \lambda\: \sup_{x\in M} \|d^2\varphi_x(q,\cdot)\|).$ Choosing $N = \max\{N_v, N_y\}$, for all $m,n \geq N,$ and all $x \in M,$ $
    \|\tau^{n}_x - \tau^{m}_x\| \leq \epsilon.$
Thus, $\left\{\tau^n\right\}$ converges uniformly. Since at all $x \in M,$ $b^n_x
= \tau^n_x \cdot q_x$, the sequence $\left\{b^n\right\}$ also converges uniformly. Moreover, note from \cref{eqn:taun} that the speed of convergence is exponential, since both $\left\{v^n\right\}$ and $\left\{y^n\right\}$ are exponentially uniformly converging as shown in Lemmas \ref{lem:existenceAndUniquenessOfv} and \ref{lem:existenceAndUniquenessOfy} respectively. The limit is unique since,
\begin{align}
\label{eqn:b}
    b &:= \lim_{n\to\infty} b^n = (\lim_{n\to\infty}\tau^n)\cdot q
    = \left(d\chi \:q - y + \dfrac{d\varphi\: y + d^2\varphi(v,q)}{\alpha}\right) \cdot q.
\end{align}

\end{proof}
\section{Proof of Theorem \ref{thm:thmS3Computation}}
\label{sec:proofOfTheorem2}
In this section, we complete the proof of Theorem \ref{thm:thmS3Computation}, which 
establishes the convergence of the S3 algorithm (section \ref{sec:algorithm}). In particular, we prove that the ergodic averages we compute in the S3 algorithm (section \ref{sec:algorithm}) for the stable and unstable contributions converge to their true values. 

\subsection{Convergence of the stable contribution}
\label{sec:convergenceOfTheStableContribution}
Having shown the S3 decomposition is differentiable on the unstable manifold (section \ref{sec:proofOfTheorem1}), we now derive an alternative expression for the stable contribution. This expression leads to a direct computation of the stable contribution using the regularized tangent equation (section \ref{sec:stableContribution}). We then show that this computation converges to the true stable contribution. 
\begin{proposition}
\label{lem:convergenceOfStableContribution}
The error in the numerically computed stable contribution converges as ${\cal O}(\sqrt{\log\log N
}/\sqrt{N}),$ where $N$ is the trajectory length used in the computation.
\end{proposition}
 
\begin{proof}
Taking $\lim_{m\to\infty}$ of Eq. \ref{eqn:norm_diff_vmn},
$\|v^n_x - v_x\| < A \dfrac{\lambda^n}{1-\lambda},\;\; \forall \: x \in M\:, \:n \in \mathbb{Z}^+$. Hence, $\left\{v^n\right\}$ is a uniformly bounded sequence. Since $dJ$ is bounded by assumption, by Lebesgue dominated convergence on the sequence of functions $\left\{dJ\cdot v^n\right\}$, $\lim_{n\to \infty} \langle dJ\cdot v^n, \mu \rangle = \langle dJ\cdot v, \mu\rangle.$ From Lemma \ref{lem:existenceAndUniquenessOfv}, an explicit expression for $v^n$ starting from $v^0 = 0 \in \mathbb{R}^m$ is 
$v^n = \sum_{k=0}^{n-1} d\varphi^k (\chi - a^{n-k}q).$ Thus, $\langle J, \partial_s \mu_s\rangle^{\rm s} = \lim_{n\to \infty} \langle dJ\cdot v^n, \mu \rangle = \langle dJ\cdot v, \mu\rangle.$
Choosing a constant $c \geq \sup_{x\in M} \|(dJ)_x\| A/(1-\lambda),$ 
\begin{align}
    |(dJ)_x\cdot v^n_x - (dJ)_x\cdot v_x| < c\: \lambda^n \:\: \forall \: x \in M\:, \:n \in \mathbb{Z}^+.
\end{align}

Let $x$ be an arbitrary point on $M$ chosen $\mu$-a.e. As usual, we use the shorthand $f_n$ to represent $f_{x_n}$, the vector field $f$ evaluated at the point $x_n$. The error in the $N$-time ergodic average is as follows,
\begin{align}
    \left|\dfrac{1}{N}\sum_{n=0}^{N-1} (dJ)_n \cdot (v^n_n - v_n)\right| \leq  \dfrac{c}{N}\: \sum_{n=0}^{N-1}\lambda^n := \frac{c_1}{N}.
\end{align}
Thus, the error in the stable contribution computation can be bounded as follows.
\begin{align}
\notag
 \left|\dfrac{1}{N}\sum_{n=0}^{N-1} (dJ)_n \cdot v^n_n - \langle dJ\cdot v, \mu\rangle\right| &\leq     \left|\dfrac{1}{N}\sum_{n=0}^{N-1} (dJ)_n \cdot (v^n_n - v_n)\right| 
 + \left|\dfrac{1}{N}\sum_{n=0}^{N-1} (dJ)_n \cdot v_n - \langle dJ\cdot v, \mu \rangle\right| \\
 &\leq \dfrac{c_1}{N} + \dfrac{c_2 \sqrt{\log\log N}}{\sqrt{N}}.
\end{align}
The bound $c_2\sqrt{\log\log N}/\sqrt{N}$ appears since 
$(dJ)\cdot v$ is H\"older continuous, and the law of the iterated logarithm \cite{denker_philipp_1984}\cite{nicolCLT} applies to the convergence of its ergodic sum. Thus, we conclude that the error in the stable contribution computation converges, with $N$ -- the trajectory length or the number of sample points according to $\mu$ -- similar to a typical Monte Carlo integration up to an iterated logarithmic factor,
\begin{align}
    \left|\dfrac{1}{N}\sum_{n=0}^{N-1} (dJ)_n \cdot v^n_n - \langle dJ\cdot v, \mu_s\rangle\right| \sim {\cal O}(\sqrt{\log\log N}/\sqrt{N}).
\end{align}
\qed
\end{proof}
\subsection{Convergence of the unstable contribution}
We establish the convergence of the unstable contribution with the help of two lemmas. The first shows the convergence of the recursive formula for the logarithmic density gradient, $g$. The second lemma, Lemma \ref{lem:convergenceOfb}, proves the convergence of the iterative procedure to differentiate the regularized perturbation field, $v,$ on the unstable manifold. Together, these two results can be used to show that the unstable contribution, with the computation described in the S3 algorithm (section \ref{sec:algorithm}), converges.
\subsubsection{Convergence of logarithmic density gradient}
\begin{lemma}
\label{lem:convergenceOfg}
Given a bounded scalar field $r,$ there exists a unique, bounded scalar field $g$ that satisfies
\begin{align}
\label{eqn:conditionForg}
    g\circ\varphi = \dfrac{g}{\alpha\circ\varphi}  + 
    r\circ\varphi,
\end{align}

\end{lemma}
Note that when $r = -\gamma/\alpha$, the logarithmic density gradient function satisfies Eq. \ref{eqn:conditionForg}. 
\begin{proof}
First we show existence. Let $h^0:M\to \mathbb{R}$ be an arbitrary bounded scalar function.
Consider the sequence of scalar functions $\left\{h^n\right\}$ which follow the recurrence relation:
\begin{align}
\label{eqn:recurrenceh}
    h^{n+1}\circ\varphi = \dfrac{h^n}{\alpha\circ\varphi} + r\circ\varphi.
\end{align}
We show that this sequence converges uniformly. From the above recurrence relation, at every $x \in M$,
\begin{align}
   h^{n+1}_x - h^n_x = \dfrac{h^n_{\varphi^{-1} x}}{\alpha_{x}}  - \dfrac{h^{n-1}_{\varphi^{-1} x}}{\alpha_{x}}.
\end{align}
Applying this relation recursively, for all $n\geq 1,$
\begin{align}
    \abs{ h^{n+1}_x - h^n_x} \leq \dfrac{1}{\prod_{k=0}^{n-1}\alpha_{\varphi^{-k}x}}\abs{h^1_{\varphi^{-n} x}  - h^{0}_{\varphi^{-n} x}}.
\end{align}
Under the assumption of uniform hyperbolicity, there exist constants $C > 0$ and 
$\lambda \in (0,1)$ such that $\prod_{k=0}^{n-1}\alpha_{\varphi^{-k}x} \geq (1/C)\lambda^{-n},$ for all $x \in M.$ Thus, 
\begin{align}
    \abs{ h^{n+1}_x - h^n_x} \leq C \lambda^n \abs{h^1_{\varphi^{-n} x}  - h^{0}_{\varphi^{-n} x}}.
\end{align}
Since $h^0$ is a bounded function, and $h^1$ is also a bounded function because $h^0$ and $r$ are bounded, there exists $A >0$ such that $|h^1_x - h^0_x| \leq A,$ for all $x \in M.$ Thus, for all $i > j \ge 0,$
\begin{align}
    \abs{ h^i_x - h^j_x} \leq \sum_{n=j}^{i-1}\abs{ h^{n+1}_x - h^n_x} \leq A\: C\: \dfrac{\lambda^j}{1-\lambda}. 
\end{align}
Hence $\left\{h^n\right\}$ is a uniformly Cauchy sequence and therefore converges uniformly. Let $g := \lim_{n\to\infty} h^n.$ We show that this limit is unique. Let $\left\{h^n\right\}$ and $\{\tilde{h}^n\}$ be two different sequences satisfying Eq. \ref{eqn:recurrenceh}. Additionally, assume that both $h^0$ and $\tilde{h}^0$ are bounded functions. Then,
\begin{align}
    |h^n_x - \tilde{h}^n_x| 
    \leq \dfrac{|h^{n-1}_{\varphi^{-1}x} - \tilde{h}^{n-1}_{\varphi^{-1}x}|}{\alpha_{\varphi^{-1}x}},
\end{align}
which by iteration, gives,
\begin{align}
     |h^n_x - \tilde{h}^n_x| 
    \leq \dfrac{|h^0_{\varphi^{-n}x} - \tilde{h}^0_{\varphi^{-n}x}|}{\prod_{k=1}^n\alpha_{\varphi^{-k}x}}\leq C\:\lambda^n \: |h^0_{\varphi^{-n}x} - \tilde{h}^0_{\varphi^{-n}x}|.
\end{align}
Thus $|h^n_x - \tilde{h}^n_x|\to 0$ as $n\to \infty$ since $|h^0_{\varphi^{-n}x} - \tilde{h}^0_{\varphi^{-n}x}| < {\rm const}$.  Since this holds for all $x \in M,$
$\lim_{n\to\infty} h^n = \lim_{n\to\infty} \tilde{h}^n = g.$
\qed
\end{proof}
We have shown that the limit $g$ is independent of the initial condition $h^0.$ Hence, without loss of generality, we may assume that $h^0_x = 0$ at all $x \in M,$ as we do in section \ref{sec:formulaforg}. We remark that, due to the algebraic simplification introduced in section \ref{sec:unstableContributionComputation}, the explicit iteration of Eq. \ref{eqn:recurrenceh} is subsumed under the combination of Eq. \ref{eqn:pcomputation} and \ref{eqn:secondOrderStableTangentEquationFinal}. 

\subsubsection{Convergence of the unstable contribution computation}
\label{sec:unstableContributionProof}
In the S3 algorithm, the computation of the unstable contribution is carried out as an ergodic average. We use the results we have set up so far to show that this ergodic average converges to the true unstable contribution.
From our derivation (section \ref{sec:integrationByParts}) and the strong decay of correlations assumption (section \ref{sec:unstableContributionComputation}), we obtain the following regularized expression for the unstable contribution,
\begin{align}
\label{eqn:series}
    \langle J, \partial_s \mu_s \rangle^{\rm u} =  -\sum_{k=0}^{\infty} \langle J\circ\varphi^k ( a \: g + b), \mu\rangle. 
\end{align}
In order to show that our computation of the unstable contribution converges, we first note that, due to exponential decay of correlations between H\"older continuous functions (section \ref{sec:integrationByParts}), the above series converges. Further, we show in Lemmas \ref{lem:convergenceOfa}, \ref{lem:convergenceOfb} and \ref{lem:convergenceOfg} that the convergence to the true values of $a$, $b$ and $g$ respectively is exponential in each case, along any $\mu$-typical orbit. Thus, iterating the equations for these quantities for a sufficient run-up time, we may assume that they are, up to machine precision, equal to their true values.  The convergence rate of the unstable contribution is thus determined by that of the following $N$-sample averages, at the first few $k \in \mathbb{Z}^+$,
\begin{align}
    \left|\sum_{k=0}^{K-1}\left(\langle J\circ\varphi^k ( a \: g + b), \mu\rangle   -\dfrac{1}{N}\sum_{n=0}^{N-1} J_{n+k} \: (a_n\: g_n + b_n)\right)\right| \leq \dfrac{\sum_{k < K}c_k\;\sqrt{\log\log N}}{\sqrt{N}}.
\end{align}
The almost sure error bound of ${\cal O}(\sqrt{\log\log N}/\sqrt{N})$ again appears because of the law of the iterated logarithm, which holds for H\"older continuous functions \cite{denker_philipp_1984}\cite{nicolCLT}. 
Given $\epsilon > 0,$ there exists a $K$ such that $\left|\sum_{k > K} \langle J\circ\varphi^k (a\: g + b), \mu\rangle\right| < \epsilon/2,$ due to the convergence of \cref{eqn:series}. Then, choosing $N$ large enough such that $\sum_{k < K} c_k\sqrt{\log\log N}/\sqrt{N} < \epsilon/2,$ we approximate the unstable contribution arbitrarily well, $\left|\langle J, \partial_s \mu_s\rangle^{\rm u} - \sum_{k=0}^{K-1} (1/N)\sum_{n=0}^{N-1} J_{n+k} \: (a_n\: g_n + b_n)\right| < \epsilon.$ This numerical approximation converges to the unstable contribution as $N \to \infty$ followed by as $K \to \infty.$
We have already shown that the stable contribution computation has an error convergence rate of ${\cal O}(\sqrt{\log\log N}/\sqrt{N})$ (Lemma \ref{lem:convergenceOfStableContribution}). This proves that the sum of stable and unstable contribution converges (discounting the bias due to a finite $K$) as ${\cal O}(\sqrt{\log\log N}/\sqrt{N})$, with the S3 algorithm implemented on an $N$-length trajectory.
\section{Discussion and conclusion}
\label{sec:conclusion}
The main contribution of this paper is a new algorithm, called space-split sensitivity (S3), to compute linear response, or the sensitivities of statistics to parameters, in chaotic dynamical systems. The paper presents the derivation of S3 with Ruelle's linear response formula as the starting point. While Ruelle's formula leads to an ill-conditioned direct computation, we derive a decomposition -- the S3 decomposition -- of the formula that yields a well-conditioned ergodic-averaging computation. The S3 decomposition is achieved by orthogonalizing the conventional tangent solution with respect to the unstable subspaces along a long trajectory. This procedure yields a regularized tangent vector field $v$ and a scalar field $a.$ 

The sensitivity to the unstable vector field so achieved -- $a\: q$ -- is called the unstable contribution. The stable contribution is the remaining linear response, and can be calculated just as tangent sensitivities are in non-chaotic systems, but using the regularized tangent solutions in place of the conventional tangent solutions. The unstable contribution is integrated by parts on the unstable manifold to yield an expression in which, unlike the original form of Ruelle's formula, the integrand is bounded at all times. Then, exploiting the decay of correlations, this regularized formula can be computed fast as a sum of Monte Carlo integrals. However, the problem is, we obtain two unknown functions -- the derivative of $a$ on the unstable manifold, and the density gradient. The density gradient is a fundamental object whose regularity we link, in a different study \cite{adam2}, to the validity of linear response \cite{wormell1}\cite{wormell2}. 

The density gradient is the unstable derivative of the logarithmic conditional density of the SRB measure on the unstable manifold. We take the unstable derivative of Pesin's formula to obtain a recursive evaluation of the density gradient along a typical trajectory. We show that starting this recursion with an arbitrary scalar field leads to an exponential convergence to the true density gradient. A recursive procedure is also derived for the unstable derivative of $a$ by combining the unstable derivatives of $v$ and $q.$ The evaluation of the latter two unstable derivatives requires solving two second-order tangent equations, which are the most computationally intensive steps of the S3 algorithm. 

Overall, the S3 algorithm computes Ruelle's formula as a well-conditioned ergodic average. The efficiency of the S3 algorithm stems from its decomposition of Ruelle's formula and its treatment of the unstable contribution. In Ruelle's formula (Eq. \ref{eqn:ruellesFormula}), the norm of the integrand increases exponentially with $n,$ which leads to the ill-conditioning of the estimation of each term as an ergodic average. However, in the unstable contribution after applying the S3 decomposition (Eq. \ref{eqn:unstableContributionRegularizedExpression}), the norm of the integrand does not grow with $n$ and hence, the estimation of the S3 formula as an ergodic average is efficient. 

An important direction of future work is to make S3 applicable to systems with higher-dimensional unstable manifolds. Instead of being one-dimensional, as we have assumed throughout this paper, suppose $E^u_x$ is an $m_u \ge 1$ dimensional subspace at each $x \in M.$ We anticipate that the S3 algorithm then involves the following changes, which will be pursued in a future work:
\begin{enumerate}
    \item We must compute, e.g. using Ginelli's algorithm \cite{ginelli}, an orthonormal basis for the unstable subspaces along an orbit. This more general procedure reduces to Step 2 of the S3 algorithm (section \ref{sec:algorithm}), in the case of a 1D unstable manifold. Let $q^i_{x_n}, 1\leq i\leq m_u, 1\leq n\leq N$ be the resultant orthogonal tangent vectors that span $E^u_{x_n}, 1\leq n \leq N$.
    \item The regularized tangent equation must now be orthogonalized with respect to the computed unstable subspaces along the trajectory, as opposed to with respect to a single direction. This orthogonalization results in a regularized tangent vector field $v$ and $m_u$ different scalar fields, $a^i, 1\leq i\leq m_u,$ in place of a single scalar field $a.$ The stable contribution may be computed using $v$ just as in the 1D unstable manifold case (Eq. \ref{eqn:stableContributionErgodicAverage}).
    \item The first step in treating the unstable contribution is the integration by parts on the unstable manifold. In the multi-dimensional unstable manifold case, the parameterization $\Phi^x$ is now a map from $[0,1]^{m_u}$ to $\Xi_x,$ such that $\nabla_{\xi} \Phi^{x_n}(0)$ maps the standard Euclidean basis vector $e^i$ in $\mathbb{R}^{m_u}$ to $q^i_{x_n}$.  Using disintegration followed by integration by parts on the $m_u$-dimensional unstable manifold results in a regularized expression for the unstable contribution, analogous to the 1D case (Eq. \ref{eqn:unstableContributionRegularized}),
    \begin{align}
        \langle J, \partial_s\mu_s \rangle^{\rm u} = - \sum_{n=0}^\infty \langle J\circ\varphi^n ( a\cdot g + {\rm div}^{\rm u} a), 
        \mu\rangle. 
    \end{align}
    Here $a:M \to \mathbb{R}^{m_u}$ and $g: M\to \mathbb{R}^{m_u}$ are now vector-valued functions on $M.$ The term ${\rm div}^{\rm u} a$ refers to the unstable divergence of the unstable perturbation field $\sum_i a^i q^i.$ Now, analogous to the computation of $ag + b$ in the 1D unstable manifold, which we tackled in this paper, we must derive a new computation for the derivative of $a\cdot g + {\rm div}^{\rm u} a$. Such a computation must similarly be recursive so that it can be efficiently carried out using information available along an orbit. 
    
    Since we used the unstable derivatives of $\alpha,$ $v$ and $q$ in the computation of $ag + b$ in the 1D unstable manifold case, we expect analogous derivatives are necessary in higher dimensions. In particular, the unstable derivatives -- which are $m_u$-dimensional gradients, i.e., partial derivatives taken with respect to each $\xi^i$ corresponding to $q^i$ -- of $v$ and each $q^i$ are needed. The higher dimensional analog of $\alpha$ are the diagonal elements of the R matrix in the QR-based iterative computation in Ginelli's algorithm; the unstable gradient of the diagonal elements of $R$ are needed. Computing these required derivatives efficiently is the crux of this future direction that will enable the application of S3 to systems with arbitrary dimensional unstable manifolds.
\end{enumerate}

From the outline sketched above, it is amply clear that the S3 algorithm presented here serves as an excellent starting point from which to extend to higher-dimensional unstable manifolds. Moreover, we have proved the existence of the S3 decomposition and the convergence of the algorithm in uniformly hyperbolic systems. The proof of convergence presented in this paper is also readily extensible to the case of higher-dimensional unstable manifolds. Here we have shown that both the stable and unstable contributions to the sensitivity, computed as per the S3 algorithm, have an error convergence that declines as a Monte Carlo computation of an ergodic average. Thus, this work is the first step toward circumventing the poor convergence rate of Ruelle's formula for linear response.

\section*{Acknowledgments}
We are greatly indebted to Malo J\'ez\'equel, the anonymous reviewers, Youssef Marzouk and Angxiu Ni for pointing out errors and substantially improving the manuscript.

\bibliographystyle{siamplain}
\bibliography{refs}
\appendix
\section{Relationship between $v$ and the shadowing direction}
\label{appx:shadowing}
Let $\chi = \chi^u + \chi^s$ be the direct sum decomposition of $\chi$ into its $E^u$ and $E^s$ components. In our notation, $\chi^s = \stableProjection \chi.$ In Ni \cite{angxiu-s3}, the shadowing vector field $v^{\rm sh}$ is given by the following expression, which has also been derived in \cite{qiqi}\cite{mario}.  
\begin{align}
    v^{\rm sh} = \sum_{n=0}^\infty d\varphi^n \chi^s  + \sum_{n=1}^\infty d\varphi^{-n} \chi^u.
\end{align}
Note that each series is well-defined since at every $x \in M$, $\|d\varphi^n_x \chi^s_x\| \leq C \lambda^n \|\chi^s_x\|$ and $\|d\varphi^{-n}_x \chi^u_x\| \leq C \lambda^n \|\chi^u_x\|.$ Hence, we can apply the operator $\unOrthProjection$ to the two sequences $\sum_{n=0}^N d\varphi^n \chi^s$ and $\sum_{n=1}^N d\varphi^{-n} \chi^u$, and the resulting sequences converge uniformly. Recall that since $\unOrthProjection \stableProjection = \unOrthProjection,$ $\unOrthProjection(I - \stableProjection)$ is zero. Thus, applying $\unOrthProjection$ on the above expression makes the second series on the right hand side vanish. Hence, using the remark following the proof of Lemma \ref{lem:existenceAndUniquenessOfv},
\begin{align}
    \unOrthProjection v^{\rm sh} = \sum_{n=0}^\infty \unOrthProjection \: d\varphi^n \: \chi^s = \sum_{n=0}^\infty \unOrthProjection \: d\varphi^n \: \stableProjection \chi = v.
\end{align}

\end{document}

%% file: ex_shared.tex
\usepackage{lipsum}
\usepackage{amsfonts,amsmath,amssymb}
\usepackage{graphicx}
\usepackage{epstopdf}
\ifpdf
  \DeclareGraphicsExtensions{.eps,.pdf,.png,.jpg}
\else
  \DeclareGraphicsExtensions{.eps}
\fi

\usepackage{enumitem}
\setlist[enumerate]{leftmargin=.5in}
\setlist[itemize]{leftmargin=.5in}

\usepackage{graphicx}
\usepackage{verbatim}
\usepackage{extarrows}
\usepackage{color}
\newcommand{\norm}[1]{\lVert #1 \rVert}
\newcommand{\abs}[1]{\lvert #1 \rvert}
\usepackage{graphicx}

\newcommand{\unOrthProjection}{\mathcal{P}}

\newcommand{\stableProjection}{\mathcal{S}}

\newcommand{\regularityClassOfObservables}{2}
\newcommand{\regularityClassOfdynamics}{3}
\newcommand{\lyapunovdim}{p}

\definecolor{azure}{rgb}{0.0, 0.5, 1.0}
\definecolor{burgundy}{rgb}{0.5, 0.0, 0.13}

\newsiamremark{notation}{Notation}
\newsiamremark{remark}{Remark}
\newsiamremark{hypothesis}{Hypothesis}
\crefname{hypothesis}{Hypothesis}{Hypotheses}
\newsiamthm{claim}{Claim}

\headers{Efficient linear response of chaos}{N. Chandramoorthy and Q. Wang}

\title{Efficient computation of linear response of chaotic attractors with one-dimensional unstable manifolds\thanks{Submitted to the editors DATE.
\funding{This work was funded by the Air Force Office of Scientific Research Grant No. FA8650-19-C-2207 and Department of Energy Research Grant No. DE-FOA-0002068-0018.}}}

\author{Nisha Chandramoorthy\thanks{Center for Computational Science and Engineering, Massachusetts Institute of Technology, Cambridge, MA 
  (\email{nishac@mit.edu}).}
\and Qiqi Wang\thanks{Center for Computational Science and Engineering, Massachusetts Institute of Technology, Cambridge, MA
  (\email{qiqi@mit.edu}).}
}

\usepackage{amsopn}
